\documentclass[twoside]{article}

\usepackage[accepted]{aistats2022}
\usepackage[round]{natbib}

\usepackage{amssymb, amsmath, amsthm, latexsym}
\usepackage{url}
\usepackage{algorithm}
\usepackage{algorithmic}
\usepackage{tabularx}
\usepackage{paralist}
\usepackage{mathtools}
\usepackage{dsfont}
\def\obf{\mathds{1}}

\usepackage{makecell}
\usepackage{multirow}
\usepackage{booktabs}

\usepackage{nicefrac}       

\usepackage[flushleft]{threeparttable} 
\usepackage{caption}
\usepackage{subcaption}
\usepackage{multirow}
\usepackage{colortbl}
\definecolor{bgcolor}{rgb}{0.8,1,1}
\definecolor{bgcolor2}{rgb}{0.8,1,0.8}

\definecolor{niceblue}{rgb}{0.0,0.19,0.56}

\usepackage{hyperref}
\hypersetup{colorlinks,linkcolor={blue},citecolor={niceblue},urlcolor={blue}}

\usepackage[dvipsnames]{xcolor}
\usepackage{pifont}
\newcommand{\cmark}{{\color{PineGreen}\ding{51}}}%
\newcommand{\xmark}{{\color{BrickRed}\ding{55}}}%

\newcommand{\R}{\mathbb{R}}

\def\<#1,#2>{\left\langle #1,#2\right\rangle}

\newtheorem{lemma}{Lemma}[section]
\newtheorem{theorem}{Theorem}[section]

\newtheorem{assumption}{Assumption}[section]
\newtheorem{corollary}{Corollary}[section]
\newtheorem{remark}{Remark}[section]
\newtheorem{example}{Example}[section]
\theoremstyle{plain}

\usepackage{xspace}

\newcommand{\algname}[1]{{\sf  #1}\xspace}

\usepackage[colorinlistoftodos,bordercolor=orange,backgroundcolor=orange!20,linecolor=orange,textsize=scriptsize]{todonotes}

\newcommand{\cC}{{\cal C}}
\newcommand{\cD}{{\cal D}}

\newcommand{\cO}{{\cal O}}

\newcommand{\mA}{{\bf A}}
\newcommand{\mB}{{\bf B}}
\newcommand{\mC}{{\bf C}}
\newcommand{\mD}{{\bf D}}

\newcommand{\mQ}{{\bf Q}}

\newcommand{\EE}{\mathbb{E}}

\newcommand{\Exp}{\mathbb{E}}
\newcommand{\Prob}[1]{\mathbb{P} \left[ #1\right]}

\usepackage{hyperref}
\usepackage{cleveref}
\graphicspath{{../plots/}}

\usepackage{makecell}

\usepackage{accents}
\newlength{\dhatheight}

\usepackage{pgfplotstable} 
\usetikzlibrary{automata, positioning, arrows, shapes, fit, calc, intersections}
\usepgfplotslibrary{statistics}
\include{figure}

\begin{document}

%

%

\twocolumn[

\aistatstitle{Stochastic Extragradient: General Analysis and Improved Rates}

\aistatsauthor{ Eduard Gorbunov \And Hugo Berard \And  Gauthier Gidel \And Nicolas Loizou }

\aistatsaddress{ \begin{tabular}{c}
    MIPT, Russia\\ Mila \& UdeM, Canada
\end{tabular} \And  Mila \& UdeM, Canada \And \begin{tabular}{c}
    Mila \& UdeM, Canada\\  Canada CIFAR AI Chair
\end{tabular}\And \begin{tabular}{c}
    Johns Hopkins University \\
    Baltimore, USA
\end{tabular} } ]

\begin{abstract}
  The Stochastic Extragradient (\algname{SEG}) method is one of the most popular algorithms for solving min-max optimization and variational inequalities problems (VIP) appearing in various machine learning tasks. However, several important questions regarding the convergence properties of \algname{SEG} are still open, including the sampling of stochastic gradients, mini-batching, convergence guarantees for the monotone finite-sum variational inequalities with possibly non-monotone terms, and others. To address these questions, in this paper, we develop a novel theoretical framework that allows us to analyze several variants of \algname{SEG} in a unified manner. Besides standard setups, like Same-Sample \algname{SEG} under Lipschitzness and monotonicity or Independent-Samples \algname{SEG} under uniformly bounded variance, our approach allows us to analyze variants of \algname{SEG} that were never explicitly considered in the literature before.  Notably, we analyze \algname{SEG} with \textit{arbitrary sampling} which includes importance sampling and various mini-batching strategies as special cases. Our rates for the new variants of \algname{SEG}  outperform the current state-of-the-art convergence guarantees and rely on less restrictive assumptions.
\end{abstract}

\section{INTRODUCTION}

In the last few years, the machine learning community has been increasingly interested in differentiable game formulations where several parameterized models/players compete to minimize their respective objective functions. Notably, these formulations include generative adversarial networks~\citep{goodfellow2014generative}, proximal gradient TD learning~\citep{liu2016accelerated}, actor-critic~\citep{pfau2016connecting}, hierarchical reinforcement learning~\citep{wayne2014hierarchical,vezhnevets2017feudal}, adversarial example games~\citep{bose2020adversarial}, and minimax estimation of conditional moment~\citep{dikkala2020minimax}.

In that context, the optimization literature has considered a slightly more general setting, namely, variational inequality problems. Given a differentiable game, its corresponding VIP designates the necessary first-order stationary optimality conditions. Under the assumption that the objectives functions of the differentiable game are convex (with respect to their respective players' variables), the solutions of the VIP are also solutions of the original game formulation. In the unconstrained case, given an operator\footnote{In the context of a differentiable game, $F$ corresponds to the concatenation of the gradients of the players' losses, e.g., see the details in \cite{gidel2018variational}.}$F:\R^d \rightarrow \R^d$, the corresponding VIP is defined as follows:
\begin{equation}
    \text{find } x^* \in \R^d \quad \text{such that}\quad F(x^*) = 0. \label{eq:main_problem} \tag{VIP}
\end{equation}

When the operator $F$ is monotone (a generalization of convexity), it is known that the standard gradient method 
does not converge without strong monotonicity~\citep{noor2003new,gidel2018variational} or cocoercivity~\citep{chen1997convergence,loizou2021stochastic}.
Because of their convergence guarantees, even when the operator $F$ is monotone, the extragradient method~\citep{korpelevich1976extragradient} and its variants~\citep{popov1980modification} have been the optimization techniques of choice to solve VIP. These techniques consist of two steps: a) an \emph{extrapolation step} that computes a gradient update from the current iterate, and b) an \emph{update step} that updates the current iterate using the value of the vector field at the \emph{extrapolated point}.

Motivated by recent applications in machine learning, in this work we
are interested in cases where the objective, operator $F$, is naturally expressed as a finite sum, $F(x)=\frac{1}{n}\sum_{i=1}^n F_i(x)$ or more generally as expectation $F(x)=\Exp_{\xi}[F_\xi(x)].$ In that setting, we only assume to have access to a \emph{stochastic} estimate of $F$.

Unfortunately, the additive value of extragradient-based techniques in the stochastic VIP setting is less apparent since the method is challenging to analyze in that setting due to the two stochastic gradient computations necessary for a single update. There are several ways to deal with the stochasticity in the \algname{SEG} update. For example, one can use either independent samples~\citep{Nemirovski-Juditsky-Lan-Shapiro-2009, juditsky2011solving} or the same sample~\citep{gidel2018variational} for the extrapolation and the update steps.

The selection of stepsizes in the update rule of \algname{SEG} (for the extrapolation step and update step) is also a challenging task. In \cite{chavdarova2019reducing} it is shown that some same-stepsize variants of \algname{SEG} diverge in the unconstrained monotone case. At the same time, in \cite{hsieh2019convergence} using a double stepsize rule, the authors provide convergence guarantees under an error-bound condition. 

This discrepancy between the deterministic and the stochastic case has motivated a whole line of work~\citep{gidel2018variational,mishchenko2020revisiting,beznosikov2020distributed,hsieh2019convergence} to understand better the properties of \algname{SEG}. However, several important questions remain open. To bridge this gap, in this work, we develop a  novel theoretical framework that allows us to analyze several variants of \algname{SEG} in a  unified manner.

\subsection{Preliminaries}
\paragraph{Notation.} We use standard notation for optimization literature. We also often use $[n]$ to denote $\{1,\ldots,n\}$ and $\Exp_{\xi}[\cdot]$ for the expectation taken w.r.t.\ the randomness coming from $\xi$ only.
\paragraph{Main assumptions.} In this work, we assume that the operator $F$ is $L$-Lipschitz and $\mu$-quasi strongly monotone.
\vspace{-2mm}
\begin{assumption}\label{as:lipschitzness}
	Operator $F(x)$ is $L$-Lipschitz, i.e., for all $x,y \in \R^d$
	\begin{equation}
		\|F(x) - F(y)\| \le L\|x - y\|. \label{eq:lipschitzness}
	\end{equation}
\end{assumption}

\begin{assumption}\label{as:str_monotonicity}
	Operator $F(x)$ is $\mu$-quasi strongly monotone, i.e., for $\mu \ge 0$ and for all $x \in \R^d$
	\begin{equation}
		\langle F(x), x - x^*\rangle \ge \mu\|x - x^*\|^2. \label{eq:str_monotonicity}
	\end{equation}
	We assume that $x^*$ is unique.
\end{assumption}

Assumption~\ref{as:lipschitzness} is relatively standard and widely used in the literature on VIP. Assumption~\ref{as:str_monotonicity} is a relaxation of $\mu$-strong monotonicity as it includes some non-monotone games as special cases. To the best of our knowledge, the term quasi-strong monotonicity was introduced in \cite{loizou2021stochastic} and has its roots in the quasi-strong convexity condition from the optimization literature \citep{Necoara-Nesterov-Glineur-2018-linear-without-strong-convexity,gower2019sgd}. In the literature of variational inequality problems, quasi strongly monotone problems are also known as strong coherent VIPs~\citep{song2020optimistic} or VIPs satisfying the strong stability condition~\citep{mertikopoulos2019learning}. If $\mu= 0$, then Assumption~\ref{as:str_monotonicity} is also known as variational stability condition~\citep{hsieh2020explore, loizou2021stochastic}.

\paragraph{Variants of \algname{SEG}.} In the literature of variational inequality problems there are two main stochastic extragradient variants. 

The first is Same-sample \algname{SEG}:
\begin{equation}
	x^{k+1} = x^k - \gamma_{2,\xi^k} F_{\xi^k}\left(x^k - \gamma_{1,\xi^k} F_{\xi^k}(x^k)\right), \tag{S-SEG} \label{eq:S_SEG}
\end{equation}
where in each iteration, the same sample $\xi^k$ is used for the exploration (computation of $x^k - \gamma_{1,\xi^k} F_{\xi^k}(x^k)$) and update (computation of $x^{k+1}$) steps. The selection of step-sizes $\gamma_{2,\xi^k}$ and $\gamma_{1,\xi^k}$ that guarantee convergence of the method in different settings varies across previous papers~\citep{mishchenko2020revisiting,beznosikov2020distributed,hsieh2019convergence}. In this work, the proposed stepsizes for \ref{eq:S_SEG} satisfy $0 < \gamma_{2,\xi^k} = \alpha\gamma_{1,\xi^k}$, where $0 <\alpha < 1$, and are allowed to depend on the sample $\xi^k$. This specific stepsize selection is one of the main contributions of this work and we discuss its benefits in more detail in the subsequent sections. 

The second variant is Independent-samples \algname{SEG}
\begin{equation}
	x^{k+1} = x^k - \gamma_2 F_{\xi_2^k}\left(x^k - \gamma_1 F_{\xi_1^k}(x^k)\right),\tag{I-SEG} \label{eq:I_SEG}
\end{equation}
where $\xi_{1}^k, \xi_2^k$ are independent samples. Similarly to \algname{S-SEG}, we assume that $0 < \gamma_{2} = \alpha\gamma_{1}$, with $0 <\alpha < 1$, but unlike \algname{S-SEG}, for \algname{I-SEG} we consider stepsizes independent of samples $\xi_{1}^k, \xi_2^k$.\footnote{This is mainly motivated by the fact that the analysis of \algname{I-SEG} does not rely on the Lipshitzness of particular stochastic realizations $F_\xi$.}

Typically, \algname{S-SEG} is analyzed under Lipschitzness and (strong) monotonicity of individual stochastic realizations $F_\xi$ \citep{mishchenko2020revisiting} that are stronger than Assumptions~\ref{as:lipschitzness}~and~\ref{as:str_monotonicity}. In contrast, \algname{I-SEG} is studied under Assumptions~\ref{as:lipschitzness}~and~\ref{as:str_monotonicity} but with additional assumptions like uniformly bounded variance or its relaxations \citep{beznosikov2020distributed, hsieh2020explore}. See Appendix~\ref{app:on_the_assumptions} for further clarifications.

\subsection{Contributions}
Our main contributions are summarized below.
    
$\diamond$ \textbf{Unified analysis of \algname{SEG}.} We develop a new theoretical framework for the analysis of \algname{SEG}. In particular, we construct a unified assumption (Assumption~\ref{as:unified_assumption_general}) on the stochastic estimator, stepsizes, and the problem itself \eqref{eq:main_problem}, and we prove a general convergence result under this assumption (Theorem~\ref{thm:main_theorem_general_main}). Next, we show that both \algname{S-SEG} and \algname{I-SEG} fit our theoretical framework and can be analyzed in different settings in a unified manner. In previous works, these variants of \algname{SEG} have been only analyzed separately using different proof techniques. Our proposed proof technique differs significantly from those existing in the literature and, therefore, is of independent interest.
    
$\diamond$ \textbf{Sharp rates for the known methods.} Despite the generality of our framework, our convergence guarantees give tight rates for several  well-known special cases. That is, the proposed analysis either recovers best-known (up to numerical factors) rates for some special cases like the deterministic \algname{EG} and the \algname{I-SEG} under uniformly bounded variance (UBV) assumption (Assumption~\ref{as:UBV_and_quadr_growth} with $\delta = 0$), or improves the previous SOTA results for other well known special cases, e.g., for \algname{S-SEG} with uniform sampling and \algname{I-SEG} under the generalized UBV assumption (Assumption~\ref{as:UBV_and_quadr_growth} with $\delta > 0$).
    
$\diamond$ \textbf{New methods with better rates.} Through our framework, we propose a general yet simple theorem describing the convergence of \algname{S-SEG} under the \textit{arbitrary sampling} paradigm \citep{gower2019sgd,loizou2021stochastic}. Using the theoretical analysis of \algname{S-SEG} with arbitrary sampling, we can provide tight convergence guarantees for several well-known methods like the deterministic/full-batch \algname{EG} and \algname{S-SEG} with uniform sampling (\algname{S-SEG-US}) as well as some variants of \algname{S-SEG} that were never explicitly considered in the literature before. For example, we are first to analyze \algname{S-SEG} with mini-batch sampling without replacement ($b$-nice sampling; \algname{S-SEG-NICE}) and show its theoretical superiority to vanilla \algname{S-SEG-US}. Moreover, we propose a new method called \algname{S-SEG-IS} that combines \algname{S-SEG} with \textit{importance sampling} -- the sampling strategy, when the $i$-th operator from the sum is chosen with probability proportional to its Lipschitz constant. We prove the theoretical superiority of \algname{S-SEG-IS} in comparison to \algname{S-SEG-US}.
    
$\diamond$ \textbf{Novel stepsize selection.} One of the key ingredients of our approach is the use of sample-dependent stepsizes. This choice of stepsizes is especially important for the \algname{S-SEG-IS}, as it allows us to obtain better theoretical guarantees compared to the \algname{S-SEG-US}. Moreover, as in \citet{hsieh2020explore}, for the update step we also use smaller stepsizes than for the exploration step: $\gamma_{2,\xi^k} \le \gamma_{1,\xi^k}$ ($\gamma_2 \le \gamma_1$). However, unlike the results by \citet{hsieh2020explore}, our theory allows using $\gamma_{2,\xi^k} = \alpha \gamma_{1,\xi^k}$ with constant parameter $\alpha < 1$ to achieve any predefined accuracy of the solution.
    
$\diamond$ \textbf{Convergence guarantees under weak conditions.} The flexibility of our approach helps us to derive our main theoretical results under weak assumptions. In particular, in the analysis of \algname{S-SEG}, we allow the stochastic realizations $F_\xi$ to be $(\mu_\xi,x^*)$-quasi strongly monotone with \textit{possibly negative} $\mu_\xi$, meaning that $F_\xi$ can be non-monotone (see Assumption~\ref{as:F_xi_str_monotonicity}). In addition, in the analysis of \algname{S-SEG} we do not require any bounded variance assumption.
    To the best of our knowledge, all previous works on the analysis of \algname{S-SEG} require monotonicity of $F_\xi$. Finally, in the analysis of \algname{I-SEG} we obtain last-iterate convergence guarantees by only assuming $\mu$-quasi strong monotonicity of $F$, which, as we explained before, is satisfied for some classes of \textit{non-monotone problems}.

$\diamond$ \textbf{Numerical evaluation.} In Section~\ref{sec:experiments}, we corroborate our theoretical results with experimental testing.

\begin{table*}[t]
    \centering
    \scriptsize
    \caption{\scriptsize Summary of the state-of-the-art convergence result for \algname{S-SEG} and \algname{I-SEG}. Our results are highlighted in green. Columns with convergence rates provide the upper bounds for $\Exp[\|x^K - x^*\|^2]$. Numerical constants are omitted. Notation: $\mu_{\min} = \min_{i\in [n]}\mu_i$; $\overline{\mu} = \frac{1}{n}\sum_{i\in[n]:\mu_i \ge 0} \mu_i + \frac{4}{n}\sum_{i\in[n]:\mu_i < 0} \mu_i$; $L_{\max} = \max_{i\in [n]}L_i$; $\overline{L} = \frac{1}{n}\sum_{i=1}^n L_i$ (can be much smaller than $L_{\max}$); $R_0^2 = \|x^0 - x^*\|^2$; $\sigma_{\algname{US}*}^2 = \frac{1}{n}\sum_{i=1}^n \|F_i(x^*)\|^2$; $\sigma_{\algname{IS}*}^2 = \frac{1}{n}\sum_{i=1}^n\frac{\overline{L}}{L_i} \|F_i(x^*)\|^2$ (can be much smaller than $\sigma_{\algname{US}*}^2$); $\delta$ and $\sigma^2$ = parameters from As.~\ref{as:UBV_and_quadr_growth}; $b = $ batchsize. Assumptions on constant stepsizes: \citet{mishchenko2020revisiting} uses $\gamma \le \nicefrac{1}{2L_{\max}}$, \citet{hsieh2019convergence} uses $\gamma_2 \le \gamma_1 \le \nicefrac{c}{L}$ for some positive $c > 0$, \citet{beznosikov2020distributed} uses $\gamma \le \nicefrac{1}{4L}$, and we use $\gamma_{1,\xi} = \gamma \le \nicefrac{1}{6L_{\max}}$ for \algname{S-SEG-US}, $\gamma_{1,\xi} = \nicefrac{\gamma\overline{L}}{L_\xi}, \gamma \le \nicefrac{1}{6\overline{L}}$ for \algname{S-SEG-IS}, $\gamma_{1,\xi} = \gamma \le \min\left\{\frac{\mu b}{18\delta}, \frac{1}{4\mu + \sqrt{6(L^2 + \nicefrac{\delta}{b})}}\right\}$ for \algname{I-SEG}, and in all cases $\gamma_{2,\xi} = \alpha\gamma_1$ with $\alpha < 1$. Numerical factors in our theoretical estimates can be tightened for \algname{S-SEG} when $\mu_i \ge 0$ and for \algname{I-SEG} when $\delta = 0$.}
    \label{tab:comparison_of_rates}
    \begin{threeparttable}
        \begin{tabular}{|c|c|c c c|}
        \hline
        \multirow{2}{*}{Setup} & \multirow{2}{*}{Method} & \multirow{2}{*}{Citation} & \multicolumn{2}{c|}{Convergence Rate for}\\
        & & & Constant Stepsize & Diminishing Stepsize \\
        \hline\hline
        \multirow{4}{2.6cm}{\begin{tabular}{c}
            $F(x) = \frac{1}{n}\sum\limits_{i=1}^n F_i(x)$\\
            + As.~\ref{as:F_xi_lip},~\ref{as:F_xi_str_monotonicity}
        \end{tabular}} & \multirow{3}{*}{\centering\algname{S-SEG-US}} & {\tiny\citep{mishchenko2020revisiting}}\tnote{{\color{blue}(1)}} & $(1-\gamma\mu_{\min})^KR_0^2 + \frac{\gamma \sigma_{\algname{US*}}^2}{\mu_{\min}}$ & $\frac{L_{\max}R_0^2}{\mu_{\min}}\exp\left(-\frac{\mu_{\min} K}{L_{\max}}\right) + \frac{\sigma_{\algname{US*}}^2}{\mu_{\min}^2 K}$\tnote{{\color{blue}(2)}}~~~\\
        & & \cellcolor{bgcolor2} This paper & \cellcolor{bgcolor2} $(1-\gamma\overline{\mu})^KR_0^2 + \frac{\gamma \sigma_{\algname{US*}}^2}{\overline{\mu}}$ & \cellcolor{bgcolor2} $\frac{L_{\max}R_0^2}{\overline{\mu}}\exp\left(-\frac{\overline{\mu} K}{L_{\max}}\right) + \frac{\sigma_{\algname{US*}}^2}{\overline{\mu}^2 K}$\\
        \cline{2-5}
        & \cellcolor{bgcolor2} \algname{S-SEG-IS} &\cellcolor{bgcolor2} This paper & \cellcolor{bgcolor2} $(1-\gamma\overline{\mu})^KR_0^2 + \frac{\gamma \sigma_{\algname{IS*}}^2}{\overline{\mu}}$ & \cellcolor{bgcolor2} $\frac{\overline{L}R_0^2}{\overline{\mu}}\exp\left(-\frac{\overline{\mu} K}{\overline{L}}\right) + \frac{\sigma_{\algname{IS*}}^2}{\overline{\mu}^2 K}$\\
        \hline\hline
        \multirow{6}{2.6cm}{\begin{tabular}{c}
            $F(x) = \Exp_{\xi}[F_\xi(x)]$\\
            + As.~\ref{as:lipschitzness},~\ref{as:str_monotonicity},~\ref{as:UBV_and_quadr_growth} 
        \end{tabular}} & \multirow{6}{*}{\centering\algname{I-SEG}} & {\tiny\citep{hsieh2020explore}}\tnote{{\color{blue}(3)}}  & \begin{tabular}{c}
            $(1-\gamma_1\gamma_2\mu^2)^KR_0^2 + \frac{C \sigma^2}{\mu^2 b}$\\
            $C = \gamma_1 L(1+ \gamma_1 L) + \frac{\gamma_2}{\gamma_1}$
        \end{tabular} & $\frac{L^2\sigma^2}{\mu^4 b  K^{\nicefrac{1}{3}}}$\tnote{{\color{blue}(4)}} \\
        & & {\tiny\citep{beznosikov2020distributed}}\tnote{{\color{blue}(5)}} & $(1-\gamma\mu)^KR_0^2 + \frac{\gamma \sigma^2}{\mu b}$ & $R_0^2\exp\left(-\frac{\mu K}{L}\right) + \frac{\sigma^2}{\mu^2 bK}$\tnote{\color{blue}(6)} \\
        & &\cellcolor{bgcolor2} This paper &\cellcolor{bgcolor2} $(1-\gamma\mu)^KR_0^2 + \frac{\gamma \sigma^2}{\mu b}$ & \cellcolor{bgcolor2}\begin{tabular}{c}
            $\kappa R_0^2\exp\left(-\frac{K}{\kappa}\right) + \frac{\sigma^2}{\mu^2 bK}$\\
            $\kappa = \max\left\{\frac{\delta}{\mu^2 b}, \frac{L+ \sqrt{\nicefrac{\delta}{b}}}{\mu}\right\}$
        \end{tabular}\\
        \hline
    \end{tabular}
    \begin{tablenotes}
        {\scriptsize\item [{\color{blue}(1)}] \citet{mishchenko2020revisiting} consider a regularized version of \eqref{eq:main_problem} with $\mu_{\min}$-strongly convex regularization, $F(x) = \Exp_{\xi}[F_\xi(x)]$ and $F_{\xi}(x)$ being monotone and $L_\xi$-Lipschitz. In this case, one can construct an equivalent problem with convex regularizer, $F(x) = \Exp_{\xi}[F_\xi(x)]$ and $F_{\xi}(x)$ being $\mu_{\min}$-strongly monotone and $L_\xi$-Lipschitz. If regularization is zero in the obtained problem and $\Exp_{\xi}[F_\xi(x)] = \frac{1}{n}\sum_{i=1}^n F_i(x)$, the problem from \citep{mishchenko2020revisiting} fits the considered setup with $\mu_i > 0$ for all $i\in[n]$.
        \item [{\color{blue}(2)}] \citet{mishchenko2020revisiting} do not consider diminishing stepsizes, but this rate can be derived from their Theorem 2 using similar steps as we use for our results.
        \item [{\color{blue}(3)}] \citet{hsieh2020explore} consider As.~\ref{as:UBV_and_quadr_growth}, but do not provide explicit rates when $\delta > 0$. Moreover, instead of $\mu$-quasi strong monotonicity they use slightly different assumption: $\|F(x)\| \ge \mu\|x - x^*\|$.
        \item [{\color{blue}(4)}] This bound holds only for large enough $K$ and $\sigma > 0$. Factor $\nicefrac{L^2}{\mu^4}$ is not explicitly given in \citep{hsieh2020explore}. We derive this rate using $\gamma_{1,k} = \nicefrac{\gamma_1}{(k+ t)^{\nicefrac{2}{3}}}$, $\gamma_{2,k} = \nicefrac{\gamma_2}{(k+ t)^{\nicefrac{1}{3}}}$ with largest possible $\gamma_1 \sim \nicefrac{1}{\mu}$, $\gamma_2 \sim \nicefrac{1}{\mu}$ and smallest possible $t \sim (\nicefrac{L}{\mu})^3$ for given $\gamma_1$ and $\gamma_2$.
        \item [{\color{blue}(5)}] Results are derived for the case $\delta = 0$. \citet{beznosikov2020distributed} study a distributed version of \algname{I-SEG}.
        \item [{\color{blue}(6)}] This result is derived for the stepsize $\gamma$ that explicitly depends on $K$ and $\sigma^2$, which makes it hard to use this stepsize in practice.
        }
    \end{tablenotes}
    \end{threeparttable}
\end{table*}

\subsection{Related Work}
\paragraph{Non-monotone VIP with special structure.}
Recent works of \cite{daskalakis2021complexity} and \cite{diakonikolas2021efficient} show that, for general non-monotone \ref{eq:main_problem}, the computation of approximate first-order locally optimal solutions is intractable, motivating the identification of structural assumptions on the objective function for which these intractability barriers can be bypassed.

In this work, we focus on such settings (structured non-monotone operators) for which we are able to provide tight convergence guarantees and avoid the standard issues (like cycling and divergence of the methods) appearing in the more general non-monotone regime. In particular, we focus on quasi-strongly monotone VIPs \eqref{eq:str_monotonicity}.
Recently, similar conditions have been used in several papers to provide convergence guarantees of algorithms for solving such structured classes of non-monotone problems. For example,\cite{yang2020global} focuses on analyzing alternating gradient descent ascent under the Two-sided Polyak- Lojasiewicz inequality, while \cite{hsieh2020explore} provides convergence guarantees of double stepsize stochastic extragradient for problems satisfying the error bound condition. \cite{song2020optimistic} and \cite{loizou2021stochastic} study the optimistic dual extrapolation and the stochastic gradient descent-ascent and stochastic consensus optimization method, respectively, for solving quasi-strongly monotone problems. \cite{kannan2019optimal} provides an analysis for the stochastic extragradient for the class of strongly pseudo-monotone VIPs. The convergence of Hamiltonian methods for solving (stochastic) sufficiently bilinear games (class of structured non-monotone games) was studied in \cite{abernethy2021last} and \cite{loizou2020stochastic}.

\paragraph{On the analysis of stochastic extragradient.} In the context of VIP, \algname{SEG} is also known as \algname{Stochastic Mirror Prox}~\citet{juditsky2011solving}. Several novel variants of \algname{SEG} have been proposed and analyzed in recent papers, such as accelerated versions~\citep{chen2017accelerated}, single-call variants (a.k.a. optimistic methods)~\citet{hsieh2019convergence}, and a version with player sampling in the context of multi-player games~\citep{jelassi2020extragradient}. Comparing our results with these variants is outside of the scope of this paper. In this work, we focus on analyzing and better understanding the properties of the standard version of \algname{SEG} with independent~\eqref{eq:I_SEG} or same sample~\eqref{eq:S_SEG}.

Recent analysis of \algname{SEG} by~\citet{mishchenko2020revisiting},~\citet{hsieh2020explore} and~\citet{beznosikov2020distributed} have extended the seminal results of~\citet{juditsky2011solving} in the unconstrained case. We compare their results with our work in Table~\ref{tab:comparison_of_rates}.

\algname{SEG} has also been analyzed in settings that significantly differ from ours such as in the constrained pseudomonotone case~\citep{kannan2019optimal} and the unconstrained bilinear case~\citep{li2021convergence}. 

\paragraph{Arbitrary sampling paradigm.}
The first analysis of a stochastic optimization algorithm with
an arbitrary sampling was performed by \cite{richtarik2016optimal} in the context of randomized coordinate descent method for strongly convex functions. This arbitrary sampling paradigm was later extended in different settings, including  accelerated coordinate descent for (strongly) convex functions \citep{hanzely2019accelerated, qu2016coordinate}, randomized iterative methods for solving linear systems \citep{richtarik2020stochastic, loizou2020momentum,loizou2020convergence}, randomized gossip algorithms \citep{loizou2021revisiting},  variance-reduced methods with convex \citep{khaled2020unified}, and nonconvex \citep{horvath2019nonconvex} objectives. The first analysis of \algname{SGD} under the arbitrary sampling was proposed in \cite{gower2019sgd} for (quasi)-strongly convex problems and later extended to the non-convex regime in \cite{gower2021sgd} and \cite{khaled2020better}. In the area of smooth games and variational inequality problems the first papers that provide an analysis of stochastic algorithms under the arbitrary sampling paradigm are \citep{loizou2020stochastic,loizou2021stochastic}. In \cite{loizou2020stochastic,loizou2021stochastic}, the authors focus on algorithms like the stochastic Hamiltonian method, the stochastic gradient descent ascent, and the stochastic consensus optimization. To the best of our knowledge, our work is the first that provides an analysis of \algname{SEG} under the arbitrary sampling paradigm.

\subsection{Paper Organization}
Section~\ref{sec:general_analysis} introduces our unified theoretical framework that is applied for the analysis of \algname{S-SEG} and \algname{I-SEG} in Sections~\ref{sec:S_SEG}~and~\ref{sec:I_SEG} respectively. In section~\ref{sec:experiments}, we report the result of our numerical experiments, and we make the concluding remarks in Section~\ref{sec:conclusion}. Proofs, technical details, and additional experiments are given in Appendix. We defer the discussion of our results for quasi monotone ($\mu = 0$) problems to Appendix~\ref{sec:discussion_quasi_mon}.

\section{GENERAL ANALYSIS OF \algname{SEG}}\label{sec:general_analysis}
To analyze the convergence of \algname{SEG}, we consider a family of methods
\begin{equation}
	x^{k+1} = x^k - \gamma_{\xi^k} g_{\xi^k}(x^k), \label{eq:general_method}
\end{equation}
where $g_{\xi^k}(x^k)$ is some stochastic operator evaluated at point $x^k$ and $\xi^k$ encodes the randomness/stochasticity appearing at iteration $k$ (e.g., it can be the sample used at step $k$). Parameter $\gamma_{\xi^k}$ is the stepsize that is allowed to depend on $\xi^k$. Inspired by \cite{gorbunov2020unified}, let us introduce the following general assumption on operator $g_{\xi^k}(x^k)$, stepsize $\gamma_{\xi^k}$, and the problem~\eqref{eq:main_problem}.

\begin{assumption}\label{as:unified_assumption_general}
	We assume that there exist non-negative constants $A, B, C, D_1, D_2 \geq 0$, $\rho \in [0,1]$, and (possibly random) non-negative sequence $\{G_k\}_{k\ge 0}$ such that
	\begin{equation}
	\label{eq:second_moment_bound_general}
	    \Exp_{\xi^k}\left[\gamma_{\xi^k}^2\|g_{\xi^k}(x^k)\|^2\right]\leq 2A P_k + C\|x^k - x^*\|^2 + D_1,
	\end{equation}
		\begin{equation}
		\label{eq:P_k_general}
 P_{k} \geq \rho\|x^k - x^*\|^2 + BG_k - D_2,
	\end{equation}
	where $P_k = \Exp_{\xi^k}\left[\gamma_{\xi^k}\langle g_{\xi^k}(x^k), x^k - x^* \rangle\right]$.
\end{assumption}

Although inequalities \eqref{eq:second_moment_bound_general} and \eqref{eq:P_k_general} may seem unnatural, they are satisfied with certain parameters for several variants of \algname{S-SEG} and \algname{I-SEG} under reasonable assumptions on the problem and the stochastic noise. Moreover, these inequalities have a simple intuition behind them. That is, inequality \eqref{eq:second_moment_bound_general} is a generalization of the expected cocoercivity introduced in \cite{loizou2021stochastic}, adjusted to the case of biased estimators $g_{\xi^k}(x^k)$ of $F(x^k)$, as it is the case for \algname{SEG}. The biasedness of  $g_{\xi^k}(x^k)$ and the (possible) dependence of $\gamma_{\xi^k}$ on $\xi^k$ force us to introduce the expected inner product $P_k = \Exp_{\xi^k}\left[\gamma_{\xi^k}\langle g_{\xi^k}(x^k), x^k - x^* \rangle\right]$ instead of using $P_k \sim \langle F(x^k), x^k - x^* \rangle$ as in \cite{loizou2021stochastic}. Moreover, unlike the expected cocoercivity, our assumption \eqref{eq:second_moment_bound_general} does not imply (star-)cocoercivity of $F$. However, when we derive \eqref{eq:second_moment_bound_general} for \algname{S-SEG} and \algname{I-SEG} we rely in Lipschitzness of $F$ or its stochastic realizations. The terms $C\|x^k - x^*\|^2$ and $D_1$ characterize the noise structure, and $A$ is typically some constant smaller than $\nicefrac{1}{2}$.

Next, inequality \eqref{eq:P_k_general} can be seen as a modification of $\mu$-quasi strong monotonicity of $F$ \eqref{eq:str_monotonicity}. Indeed, if we had $\gamma_{\xi^k} = \gamma$ and $\Exp_{\xi^k}[g_{\xi^k}(x^k)] = F(x^k)$, then we would have $P_k = \gamma \langle F(x^k), x^k - x^* \rangle$ and inequality \eqref{eq:P_k_general} would have been satisfied with $\rho = \gamma\mu$, $B = 0$, $G_k = 0$, $D_2 = 0$ for $F$ being $\mu$-quasi strongly monotone. However, because of the biasedness of $g_{\xi^k}(x^k)$ we have to account to the noise encoded by $D_2$. In inequality \eqref{eq:P_k_general}, $\rho$ also typically depends on some quantity related to the quasi-strong monotonicity and the stepsize. Moreover, when $g_{\xi^k}(x^k)$ corresponds to \algname{SEG}, we are able to show that $B > 0$ with $G_k$ being an upper bound for $\|F(x^k)\|^2$ up to the factors depending on the stepsize selection (see Sections~~\ref{sec:S_SEG}~and~\ref{sec:I_SEG}).

Under this assumption, we derive the following result.

\begin{theorem}\label{thm:main_theorem_general_main}
	 Let Assumption~\ref{as:unified_assumption_general} hold with $A\leq \nicefrac{1}{2}$ and $\rho > C \ge 0$. Then, the iterates of \algname{SEG} given by \eqref{eq:general_method} satisfy
	\begin{equation}
		\Exp \left[\|x^K - x^*\|^2\right] \leq  (1+C-\rho)^K\|x^0 - x^*\|^2 + \frac{D_1+D_2}{\rho - C}.\notag 
	\end{equation}
	In the case that Assumption~\ref{as:unified_assumption_general} holds with $\rho = C = 0$, $B > 0$, then for all $K\ge 0$, the iterates of \algname{SEG} given by \eqref{eq:general_method} satisfy
	\vspace{-3mm}
	\begin{equation}
		\frac{1}{K+1}\sum\limits_{k=0}^K \Exp[G_k] \leq  \frac{\|x^0 - x^*\|^2}{B(K+1)} + \frac{D_1+D_2}{B}. \notag 
	\end{equation}
\end{theorem}
This theorem establishes linear convergence rate when $\rho > C \ge 0$ and $\cO(\nicefrac{1}{K})$ rate when $\rho = C = 0$, $B > 0$ to a neighborhood of the solution with the size proportional to the noise parameters $D_1, D_2$. In all special cases that we consider, the first case corresponds to the quasi-strongly monotone problems and the second one -- to quasi-monotone problems. All the rates from this paper are derived via Theorem~\ref{thm:main_theorem_general_main}.

\section{SAME-SAMPLE \algname{SEG} (\algname{S-SEG})}\label{sec:S_SEG}

Consider the situation when we have access to Lipschitz-continuous stochastic realization $F_\xi(x)$ and can compute $F_\xi$ at different points for the same $\xi$. For such problems, we consider \ref{eq:S_SEG}.

\subsection{Arbitrary Sampling}\label{sec:AS_examples}
Below we introduce reasonable assumptions on the stochastic trajectories that cover a wide range of sampling strategies. Therefore, following \cite{gower2019sgd,loizou2021stochastic}, we use the name \textit{arbitrary sampling} to define this setup. First, we assume Lischitzness of $F_\xi$.

\begin{assumption}\label{as:F_xi_lip}
	We assume that for all $\xi$ there exists $L_\xi > 0$ such that operator $F_\xi(x)$ is $L_\xi$-Lipschitz, i.e., for all $x\in \R^d$
	\begin{equation}
		\|F_\xi(x) - F_\xi(y)\| \le L_\xi\|x-y\|. \label{eq:F_xi_lip}
	\end{equation}
\end{assumption}

The next assumption can be considered as a relaxation of standard strong monotonicity allowing $F_\xi(x)$ to be non-monotone with a certain structure.

\begin{assumption}\label{as:F_xi_str_monotonicity}
	We assume that for all $\xi$ operator $F_\xi(x)$ is $(\mu_{\xi},x^*)$-strongly monotone, i.e., there exists \textbf{(possibly negative)} $\mu_\xi \in \R$ such that for all $x\in \R^d$
	\begin{equation}
		\langle F_\xi(x) - F_\xi(x^*), x- x^*\rangle \ge \mu_\xi\|x - x^*\|^2. \label{eq:F_xi_str_monotonicity}
	\end{equation}
\end{assumption}

We emphasize that some $\mu_{\xi}$ are allowed to be arbitrary heterogeneous and even negative, which allows to have non-monotone $F_\xi$. Moreover, if $F_\xi$ is $L_\xi$-Lipschitz, then in view of Cauchy-Schwarz inequality, \eqref{eq:F_xi_str_monotonicity} holds with $-L_\xi \le \mu_\xi \le L_\xi$. Indeed, inequality \eqref{eq:F_xi_lip} implies $-L_\xi \| x- x^*\|^2\le -\| F_\xi(x) - F_\xi(x^*)\|\cdot \| x- x^*\|\le \langle F_\xi(x) - F_\xi(x^*), x- x^*\rangle \le \| F_\xi(x) - F_\xi(x^*)\|\cdot \| x- x^*\| \le L_\xi \| x- x^*\|^2$. However, $\mu_\xi$ can be much larger than $-L_\xi$. When $F_\xi(x^*) = 0$ and $\mu_\xi \ge 0$, inequality \eqref{eq:F_xi_str_monotonicity} recovers quasi-strong monotonicity of $F_\xi$, i.e., $F_\xi$ can be non-monotone even when $\mu_\xi \ge 0$.

Finally, we assume that the following two conditions are satisfied:
\begin{eqnarray}
    \Exp_{\xi^k}[\gamma_{1,\xi^k}F_{\xi^k}(x^*)] &=& 0, \label{eq:SS_SEG_AS_stepsizes_1}\\ \Exp_{\xi^k}[\gamma_{1,\xi^k}\mu_{\xi^k}(\obf_{\{\mu_{\xi^k} \ge 0\}} + 4\cdot\obf_{\{\mu_{\xi^k} < 0\}})] &\ge& 0, \label{eq:SS_SEG_AS_stepsizes_2}
\end{eqnarray}
where $\obf_{\texttt{condition}} = 1$ if \texttt{condition} holds, and $\obf_{\texttt{condition}} = 0$ otherwise. Here, \eqref{eq:SS_SEG_AS_stepsizes_1} is a generalization of unbiasedness at $x^*$, since $F(x^*) = 0$, and the left-hand side of \eqref{eq:SS_SEG_AS_stepsizes_2} is a generalization of the averaged quasi-strong monotonicity constant multiplied by the stepsize. Moreover, \eqref{eq:SS_SEG_AS_stepsizes_2} holds when all $\mu_\xi \ge 0$, which is typically assumed in the analysis of \algname{S-SEG}. The numerical constant $4$ in \eqref{eq:SS_SEG_AS_stepsizes_2} appears mainly due to the technical reasons coming from our proof technique.

To better illustrate the generality of conditions \eqref{eq:SS_SEG_AS_stepsizes_1}-\eqref{eq:SS_SEG_AS_stepsizes_2}, let us provide three different examples where these conditions are satisfied. 
 In all examples, we assume that $F(x) = \frac{1}{n}\sum_{i=1}^n F_i(x)$ and $F_i(x)$ is $(\mu_i,x^*)$-strongly monotone and $L_i$-Lipschitz.

Let us start by considering the standard single-element uniform sampling strategy.
\begin{example}[Uniform sampling]\label{ex:uniform_sampling}
    Let $\xi^k$ be sampled from the uniform distribution on $[n]$, i.e., for all $i\in [n]$ we have $\Prob{\xi^k = i} = p_i \equiv \nicefrac{1}{n}$. If
    \begin{equation}
        \overline{\mu} = \frac{1}{n}\sum\limits_{i: \mu_i \geq 0} \mu_i + \frac{4}{n}\sum\limits_{i: \mu_i < 0} \mu_i \ge 0 \label{eq:average_mu}
    \end{equation}
    and $\gamma_{1,\xi} \equiv \gamma > 0$, then conditions \eqref{eq:SS_SEG_AS_stepsizes_1}-\eqref{eq:SS_SEG_AS_stepsizes_2} hold.
\end{example}

In the above example, the oracle is unbiased and, as the result, we use constant stepsize $\gamma_{1,\xi} = \gamma$. Next, we note that $\overline{\mu}$  satisfies: $\mu \ge \overline{\mu} \ge \mu_{\min}$, where $\mu$ is the parameter from \eqref{eq:str_monotonicity}, and $\mu_{\min} = \min_{i\in [n]}\mu_i$. 
Moreover, we emphasize that to fulfill conditions \eqref{eq:SS_SEG_AS_stepsizes_1}-\eqref{eq:SS_SEG_AS_stepsizes_2} in Example~\ref{ex:uniform_sampling}, and in the following examples we only need to assume that parameter $\gamma$ is positive. However, to be able to derive convergence guarantees for \algname{S-SEG} under different sampling strategies we will later introduce an additional upper bound for $\gamma$ (see Section~\ref{sec:sseg_convergence}).

Next, we consider a uniform sampling strategy of mini-batching without replacement.

\begin{example}[$b$-nice sampling]\label{ex:nice}
    Let $\xi$ be a random subset of size $b \in [n]$ chosen from the uniform distribution on the family of all $b$-elements subsets of $[n]$. Next, let $F_\xi(x) = \frac{1}{b}\sum_{i\in \xi} F_i(x)$. If
    \begin{equation*}
        \overline{\mu}_{b-\algname{NICE}} = \frac{1}{\binom{n}{b}}\left(\sum\limits_{\overset{S\subseteq [n],}{|S| = b :\mu_{S} \ge 0}}\!\mu_{S} + 4\sum\limits_{\overset{S\subseteq [n],}{|S| = b :\mu_{S} < 0}}\!\mu_{S}\right) \ge 0,
    \end{equation*}
    where $\mu_{S} \ge \frac{1}{b}\sum_{i\in S} \mu_{i}$ is such that the operator $\frac{1}{b}\sum_{i\in S} F_{i}(x)$ is $(\mu_{S},x^*)$-strongly monotone, and $\gamma_{1,\xi} \equiv \gamma > 0$, then conditions \eqref{eq:SS_SEG_AS_stepsizes_1}-\eqref{eq:SS_SEG_AS_stepsizes_2} hold.
\end{example}

Finally, we provide an example of a non-uniform sampling.

\begin{example}[Importance sampling]\label{ex:importance_sampling}
    Let $\xi^k$ be sampled from the following distribution: for $i \in [n]$
    \begin{equation}
        \Prob{\xi^k = i} = p_i = \frac{L_i}{\sum\limits_{j=1}^n L_j}. \label{eq:importance_smapling_prob_main}
    \end{equation}
    If \eqref{eq:average_mu} is satisfied and $\gamma_{1,\xi} = \nicefrac{\gamma \overline{L}}{L_\xi}$, where $\overline{L} = \frac{1}{n}\sum_{i=1}^n L_i$, $\gamma > 0$, then conditions \eqref{eq:SS_SEG_AS_stepsizes_1}-\eqref{eq:SS_SEG_AS_stepsizes_2} hold.
\end{example}

We provide rigorous proofs that the above examples, as well as additional ones, fit the conditions \eqref{eq:SS_SEG_AS_stepsizes_1}-\eqref{eq:SS_SEG_AS_stepsizes_2} in Appendix~\ref{sec:AS_examples_appendix}.

\subsection{Convergence of \algname{S-SEG}}\label{sec:sseg_convergence}
Having explained the main sampling strategies of \algname{S-SEG} that we are focusing on this work, let us now present our main convergence analysis results for this method.

Let assumptions~\ref{as:F_xi_lip}~and~\ref{as:F_xi_str_monotonicity} hold, and let us select the stepsize $\gamma_{1,\xi^k}$ such that conditions \eqref{eq:SS_SEG_AS_stepsizes_1}-\eqref{eq:SS_SEG_AS_stepsizes_2} are satisfied, and 
\begin{equation}
	    \gamma_{1,\xi^k} \leq \frac{1}{4|\mu_{\xi^k}| + \sqrt{2}L_{\xi^k}}. \label{eq:SS_SEG_AS_stepsizes_3}
\end{equation}
Then one is able to show that  Assumption~\ref{as:unified_assumption_general} is satisfied (see Appendix~\ref{Appendix_MainResult} for this derivation) for $g_{\xi^k}(x^k) = F_{\xi^k}\left(x^k - \gamma_{1,\xi^k} F_{\xi^k}(x^k)\right)$ and $\gamma_{\xi^k} = \gamma_{2,\xi^k}$. In particular, under these conditions, Assumption~\ref{as:unified_assumption_general} holds  with $A = 2\alpha$, $C = 0$, $B = \nicefrac{1}{2}$, $D_1 = 6\alpha^2\sigma_{\algname{AS}}^2$, $D_2 = \nicefrac{3\alpha\sigma_{\algname{AS}}^2}{2}$, and 
\begin{gather}
    \sigma_{\algname{AS}}^2 = \Exp_\xi\left[\gamma_{1,\xi}^2\|F_\xi(x^*)\|^2\right],\label{eq:sigma_AS_definition}\\
    \rho = \frac{\alpha}{2}\Exp_{\xi^k}[\gamma_{1,\xi^k}\mu_{\xi^k}(\obf_{\{\mu_{\xi^k} \ge 0\}} + 4\cdot\obf_{\{\mu_{\xi^k} < 0\}})],\label{eq:rho_AS_definition}\\
    G_k = \alpha \Exp_{\xi^k}\left[\gamma_{1,\xi^k}^2\widehat{B}_{\xi^k}\|F_{\xi^k}(x^k)\|^2\right],\label{eq:G_k_AS_definition}
\end{gather}
where $\widehat{B}_{\xi^k} = 1 - 4|\mu_{\xi^k}|\gamma_{1,\xi^k} - 2L_{\xi^k}^2\gamma_{1,\xi^k}^2$. Here \eqref{eq:SS_SEG_AS_stepsizes_3} implies that $\widehat{B}_{\xi^k} \ge 0$. Therefore, applying our general result (Theorem~\ref{thm:main_theorem_general_main}), we derive\footnote{For simplicity of exposition, in the main paper we focus on the case $\rho > 0$. For our results for $\rho = 0$, we refer the reader to Appendix~\ref{sec:discussion_quasi_mon} and~\ref{Appendix_MainResult}.} the following convergence guarantees for \algname{S-SEG}.

\begin{theorem}\label{thm:SEG_same_sample_convergence_AS_main}
	Let Assumptions~\ref{as:F_xi_str_monotonicity}~and~\ref{as:F_xi_lip} hold. If $\gamma_{2,\xi^k} = \alpha \gamma_{1,\xi_k}$, $0 < \alpha \le \nicefrac{1}{4}$, and $\gamma_{1,\xi^k}$ satisfies \eqref{eq:SS_SEG_AS_stepsizes_1}-\eqref{eq:SS_SEG_AS_stepsizes_2} and \eqref{eq:SS_SEG_AS_stepsizes_3} and $\rho$ from \eqref{eq:rho_AS_definition} is positive, then the iterates of \ref{eq:S_SEG} satisfy
	\begin{equation*}
		\Exp\left[\|x^K - x^*\|^2\right]\! \leq\! \left(1\! - \!\rho\right)^K\!\|x^0 - x^*\|^2 + \frac{3\alpha\left(4\alpha\! + \!1\right)\sigma_{\algname{AS}}^2}{2\rho},
	\end{equation*}
	where $\sigma_{\algname{AS}}^2$ is defined in \eqref{eq:rho_AS_definition}.
\end{theorem}

The next corollary establishes the convergence rate with diminishing stepsizes allowing to reduce the size of the neighborhood.

\begin{corollary}\label{cor:str_mon_AS_main}
    Let Assumptions~\ref{as:F_xi_str_monotonicity}~and~\ref{as:F_xi_lip} hold, and let $\gamma_{2,\xi^k} = \alpha \gamma_{1,\xi_k}$ with $\alpha = \nicefrac{1}{4}$, and $\gamma_{1,\xi^k} = \beta_k \cdot\gamma_{\xi^k}$, where $\gamma_{\xi^k}$ satisfies \eqref{eq:SS_SEG_AS_stepsizes_1}, \eqref{eq:SS_SEG_AS_stepsizes_2}, \eqref{eq:SS_SEG_AS_stepsizes_3}, and $\widetilde{\rho} = \frac{1}{8}\Exp_{\xi^k}[\gamma_{\xi^k}\mu_{\xi^k}(\obf_{\{\mu_{\xi^k} \ge 0\}} + 4\cdot\obf_{\{\mu_{\xi^k} < 0\}})]$.
    Assume that $\widetilde{\rho} > 0$.
    Then, for all $K \ge 0$ and $\{\beta_k\}_{k\ge 0}$ such that
	\begin{eqnarray}
		\text{if } K \le \frac{1}{\widetilde{\rho}}, && \beta_k = 1,\notag \\
		\text{if } K > \frac{1}{\widetilde{\rho}} \text{ and } k < k_0, && \beta_k = 1,\label{eq:stepsize_policy}\\
		\text{if } K > \frac{1}{\widetilde{\rho}} \text{ and } k \ge k_0, && \beta_k = \frac{2}{2 + \widetilde{\rho}(k - k_0)},\notag
	\end{eqnarray}
where $k_0 = \left\lceil\nicefrac{K}{2} \right\rceil$, we have that the iterates of \ref{eq:S_SEG} satisfy
	\begin{equation*}
	    \Exp\left[\|x^K - x^*\|^2\right] \!\le\! \frac{32\|x^0\! - \!x^*\|^2}{\widetilde{\rho}}\exp\left(-\frac{\widetilde{\rho} K}{2}\right) + \frac{27\sigma_{\algname{AS}}^2}{\widetilde{\rho}^2 K}.
	\end{equation*}
\end{corollary}
We notice that the stepsize schedule from the above corollary requires the knowledge of the total number of iterations $K$.

Next, we provide the results for the special cases described in Section~\ref{sec:AS_examples}. These results are direct corollaries of Theorem~\ref{thm:SEG_same_sample_convergence_AS_main} and Corollary~\ref{cor:str_mon_AS_main}.

\paragraph{\algname{S-SEG-US}: \algname{S-SEG} with Uniform Sampling.} Consider the setup from Example~\ref{ex:uniform_sampling}. Then, Theorem~\ref{thm:SEG_same_sample_convergence_AS_main} implies that for constant stepsizes  $\gamma_{1,\xi^k}$ and $\gamma_{2,\xi^k}$, the iterates of \algname{S-SEG-US} satisfy
\begin{eqnarray*}
    \Exp\left[\|x^K - x^*\|^2\right] &\leq& \left(1 - \frac{\alpha\gamma\overline{\mu}}{2}\right)^K\|x^0 - x^*\|^2\\
    &&\quad+ \frac{3\left(4\alpha + 1\right)\gamma\sigma_{\algname{US*}}^2}{\overline{\mu}},
\end{eqnarray*}
	where $\gamma \leq \nicefrac{1}{6L_{\max}}$, $L_{\max} = \max_{i\in [n]}L_i$, and $\sigma_{\algname{US*}}^2 = \frac{1}{n}\sum\limits_{i=1}^n\|F_i(x^*)\|^2$. For diminishing stepsizes following  \eqref{eq:stepsize_policy}, Corollary~\ref{cor:str_mon_AS_main} implies that for  the iterates of \algname{S-SEG-US} $\Exp\left[\|x^K - x^*\|^2\right]$ is of the order
	\begin{equation*}
	     \cO\left(\frac{L_{\max}R_0^2}{\overline{\mu}}\exp\left(-\frac{\overline{\mu} K}{ L_{\max}}\right) + \frac{\sigma_{\algname{US*}}^2}{\overline{\mu}^2 K}\right),
	\end{equation*}
	where $R_0 = \|x^0 - x^*\|^2$. The previous SOTA rate for \algname{S-SEG-US} \citep{mishchenko2020revisiting} assumes that $\mu_i > 0$ for all $i\in [n]$ and depends on $\mu_{\min} = \min_{i\in[n]}\mu_i$ which can be much smaller than $\overline{\mu}$. That is, our results for \algname{S-SEG-US} are derived under weaker assumptions and are tighter than the previous ones for this method.

\paragraph{\algname{S-SEG-NICE}: \algname{S-SEG} with $b$-Nice Sampling.}

Consider the setup from Example~\ref{ex:nice}.
Then, Theorem~\ref{thm:SEG_same_sample_convergence_AS_main} implies that for constant stepsizes  $\gamma_{1,\xi^k}$ and $\gamma_{2,\xi^k}$, the iterates of \algname{S-SEG-NICE} satisfy
\begin{eqnarray*}
    \Exp\left[\|x^K - x^*\|^2\right] &\leq& \left(1 - \frac{\alpha\gamma\overline{\mu}_{b-\algname{NICE}}}{2}\right)^K\|x^0 - x^*\|^2\\
    &&\quad + \frac{3\left(4\alpha + 1\right)\gamma\sigma_{b-\algname{NICE*}}^2}{\overline{\mu}_{b-\algname{NICE}}},
\end{eqnarray*}
where $\gamma \leq \nicefrac{1}{6L_{b-\algname{NICE}}}$, $L_{b-\algname{NICE}} = \max_{S\subseteq [n], |S|=b}L_S$, and $\sigma_{b-\algname{NICE*}}^2 = \frac{n-b}{b(n-1)}\sigma_{\algname{US}*}$. For diminishing stepsizes following  \eqref{eq:stepsize_policy}, Corollary~\ref{cor:str_mon_AS_main} implies that for the iterates of \algname{S-SEG-NICE} $\Exp\left[\|x^K - x^*\|^2\right]$ is of the order
\begin{equation*}
    \cO\left(\frac{L_{b-\algname{NICE}}R_0^2}{\overline{\mu}_{b-\algname{NICE}}}\exp\left(-\frac{\overline{\mu}_{b-\algname{NICE}} K}{ L_{b-\algname{NICE}}}\right) + \frac{\sigma_{b-\algname{NICE*}}^2}{\overline{\mu}_{b-\algname{NICE}}^2 K}\right).
\end{equation*}
These rates show the benefits of mini-batching without replacement: the linearly decaying term decreases faster than the corresponding one for \algname{S-SEG-US} since $L_{b-\algname{NICE}} \le L_{\max}$ and $\overline{\mu}_{b-\algname{NICE}} \ge \overline{\mu}$, and the variance $\sigma_{b-\algname{NICE*}}^2$ is smaller than $\sigma_{\algname{US}*}^2$, i.e., $\cO(1/K)$ term for \algname{S-SEG-NICE} is more than $b$-times smaller than the corresponding term for \algname{S-SEG-US}. 

Moreover, we highlight that for $n = b$ we recover the rate of deterministic \algname{EG} up to numerical factors. That is, for the deterministic \algname{EG} we obtain $\|x^K - x^*\|^2 \le \left(1 - \nicefrac{\alpha\gamma \mu}{2}\right)^K\|x^0 - x^*\|^2$ with $\gamma \le \nicefrac{1}{6L}$, since $\overline{\mu}_{b-\algname{NICE}} = \mu$ and $L_{b-\algname{NICE}} = L$ in this case. This fact highlights the tightness of our analysis, since in the known special cases our general theorem either recovers the best-known results (as for \algname{EG}) or improves them (as for \algname{S-SEG-US}).

\paragraph{\algname{S-SEG-IS}: \algname{S-SEG} with Importance Sampling.} Finally, let us consider the third special case described in Example~\ref{ex:importance_sampling}. In this case, if $\gamma \leq \nicefrac{1}{6\overline{L}}$, $\overline{L} = \frac{1}{n}\sum_{i=1}^n L_i$, Theorem~\ref{thm:SEG_same_sample_convergence_AS_main} implies that for constant stepsizes  $\gamma_{1,\xi^k}$ and $\gamma_{2,\xi^k}$, the iterates of \algname{S-SEG-IS} satisfy
\begin{eqnarray*}
    \Exp\left[\|x^K - x^*\|^2\right] &\leq& \left(1 - \frac{\alpha\gamma\overline{\mu}}{2}\right)^K\|x^0 - x^*\|^2\\
    &&\quad + \frac{3\left(4\alpha + 1\right)\gamma\sigma_{\algname{IS*}}^2}{\overline{\mu}},
\end{eqnarray*}
where $\sigma_{\algname{IS*}}^2 = \frac{1}{n}\sum\limits_{i=1}^n\frac{\overline{L}}{L_i}\|F_i(x^*)\|^2$. For diminishing stepsizes following  \eqref{eq:stepsize_policy}, Corollary~\ref{cor:str_mon_AS_main} implies that for the iterates of \algname{S-SEG-IS} $\Exp\left[\|x^K - x^*\|^2\right]$ is of the order
\begin{equation*}
    \cO\left(\frac{\overline{L}R_0^2}{\overline{\mu}}\exp\left(-\frac{\overline{\mu} K}{ \overline{L}}\right) + \frac{\sigma_{\algname{IS*}}^2}{\overline{\mu}^2 K}\right).
\end{equation*}
Note that, in contrast to the rate of \algname{S-SEG-US}, the above rate depends on the averaged Lipschitz constant $\overline{L}$ that can be much smaller than the worst constant $L_{\max}$. In such cases, exponentially decaying term for \algname{S-SEG-IS} is much better than the one for \algname{S-SEG-US}. Moreover, theory for \algname{S-SEG-IS} allows to use much larger $\gamma$. Next, typically, larger norm of $F_i(x^*)$ implies larger $L_i$, e.g., $\|F_i(x^*)\|^2 \sim L_i^2$. In such situations, $\sigma_{\algname{IS}*}^2 \sim (\overline{L})^2$ and $\sigma_{\algname{US}*}^2 \sim \overline{L^2} = \frac{1}{n}\sum_{i=1}^n L_i^2 \ge (\overline{L})^2$.

\section{INDEPENDENT-SAMPLES \algname{SEG} (\algname{I-SEG})}\label{sec:I_SEG}

In this subsection, we consider \ref{eq:I_SEG}. We make the following assumption used in \cite{hsieh2020explore}.\footnote{Although the analysis of \citet{hsieh2019convergence} can be conducted with $\delta > 0$, the authors do not provide explicit rates in their paper for the case $\delta > 0$.}

\begin{assumption}\label{as:UBV_and_quadr_growth}
	For all $x\in \R^d$ the unbiased estimator $F_\xi(x)$ of $F(x)$, i.e., $\Exp_\xi[F_\xi(x)] = F(x)$, satisfies 
	\begin{equation}
		\Exp_{\xi}\left[\|F_\xi(x) - F(x)\|^2\right] \le \delta\|x - x^*\|^2 + \sigma^2, \label{eq:UBV_main}
	\end{equation}	 
	where $\delta \ge 0$, $\sigma \ge 0$, and $x^*$ is the solution of \ref{eq:main_problem}.
\end{assumption}

Note that when $\delta = 0$, \eqref{eq:UBV_main} recovers the classical assumption of uniformly bounded variance \citep{juditsky2011solving}.

In \algname{I-SEG}, we use mini-batched estimators:
\begin{eqnarray*}
    F_{\xi_1^k}(x^k) &=& \frac{1}{b}\sum\limits_{i=1}^b F_{\xi_1^k(i)}(x^k),\\ F_{\xi_2^k}(x^k) &=& \frac{1}{b}\sum\limits_{i=1}^b F_{\xi_2^k(i)}(x^k - \gamma_1 F_{\xi_1^k}(x^k)),
\end{eqnarray*}
where $\xi_1^k(1),\ldots, \xi_1^k(b), \xi_2^k(1), \ldots, \xi_2^k(b)$ are i.i.d.\ samples satisfying Assumption~\ref{as:UBV_and_quadr_growth}.

In this setup (where Assumption~\ref{as:UBV_and_quadr_growth} holds), if $\gamma_2 = \alpha\gamma_1$ with $0< \alpha < 1$, and 
\begin{equation}
    \gamma_1 = \gamma \le \min\left\{\frac{\mu b}{18\delta}, \frac{1}{4\mu + \sqrt{6(L^2 + \nicefrac{\delta}{b})}}\right\}, \label{eq:I_SEG_stepsize_condition}
\end{equation}
then Assumption~\ref{as:unified_assumption_general} is satisfied\footnote{See Appendix~\ref{AppendixISEG} for the derivation.} for $g_{\xi^k}(x^k) = F_{\xi_2^k}\left(x^k - \gamma_{1} F_{\xi_1^k}(x^k)\right)$ and $\gamma_{\xi^k} = \gamma_{2}$. In particular, in this setting, Assumption~\ref{as:unified_assumption_general} holds with $A = 2\alpha$, $C = \nicefrac{9\delta\alpha^2\gamma^2}{b}$, $B = \nicefrac{1}{2}$, $D_1 = D_2 = \nicefrac{6\alpha^2\sigma^2}{b}$, $\rho = \nicefrac{\alpha\gamma\mu}{4}$, $G_k = \Exp_{\xi_1^k}\left[\|F_{\xi_1^k}(x^k)\|^2\right],$ $B = \nicefrac{\alpha\gamma^2}{2}$. Therefore, applying our general result (Theorem~\ref{thm:main_theorem_general_main}), we obtain the following convergence guarantees for \algname{I-SEG}.
\begin{theorem}\label{thm:I_SEG_convergence_main}
	Let Assumptions~\ref{as:lipschitzness},~\ref{as:str_monotonicity}~and~\ref{as:UBV_and_quadr_growth} hold. If $\mu > 0$, $\gamma_{2} = \alpha \gamma_{1}$, $0 < \alpha \le \nicefrac{1}{4}$, and $\gamma_1 = \gamma$ satisfies \eqref{eq:I_SEG_stepsize_condition},  then the iterates of \ref{eq:I_SEG} satisfy 
	\begin{equation*}
		\Exp\left[\|x^K - x^*\|^2\right] \leq \left(1 - \frac{\alpha\gamma\mu}{8}\right)^KR_0^2 + \frac{48\left(\alpha + 1\right)\gamma\sigma^2}{\mu b}.
	\end{equation*}
\end{theorem}

Similarly to \algname{S-SEG}, we also consider the diminishing stepsize policy \eqref{eq:stepsize_policy} for the \algname{I-SEG}.

\begin{corollary}\label{cor:str_mon_ISEG_main}
    Let Assumptions~\ref{as:lipschitzness},~\ref{as:str_monotonicity}~and~\ref{as:UBV_and_quadr_growth} hold. Assume that $\mu > 0$, $\gamma_{2,k} = \alpha \gamma_{1,k}$, $0 < \alpha \le \nicefrac{1}{4}$, $\gamma_{1,k} = \beta_k\gamma$, $0 < \beta_k \le 1$, where $\gamma$ equals the right-hand side of \eqref{eq:I_SEG_stepsize_condition}.
	Then, for all $K \ge 0$ and $\{\beta_k\}_{k\ge 0}$ satisfying \eqref{eq:stepsize_policy} with $\widetilde{\rho} = \nicefrac{\gamma\mu}{32}$, the iterates of \ref{eq:I_SEG} satisfy
	\begin{equation*}
	     \Exp\left[\|x^K - x^*\|^2\right] = \cO\left(\kappa R_0^2\exp\left(-\frac{K}{\kappa}\right) + \frac{\sigma^2}{\mu^2 b K}\right),
	\end{equation*}
	where $R_0 = \|x^0 - x^*\|^2$ and $\kappa = \max\left\{\frac{\delta}{\mu^2 b}, \frac{L +\sqrt{\nicefrac{\delta}{b}}}{\mu}\right\}$.
\end{corollary}

When $\delta = 0$ our rate recovers the best-known one for \algname{I-SEG} under uniformly bounded variance assumption \citep{beznosikov2020distributed}. Next, when $\delta > 0$ the slowest term in our rate evolves as $\cO(\nicefrac{1}{K})$, whereas the previous SOTA result for \algname{I-SEG} under Assumption~\ref{as:UBV_and_quadr_growth} depends on $K$ as $\cO(\nicefrac{1}{K^{\nicefrac{1}{3}}})$ \citep{hsieh2020explore}, which is much slower than $\cO(\nicefrac{1}{K})$. However, we emphasize that unlike our stepsize schedule the one from \citet{hsieh2020explore} is independent of $K$.

\begin{figure*}[t]
\centering
\begin{subfigure}[b]{0.24\linewidth}
    \centering
    \includegraphics[width=\textwidth]{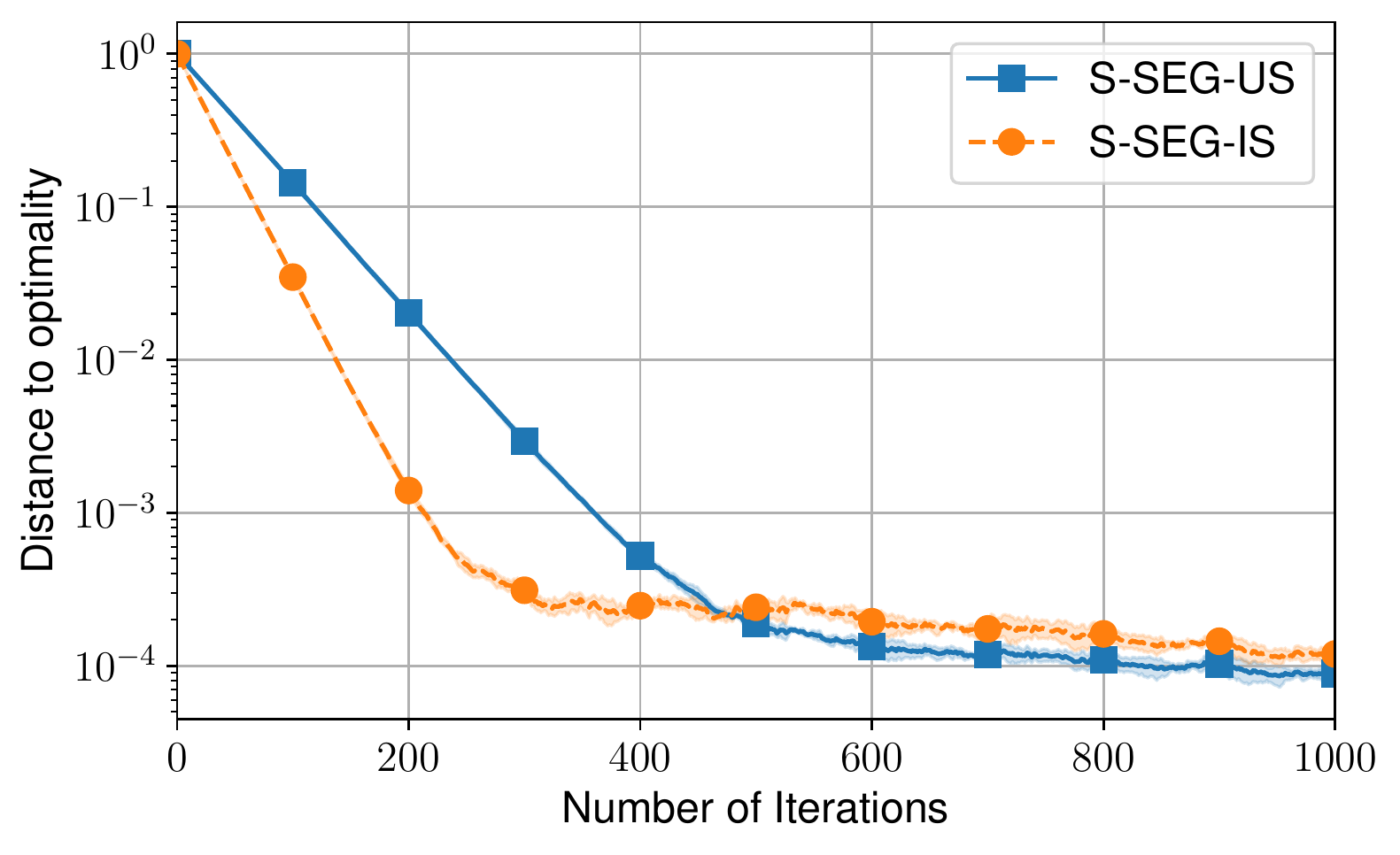}
    \caption{$L_{\max} = 2 $}
\end{subfigure}
\begin{subfigure}[b]{0.24\linewidth}
    \centering
    \includegraphics[width=\textwidth]{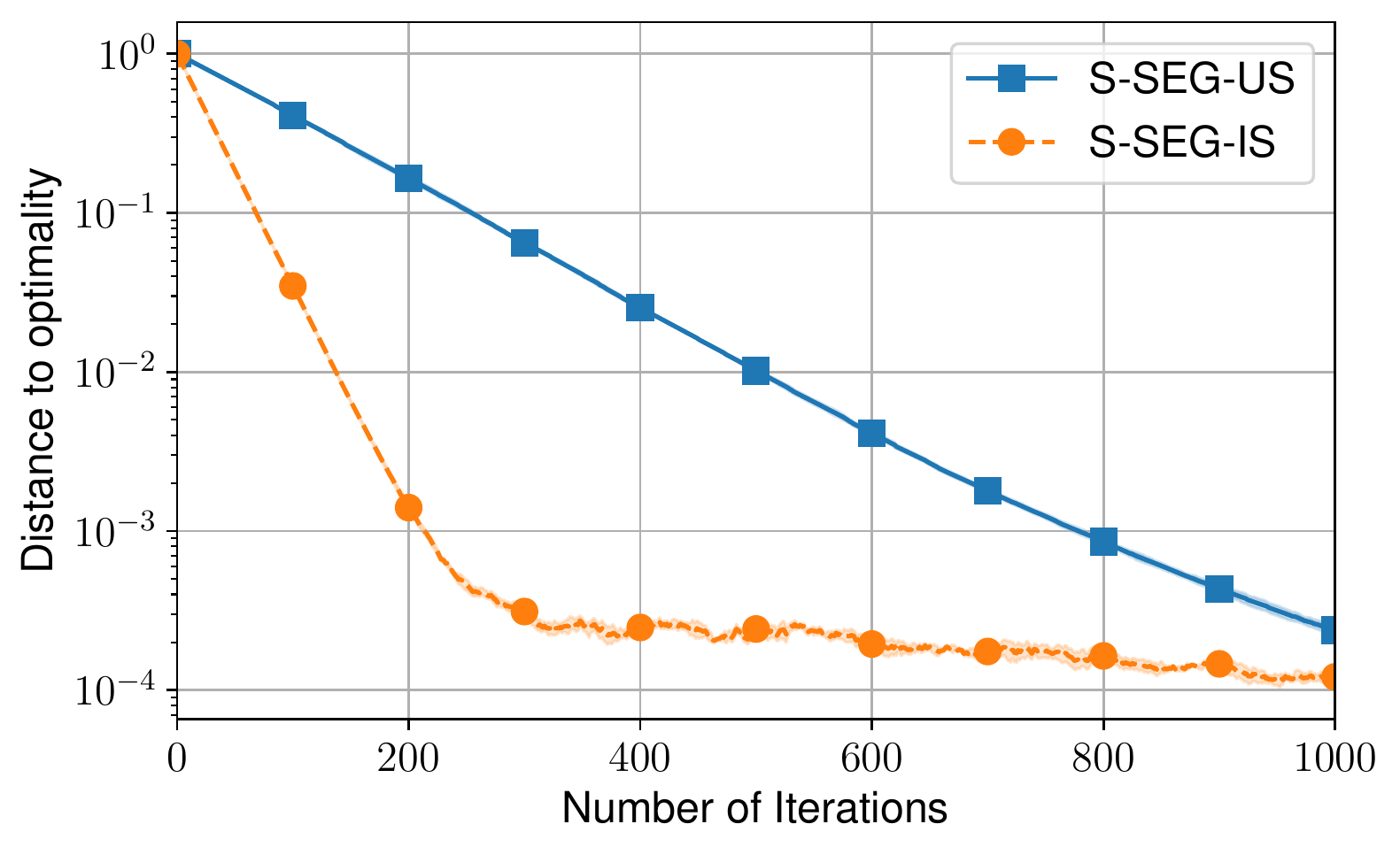}
    \caption{$L_{\max} = 5 $}
\end{subfigure}
\begin{subfigure}[b]{0.24\textwidth}
    \centering
    \includegraphics[width=\textwidth]{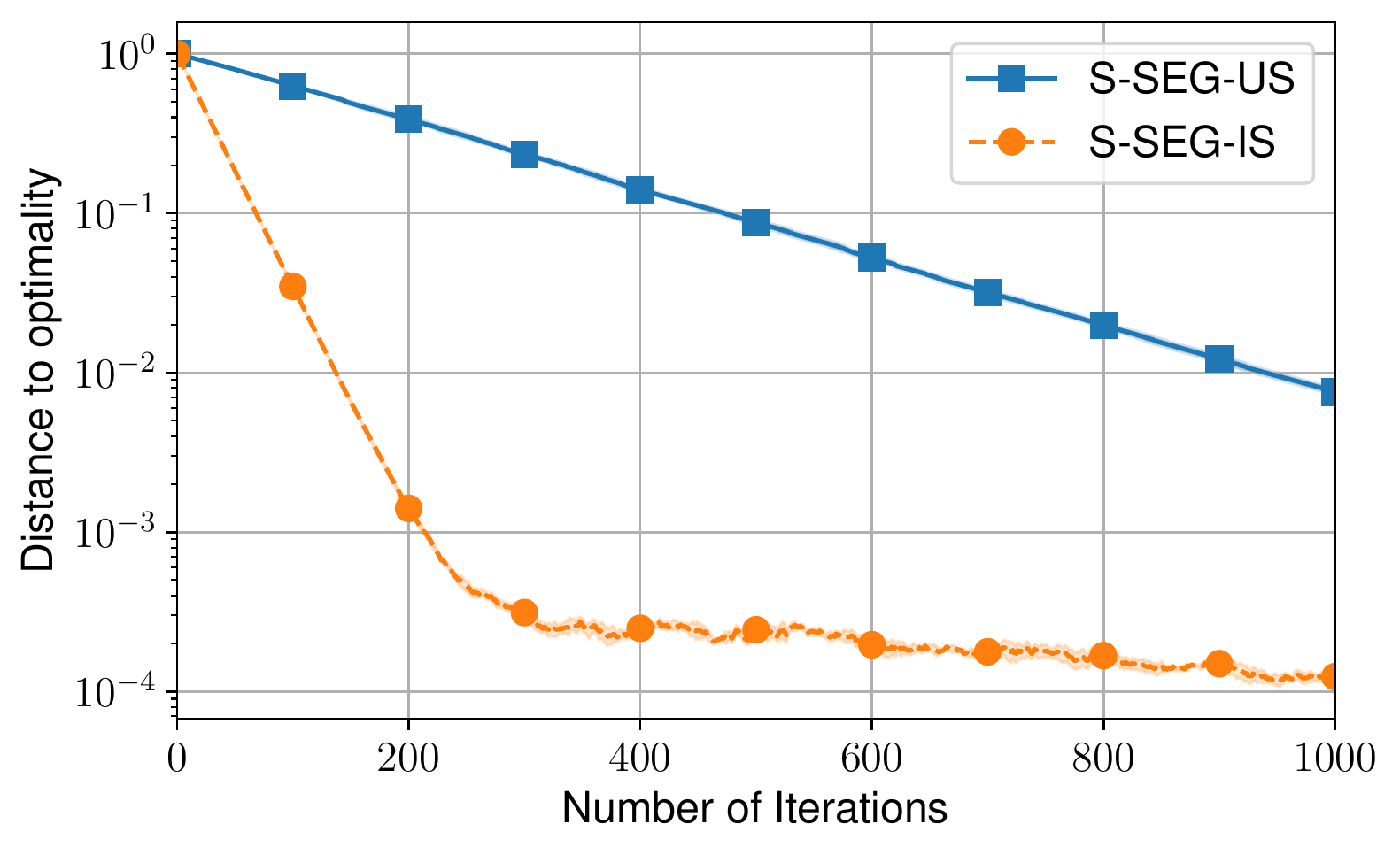}
    \caption{$L_{\max} = 10$}
\end{subfigure}
\begin{subfigure}[b]{0.24\textwidth}
    \centering
    \includegraphics[width=\textwidth]{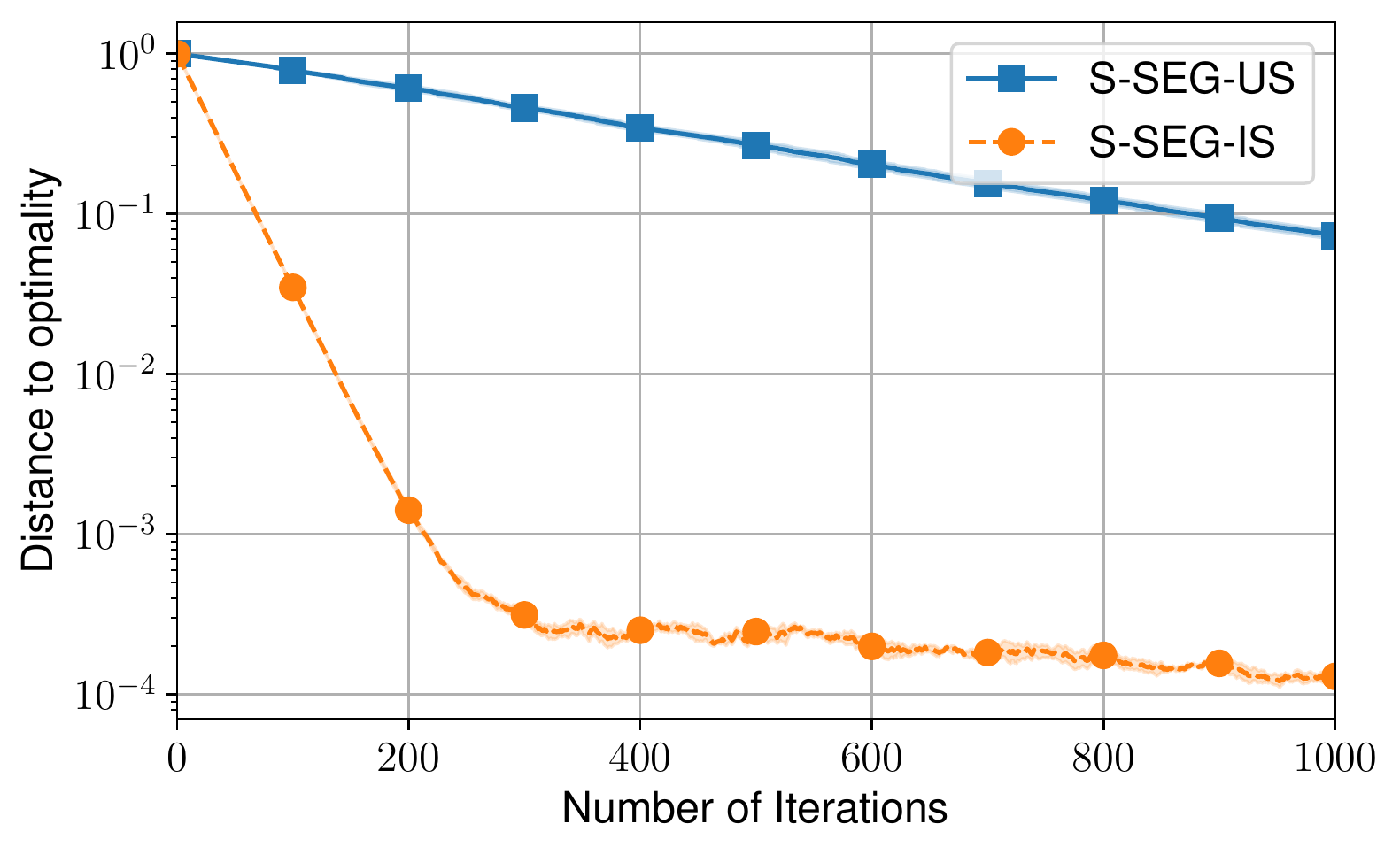}
    \caption{$L_{\max} = 20$}
\end{subfigure}
\vspace{-3mm}
\caption{Comparison of \algname{S-SEG-US} vs \algname{S-SEG-IS} for different values of $L_{\max}$. While the rate of convergence of \algname{S-SEG-US} becomes slower as $L_{\max}$ increases, the rate of convergence of \algname{S-SEG-IS} remains the same.}
\label{fig:us_vs_is}
\end{figure*}

\begin{figure*}[t]
\centering
\begin{subfigure}[b]{0.32\linewidth}
    \centering
    \includegraphics[width=\textwidth]{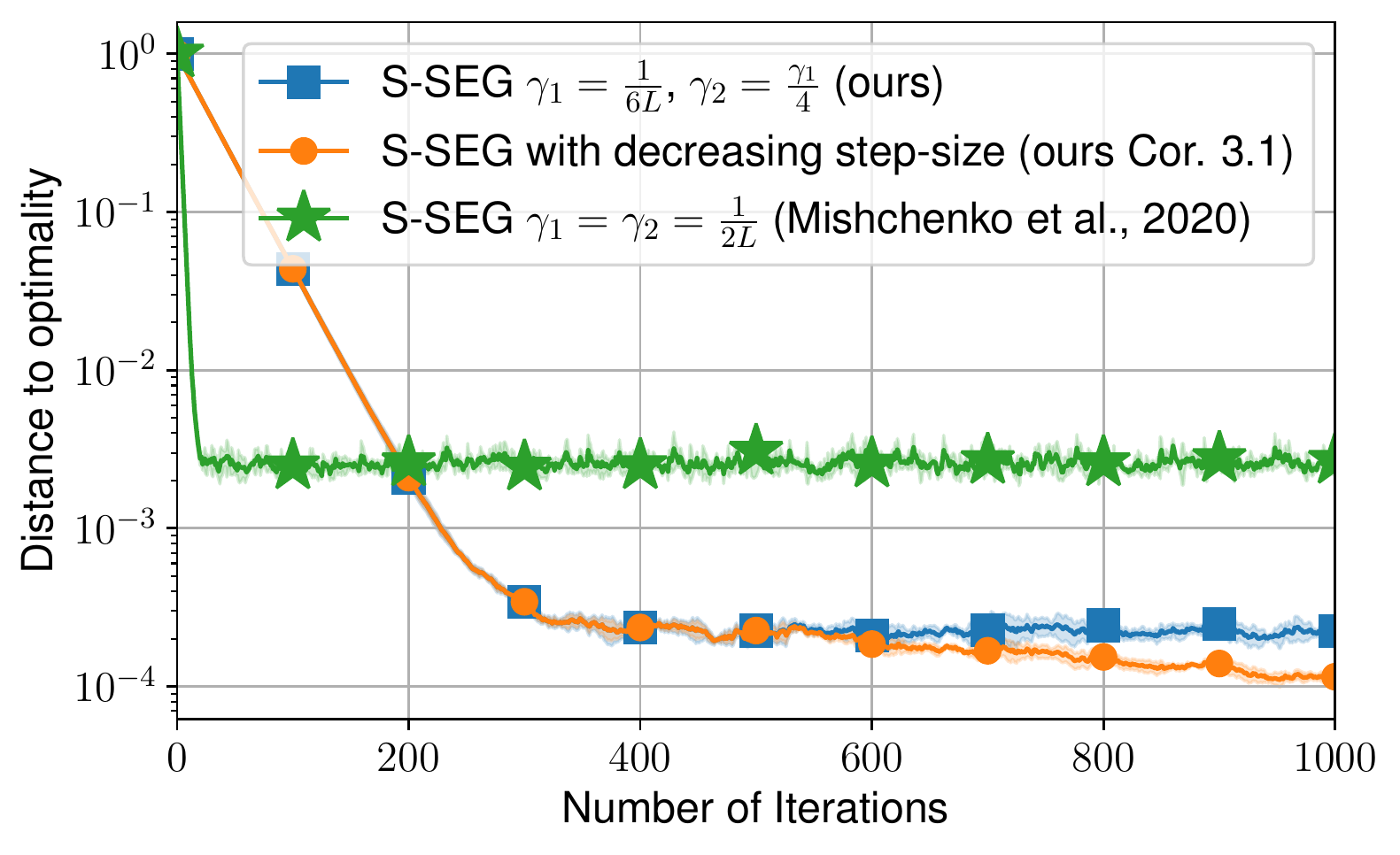}
    \caption{}
    \label{fig:sseg}
\end{subfigure}
\begin{subfigure}[b]{0.32\linewidth}
    \centering
    \includegraphics[width=\textwidth]{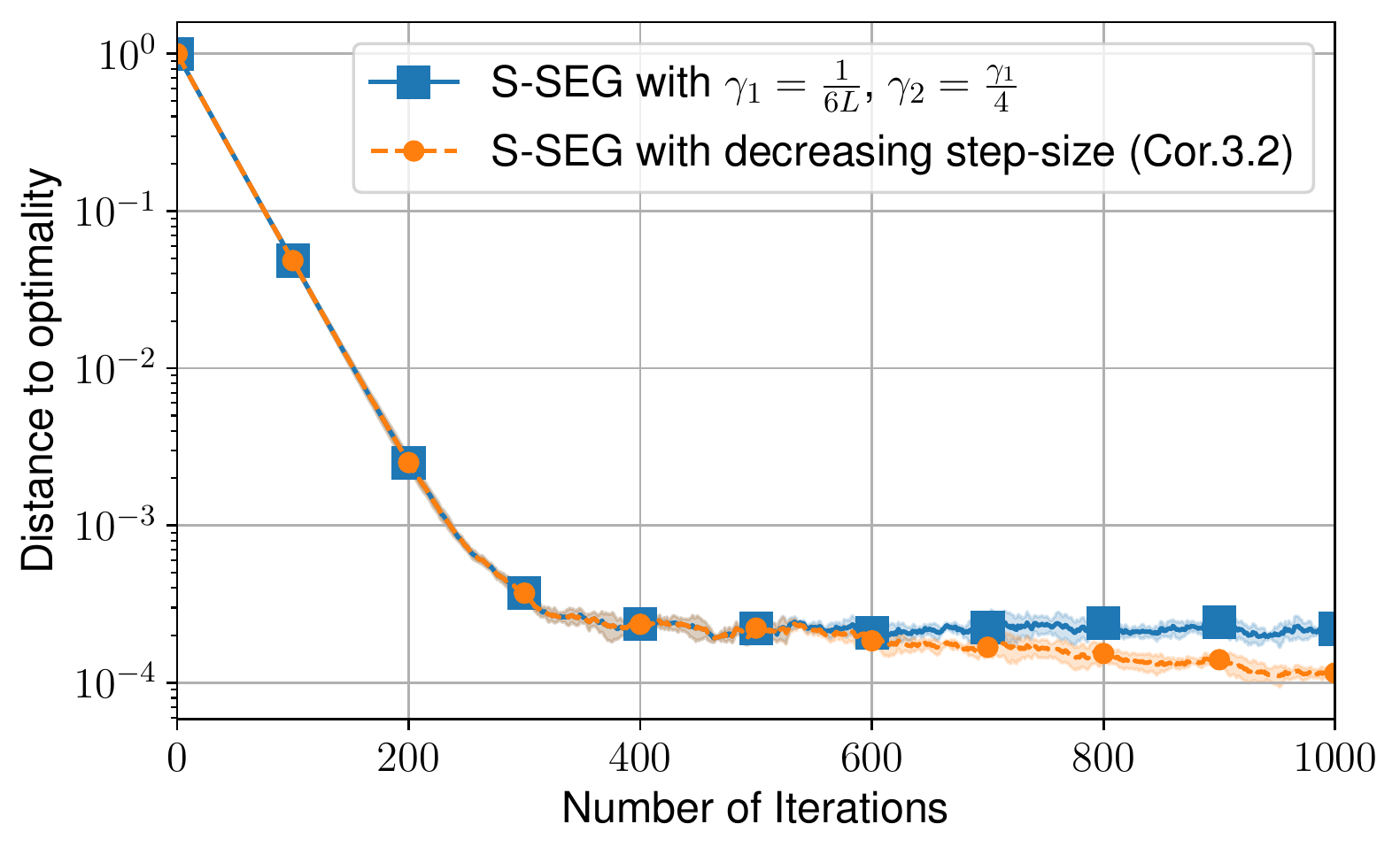}
    \caption{}
    \label{fig:negative_mu}
\end{subfigure}
\begin{subfigure}[b]{0.32\linewidth}
    \centering
    \includegraphics[width=\textwidth]{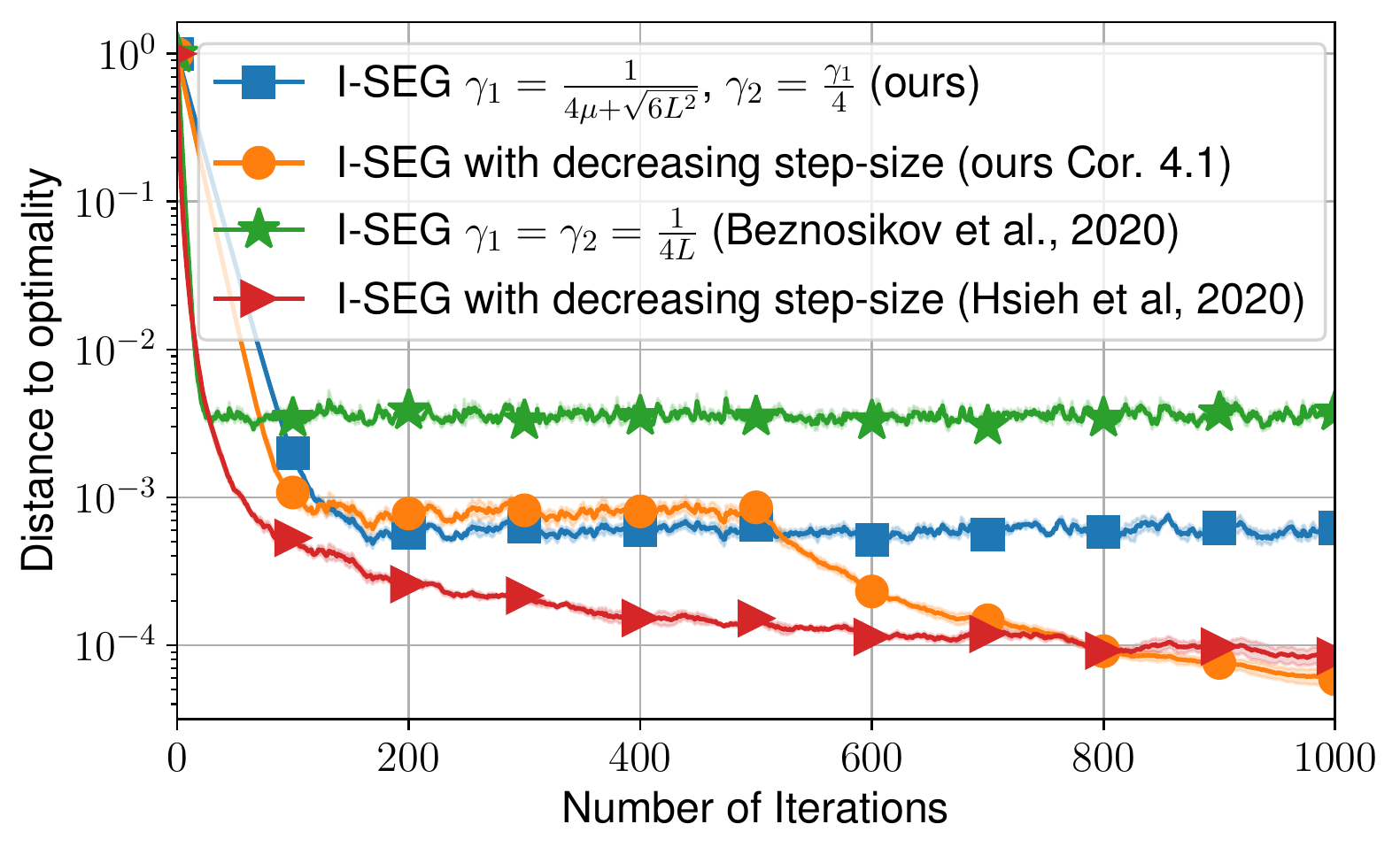}
    \caption{}
    \label{fig:iseg}
\end{subfigure}
\vspace{-3mm}
\caption{Experiments on quadratic games illustrating the theoretical results of the paper. (a) Comparison of different stepsize choices for \algname{S-SEG}. (b) Convergence of \algname{S-SEG} on quadratic games with negative $\mu_\xi$. (c) Comparison of different stepsize choices for \algname{I-SEG}.
}
\vspace{-0.4cm}
\end{figure*}

\section{NUMERICAL EXPERIMENTS}\label{sec:experiments}
To illustrate the theoretical results, we conduct experiments on quadratic games of the form:
\begin{align*}
\min _{x_{1} \in \mathbb{R}^{d}} \max _{x_{2} \in \mathbb{R}^{p}} \frac{1}{n} \sum_{i=1}^n &\frac{1}{2} x_{1}^{\top} \mathbf{A}_{i} x_{1} + x_{1}^{\top} \mathbf{B}_{i} x_{2}-\frac{1}{2} x_{2}^{\top} \mathbf{C}_{i} x_{2}\\
&\quad +a_{i}^{\top} x_{1}-c_{i}^{\top} x_{2}.
\end{align*}
By choosing the matrices such that $\mu_i \mathbf{I} \preceq {\mA}_{i} \preceq L_i \mathbf{I}$ and $\mu_i \mathbf{I} \preceq \mathbf{C}_{i} \preceq L_i \mathbf{I}$ we can ensure that the game satisfies the assumptions for our theory, i.e., the game is strongly monotone and smooth. In all the experiments, we report the average over 5 different runs. Further details about the experiments can be found in Appendix~\ref{app:extra_experiments}.

\textbf{Experiment 1: \algname{S-SEG-US} vs \algname{S-SEG-IS}.} To illustrate the advantages of importance sampling compared to uniform sampling, we construct quadratic games such that $L_1=L_{\max}$ and $L_i=1 \quad \forall i > 1 $. We show in Fig.~\ref{fig:us_vs_is} that while the rate of convergence of \algname{S-SEG-US} becomes slower as $L_{\max}$ increases, the rate of convergence of \algname{S-SEG-IS} remains almost the same, because $\overline{L}$ does not change significantly.

\textbf{Experiment 2: \algname{S-SEG} with different stepsizes.} We compare \algname{S-SEG} with different stepsize choices in Fig.~\ref{fig:sseg}. We compare the decreasing stepsize proposed in Corollary~\ref{cor:str_mon_AS_main} to the constant stepsize proposed in \cite{mishchenko2020revisiting} where $\gamma_1=\gamma_2 \leq \frac{1}{2L}$, and to the constant stepsize proposed in Theorem~\ref{thm:SEG_same_sample_convergence_AS_main}. \algname{S-SEG} with the proposed decreasing stepsize strategy converges faster to a smaller neighborhood of the solution compared to constant stepsize, see Fig.~\ref{fig:sseg}.

\textbf{Experiment 3: Convergence of \algname{S-SEG} when some $\mu_\xi$ are negative.} To illustrate the generality of Assumption~\ref{as:F_xi_str_monotonicity}, we construct a quadratic game where one of the $\mu_\xi$ is negative. We illustrate the generality of Theorem~\ref{thm:SEG_same_sample_convergence_AS_main} in Fig.~\ref{fig:negative_mu} by showing that \algname{S-SEG} converges to the solution in such games.

\textbf{Experiment 4: \algname{I-SEG} with different stepsizes.} In Fig.~\ref{fig:iseg} we  compare \algname{I-SEG} under different stepsize choices. In particular, we show how the decreasing stepsize strategy proposed in Corollary~\ref{cor:str_mon_ISEG_main} converges to a smaller neighborhood than existing stepsize choices and it has comparable performance to the stepsize rule proposed in \cite{hsieh2020explore}. However, let us note again that our theoretical rate is better than the one from \cite{hsieh2020explore} (see Table~\ref{tab:comparison_of_rates}).

\vspace{-0.2cm}
\section{CONCLUSION}\label{sec:conclusion}
\vspace{-0.2cm}
In this paper, we develop a novel theoretical framework that allows us to analyze several variants of \algname{SEG} in a unified manner. We provide new convergence analysis for well-known variants of \algname{SEG} and derive new variants (e.g., S-SEG-IS) that outperform previous SOTA results. However, several important questions remain still open, such as the analysis of \algname{SEG} for quasi-monotone problems $(\mu=0)$ with unbounded domains without using large batchsizes, the analysis of \algname{S-SEG} with arbitrary sampling, and the same stepsizes $\gamma_{1,\xi^k} = \gamma_{2,\xi^k}$, and the improvement of the dependence of $\overline{\mu}$ on negative $\mu_i$.

\subsubsection*{Acknowledgements}
This work was partially supported by a grant for research centers in the field of artificial intelligence, provided by the Analytical Center for the Government of the Russian Federation in accordance with the subsidy agreement (agreement identifier 000000D730321P5Q0002) and the agreement with the Moscow Institute of Physics and Technology dated November 1, 2021 No. 70-2021-00138. Part of this work was done while Nicolas Loizou was a postdoctoral research fellow at Mila,  Université de Montréal, supported by the  IVADO Postdoctoral Funding Program. Gauthier Gidel is supported by an IVADO grant. Part of this work was done while Eduard Gorbunov was an intern at Mila, Universit\'e de Montr\'eal under the supervision of Gauthier Gidel and Nicolas Loizou.

{\small\bibliography{refs}}

\begin{thebibliography}{}

\bibitem[Abernethy et~al., 2021]{abernethy2021last}
Abernethy, J., Lai, K.~A., and Wibisono, A. (2021).
\newblock Last-iterate convergence rates for min-max optimization: Convergence
  of hamiltonian gradient descent and consensus optimization.
\newblock In {\em Algorithmic Learning Theory}, pages 3--47. PMLR.

\bibitem[Beznosikov et~al., 2020]{beznosikov2020distributed}
Beznosikov, A., Samokhin, V., and Gasnikov, A. (2020).
\newblock Distributed saddle-point problems: Lower bounds, optimal algorithms
  and federated gans.
\newblock {\em arXiv preprint arXiv:2010.13112}.

\bibitem[Bose et~al., 2020]{bose2020adversarial}
Bose, J., Gidel, G., Berard, H., Cianflone, A., Vincent, P., Lacoste-Julien,
  S., and Hamilton, W. (2020).
\newblock Adversarial example games.
\newblock In Larochelle, H., Ranzato, M., Hadsell, R., Balcan, M.~F., and Lin,
  H., editors, {\em Advances in Neural Information Processing Systems},
  volume~33, pages 8921--8934. Curran Associates, Inc.

\bibitem[Chavdarova et~al., 2019]{chavdarova2019reducing}
Chavdarova, T., Gidel, G., Fleuret, F., and Lacoste-Julien, S. (2019).
\newblock Reducing noise in {G}{A}{N} training with variance reduced
  extragradient.
\newblock In Wallach, H., Larochelle, H., Beygelzimer, A., d\textquotesingle
  Alch\'{e}-Buc, F., Fox, E., and Garnett, R., editors, {\em Advances in Neural
  Information Processing Systems}, volume~32. Curran Associates, Inc.

\bibitem[Chen and Rockafellar, 1997]{chen1997convergence}
Chen, G.~H. and Rockafellar, R.~T. (1997).
\newblock Convergence rates in forward--backward splitting.
\newblock {\em SIAM Journal on Optimization}, 7(2):421--444.

\bibitem[Chen et~al., 2017]{chen2017accelerated}
Chen, Y., Lan, G., and Ouyang, Y. (2017).
\newblock Accelerated schemes for a class of variational inequalities.
\newblock {\em Mathematical Programming}, 165(1):113--149.

\bibitem[Daskalakis et~al., 2021]{daskalakis2021complexity}
Daskalakis, C., Skoulakis, S., and Zampetakis, M. (2021).
\newblock The complexity of constrained min-max optimization.
\newblock In {\em Proceedings of the 53rd Annual ACM SIGACT Symposium on Theory
  of Computing}, pages 1466--1478.

\bibitem[Diakonikolas et~al., 2021]{diakonikolas2021efficient}
Diakonikolas, J., Daskalakis, C., and Jordan, M. (2021).
\newblock Efficient methods for structured nonconvex-nonconcave min-max
  optimization.
\newblock In {\em International Conference on Artificial Intelligence and
  Statistics}, pages 2746--2754. PMLR.

\bibitem[Dikkala et~al., 2020]{dikkala2020minimax}
Dikkala, N., Lewis, G., Mackey, L., and Syrgkanis, V. (2020).
\newblock Minimax estimation of conditional moment models.
\newblock In Larochelle, H., Ranzato, M., Hadsell, R., Balcan, M.~F., and Lin,
  H., editors, {\em Advances in Neural Information Processing Systems},
  volume~33, pages 12248--12262. Curran Associates, Inc.

\bibitem[Gidel et~al., 2019]{gidel2018variational}
Gidel, G., Berard, H., Vignoud, G., Vincent, P., and Lacoste-Julien, S. (2019).
\newblock A variational inequality perspective on generative adversarial
  networks.
\newblock In {\em International Conference on Learning Representations (ICLR)}.

\bibitem[Golowich et~al., 2020]{golowich2020last}
Golowich, N., Pattathil, S., Daskalakis, C., and Ozdaglar, A. (2020).
\newblock Last iterate is slower than averaged iterate in smooth convex-concave
  saddle point problems.
\newblock In {\em Conference on Learning Theory}, pages 1758--1784. PMLR.

\bibitem[Goodfellow et~al., 2014]{goodfellow2014generative}
Goodfellow, I., Pouget-Abadie, J., Mirza, M., Xu, B., Warde-Farley, D., Ozair,
  S., Courville, A., and Bengio, Y. (2014).
\newblock Generative adversarial nets.
\newblock In Ghahramani, Z., Welling, M., Cortes, C., Lawrence, N., and
  Weinberger, K.~Q., editors, {\em Advances in Neural Information Processing
  Systems}, volume~27. Curran Associates, Inc.

\bibitem[Gorbunov et~al., 2020]{gorbunov2020unified}
Gorbunov, E., Hanzely, F., and Richtarik, P. (2020).
\newblock A {U}nified {T}heory of {SGD}: {V}ariance {R}eduction, {S}ampling,
  {Q}uantization and {C}oordinate {D}escent.
\newblock In Chiappa, S. and Calandra, R., editors, {\em Proceedings of the
  Twenty Third International Conference on Artificial Intelligence and
  Statistics}, volume 108 of {\em Proceedings of Machine Learning Research},
  pages 680--690. PMLR.

\bibitem[Gorbunov et~al., 2021]{gorbunov2021extragradient}
Gorbunov, E., Loizou, N., and Gidel, G. (2021).
\newblock Extragradient method: {O}$(1/{K}) $ last-iterate convergence for
  monotone variational inequalities and connections with cocoercivity.
\newblock {\em arXiv preprint arXiv:2110.04261}.

\bibitem[Gower et~al., 2021]{gower2021sgd}
Gower, R., Sebbouh, O., and Loizou, N. (2021).
\newblock Sgd for structured nonconvex functions: Learning rates, minibatching
  and interpolation.
\newblock In {\em International Conference on Artificial Intelligence and
  Statistics}, pages 1315--1323. PMLR.

\bibitem[Gower et~al., 2019]{gower2019sgd}
Gower, R.~M., Loizou, N., Qian, X., Sailanbayev, A., Shulgin, E., and
  Richt{\'a}rik, P. (2019).
\newblock {SGD}: {G}eneral {A}nalysis and {I}mproved {R}ates.
\newblock In {\em Proceedings of the 36th International Conference on Machine
  Learning}, volume~97 of {\em Proceedings of Machine Learning Research}, pages
  5200--5209.

\bibitem[Hanzely and Richt{\'a}rik, 2019]{hanzely2019accelerated}
Hanzely, F. and Richt{\'a}rik, P. (2019).
\newblock Accelerated coordinate descent with arbitrary sampling and best rates
  for minibatches.
\newblock In {\em The 22nd International Conference on Artificial Intelligence
  and Statistics}, pages 304--312. PMLR.

\bibitem[Horv{\'a}th and Richt{\'a}rik, 2019]{horvath2019nonconvex}
Horv{\'a}th, S. and Richt{\'a}rik, P. (2019).
\newblock Nonconvex variance reduced optimization with arbitrary sampling.
\newblock In {\em International Conference on Machine Learning}, pages
  2781--2789. PMLR.

\bibitem[Hsieh et~al., 2019]{hsieh2019convergence}
Hsieh, Y.-G., Iutzeler, F., Malick, J., and Mertikopoulos, P. (2019).
\newblock On the convergence of single-call stochastic extra-gradient methods.
\newblock In Wallach, H., Larochelle, H., Beygelzimer, A., d\textquotesingle
  Alch\'{e}-Buc, F., Fox, E., and Garnett, R., editors, {\em Advances in Neural
  Information Processing Systems}, volume~32. Curran Associates, Inc.

\bibitem[Hsieh et~al., 2020]{hsieh2020explore}
Hsieh, Y.-G., Iutzeler, F., Malick, J., and Mertikopoulos, P. (2020).
\newblock Explore aggressively, update conservatively: Stochastic extragradient
  methods with variable stepsize scaling.
\newblock {\em Advances in Neural Information Processing Systems}, 33.

\bibitem[Jelassi et~al., 2020]{jelassi2020extragradient}
Jelassi, S., Domingo-Enrich, C., Scieur, D., Mensch, A., and Bruna, J. (2020).
\newblock Extra-gradient with player sampling for faster convergence in
  n-player games.
\newblock In III, H.~D. and Singh, A., editors, {\em Proceedings of the 37th
  International Conference on Machine Learning}, volume 119 of {\em Proceedings
  of Machine Learning Research}, pages 4736--4745. PMLR.

\bibitem[Juditsky et~al., 2011]{juditsky2011solving}
Juditsky, A., Nemirovski, A., and Tauvel, C. (2011).
\newblock Solving variational inequalities with stochastic mirror-prox
  algorithm.
\newblock {\em Stochastic Systems}, 1(1):17--58.

\bibitem[Kannan and Shanbhag, 2019]{kannan2019optimal}
Kannan, A. and Shanbhag, U.~V. (2019).
\newblock Optimal stochastic extragradient schemes for pseudomonotone
  stochastic variational inequality problems and their variants.
\newblock {\em Computational Optimization and Applications}, 74(3):779--820.

\bibitem[Khaled and Richt{\'a}rik, 2020]{khaled2020better}
Khaled, A. and Richt{\'a}rik, P. (2020).
\newblock Better theory for sgd in the nonconvex world.
\newblock {\em arXiv preprint arXiv:2002.03329}.

\bibitem[Khaled et~al., 2020]{khaled2020unified}
Khaled, A., Sebbouh, O., Loizou, N., Gower, R.~M., and Richtárik, P. (2020).
\newblock Unified analysis of stochastic gradient methods for composite convex
  and smooth optimization.
\newblock {\em arXiv preprint arXiv:2006.11573}.

\bibitem[Korpelevich, 1976]{korpelevich1976extragradient}
Korpelevich, G.~M. (1976).
\newblock The extragradient method for finding saddle points and other
  problems.
\newblock {\em Matecon}, 12:747--756.

\bibitem[Li et~al., 2021]{li2021convergence}
Li, C.~J., Yu, Y., Loizou, N., Gidel, G., Ma, Y., Roux, N.~L., and Jordan,
  M.~I. (2021).
\newblock On the convergence of stochastic extragradient for bilinear games
  with restarted iteration averaging.
\newblock {\em arXiv preprint arXiv:2107.00464}.

\bibitem[Liu and Wright, 2016]{liu2016accelerated}
Liu, J. and Wright, S. (2016).
\newblock An accelerated randomized {Kaczmarz} algorithm.
\newblock {\em Mathematics of Computation}, 85(297):153--178.

\bibitem[Loizou et~al., 2021]{loizou2021stochastic}
Loizou, N., Berard, H., Gidel, G., Mitliagkas, I., and Lacoste-Julien, S.
  (2021).
\newblock Stochastic gradient descent-ascent and consensus optimization for
  smooth games: Convergence analysis under expected co-coercivity.
\newblock {\em arXiv preprint arXiv:2107.00052}.

\bibitem[Loizou et~al., 2020]{loizou2020stochastic}
Loizou, N., Berard, H., Jolicoeur-Martineau, A., Vincent, P., Lacoste-Julien,
  S., and Mitliagkas, I. (2020).
\newblock Stochastic hamiltonian gradient methods for smooth games.
\newblock In {\em International Conference on Machine Learning}, pages
  6370--6381. PMLR.

\bibitem[Loizou and Richt{\'a}rik, 2020a]{loizou2020convergence}
Loizou, N. and Richt{\'a}rik, P. (2020a).
\newblock Convergence analysis of inexact randomized iterative methods.
\newblock {\em SIAM Journal on Scientific Computing}, 42(6):A3979--A4016.

\bibitem[Loizou and Richt{\'a}rik, 2020b]{loizou2020momentum}
Loizou, N. and Richt{\'a}rik, P. (2020b).
\newblock Momentum and stochastic momentum for stochastic gradient, newton,
  proximal point and subspace descent methods.
\newblock {\em Computational Optimization and Applications}, 77(3):653--710.

\bibitem[Loizou and Richt{\'a}rik, 2021]{loizou2021revisiting}
Loizou, N. and Richt{\'a}rik, P. (2021).
\newblock Revisiting randomized gossip algorithms: General framework,
  convergence rates and novel block and accelerated protocols.
\newblock {\em IEEE Transactions on Information Theory}.

\bibitem[Martinet, 1970]{martinet1970regularisation}
Martinet, B. (1970).
\newblock Regularisation d'inequations variationelles par approximations
  successives.
\newblock {\em Revue Francaise d'Informatique et de Recherche Operationelle},
  4:154--159.

\bibitem[Mertikopoulos and Zhou, 2019]{mertikopoulos2019learning}
Mertikopoulos, P. and Zhou, Z. (2019).
\newblock Learning in games with continuous action sets and unknown payoff
  functions.
\newblock {\em Mathematical Programming}, 173(1):465--507.

\bibitem[Mishchenko et~al., 2020]{mishchenko2020revisiting}
Mishchenko, K., Kovalev, D., Shulgin, E., Richtarik, P., and Malitsky, Y.
  (2020).
\newblock Revisiting stochastic extragradient.
\newblock In Chiappa, S. and Calandra, R., editors, {\em Proceedings of the
  Twenty Third International Conference on Artificial Intelligence and
  Statistics}, volume 108 of {\em Proceedings of Machine Learning Research},
  pages 4573--4582. PMLR.

\bibitem[Necoara et~al.,
  2019]{Necoara-Nesterov-Glineur-2018-linear-without-strong-convexity}
Necoara, I., Nesterov, Y., and Glineur, F. (2019).
\newblock Linear convergence of first order methods for non-strongly convex
  optimization.
\newblock {\em Mathematical Programming}, 175(1):69--107.

\bibitem[Nemirovski et~al., 2009]{Nemirovski-Juditsky-Lan-Shapiro-2009}
Nemirovski, A., Juditsky, A., Lan, G., and Shapiro, A. (2009).
\newblock Robust stochastic approximation approach to stochastic programming.
\newblock {\em SIAM Journal on Optimization}, 19(4):1574--1609.

\bibitem[Nesterov, 2007]{nesterov2007dual}
Nesterov, Y. (2007).
\newblock Dual extrapolation and its applications to solving variational
  inequalities and related problems.
\newblock {\em Mathematical Programming}, 109(2):319--344.

\bibitem[Noor, 2003]{noor2003new}
Noor, M.~A. (2003).
\newblock New extragradient-type methods for general variational inequalities.
\newblock {\em Journal of Mathematical Analysis and Applications},
  277(2):379--394.

\bibitem[Pfau and Vinyals, 2016]{pfau2016connecting}
Pfau, D. and Vinyals, O. (2016).
\newblock Connecting generative adversarial networks and actor-critic methods.
\newblock {\em arXiv preprint arXiv:1610.01945}.

\bibitem[Popov, 1980]{popov1980modification}
Popov, L.~D. (1980).
\newblock A modification of the arrow-hurwicz method for search of saddle
  points.
\newblock {\em Mathematical notes of the Academy of Sciences of the USSR},
  28(5):845--848.

\bibitem[Qu and Richt{\'a}rik, 2016]{qu2016coordinate}
Qu, Z. and Richt{\'a}rik, P. (2016).
\newblock Coordinate descent with arbitrary sampling i: Algorithms and
  complexity.
\newblock {\em Optimization Methods and Software}, 31(5):829--857.

\bibitem[Richt{\'a}rik and Tak{\'a}{\v{c}}, 2016]{richtarik2016optimal}
Richt{\'a}rik, P. and Tak{\'a}{\v{c}}, M. (2016).
\newblock On optimal probabilities in stochastic coordinate descent methods.
\newblock {\em Optimization Letters}, 10(6):1233--1243.

\bibitem[Richt{\'a}rik and Tak{\'a}c, 2020]{richtarik2020stochastic}
Richt{\'a}rik, P. and Tak{\'a}c, M. (2020).
\newblock Stochastic reformulations of linear systems: algorithms and
  convergence theory.
\newblock {\em SIAM Journal on Matrix Analysis and Applications},
  41(2):487--524.

\bibitem[Rockafellar, 1976]{rockafellar1976monotone}
Rockafellar, R.~T. (1976).
\newblock Monotone operators and the proximal point algorithm.
\newblock {\em SIAM journal on control and optimization}, 14(5):877--898.

\bibitem[Song et~al., 2020]{song2020optimistic}
Song, C., Zhou, Z., Zhou, Y., Jiang, Y., and Ma, Y. (2020).
\newblock Optimistic dual extrapolation for coherent non-monotone variational
  inequalities.
\newblock In Larochelle, H., Ranzato, M., Hadsell, R., Balcan, M.~F., and Lin,
  H., editors, {\em Advances in Neural Information Processing Systems},
  volume~33, pages 14303--14314. Curran Associates, Inc.

\bibitem[Stich, 2019]{stich2019unified}
Stich, S.~U. (2019).
\newblock Unified optimal analysis of the (stochastic) gradient method.
\newblock {\em arXiv preprint arXiv:1907.04232}.

\bibitem[Vezhnevets et~al., 2017]{vezhnevets2017feudal}
Vezhnevets, A.~S., Osindero, S., Schaul, T., Heess, N., Jaderberg, M., Silver,
  D., and Kavukcuoglu, K. (2017).
\newblock Feudal networks for hierarchical reinforcement learning.
\newblock In {\em International Conference on Machine Learning}, pages
  3540--3549. PMLR.

\bibitem[Wayne and Abbott, 2014]{wayne2014hierarchical}
Wayne, G. and Abbott, L. (2014).
\newblock Hierarchical control using networks trained with higher-level forward
  models.
\newblock {\em Neural computation}, 26(10):2163--2193.

\bibitem[Yang et~al., 2020]{yang2020global}
Yang, J., Kiyavash, N., and He, N. (2020).
\newblock Global convergence and variance reduction for a class of
  nonconvex-nonconcave minimax problems.
\newblock In Larochelle, H., Ranzato, M., Hadsell, R., Balcan, M.~F., and Lin,
  H., editors, {\em Advances in Neural Information Processing Systems},
  volume~33, pages 1153--1165. Curran Associates, Inc.

\end{thebibliography}


\clearpage
\appendix

\thispagestyle{empty}

\onecolumn \makesupplementtitle

{\small \tableofcontents}

\section{ON EXPERIMENTS}
\label{app:extra_experiments}

\subsection{Experimental Details} We describe here in more details the exact settings we use for evaluating the different algorithms. 
As mentioned in Section~\ref{sec:experiments}, we evaluate the different algorithms on the class of quadratic games:
\begin{equation*}
  \min_{x_1 \in \mathbb{R}^{d}} \max_{x_2 \in \mathbb{R}^{p}} \frac{1}{n}\sum_i \frac{1}{2} x_1^\top \mA_i x_1 + x_1^\top \mB_i x_2 - \frac{1}{2} x_2^\top \mC_i x_2 + a_i^\top x_1 - c_i^\top x_2
\end{equation*}

In all our experiments, we choose $d = p = 100$ and $n=100$. To sample the matrices $\mA_i$ (resp. $\mC_i$) we first generate a random orthogonal matrix $\mQ_i$ (resp. $\mQ'_i$), we then sample a random diagonal matrix $\mD_i$ (resp. $\mD'_i$) where the elements on the diagonal are sampled uniformly in $[\mu_A, L_A]$ (resp. $[\mu_C, L_C]$), such that at least one of the matrices has a minimum eigenvalue equal to $\mu_A$ (resp. $\mu_C$) and one matrix has a maximum eigenvalue equal to $L_A$ (resp. $L_B$).  Finally we construct the matrices by computing $\mA_i=\mQ_i\mD_i\mQ_i^\top$ (resp. $\mC_i = \mQ'_i\mD'_i{\mQ'}_i^\top$). This ensures that the matrices $\mA_i$ and $\mC_i$ for all $i \in [n]$, are symmetric and positive definite. We sample the matrices $\mB_i$ in a similar fashion with the diagonal matrix $\mD_i$ to lie between $[\mu_B, L_B]$\footnote{We highlight that matrices $\mB_i$ are not necessarily symmetric.}. The bias terms $a_i, c_i$ are sampled from a normal distribution. 
In all our experiments we choose $\mu_A=\mu_C=0.1$, $L_A=L_C=1$, $\mu_B=0$ and $L_B=1$ unless stated otherwise. For further details please refer to the code: \url{https://github.com/hugobb/Stochastic-Extragradient}.

\subsection{Additional Experiment: \algname{S-SEG} with $b$-Nice Sampling (\algname{S-SEG-NICE})} To illustrate Remark~\ref{rem:b_nice} about the advantages of \algname{S-SEG-NICE} compared to \algname{S-SEG-US} with i.i.d.\ batching, we construct a quadratic game such that $L_1 = L_{\max}$ and $L_i = 1, \quad \forall i > 1$. We use the constant stepsize specified in Section~\ref{sec:sseg_convergence}. We show in Fig.~\ref{fig:sseg_nice} that the rate of convergence of \algname{S-SEG-NICE} is faster than \algname{S-SEG-US} with i.i.d.\ batching when using the same batch size. However \algname{S-SEG-NICE} converges to a slightly larger neighborhood of the solution. 

\begin{figure}[H]
    \centering
    \includegraphics[width=0.6\textwidth]{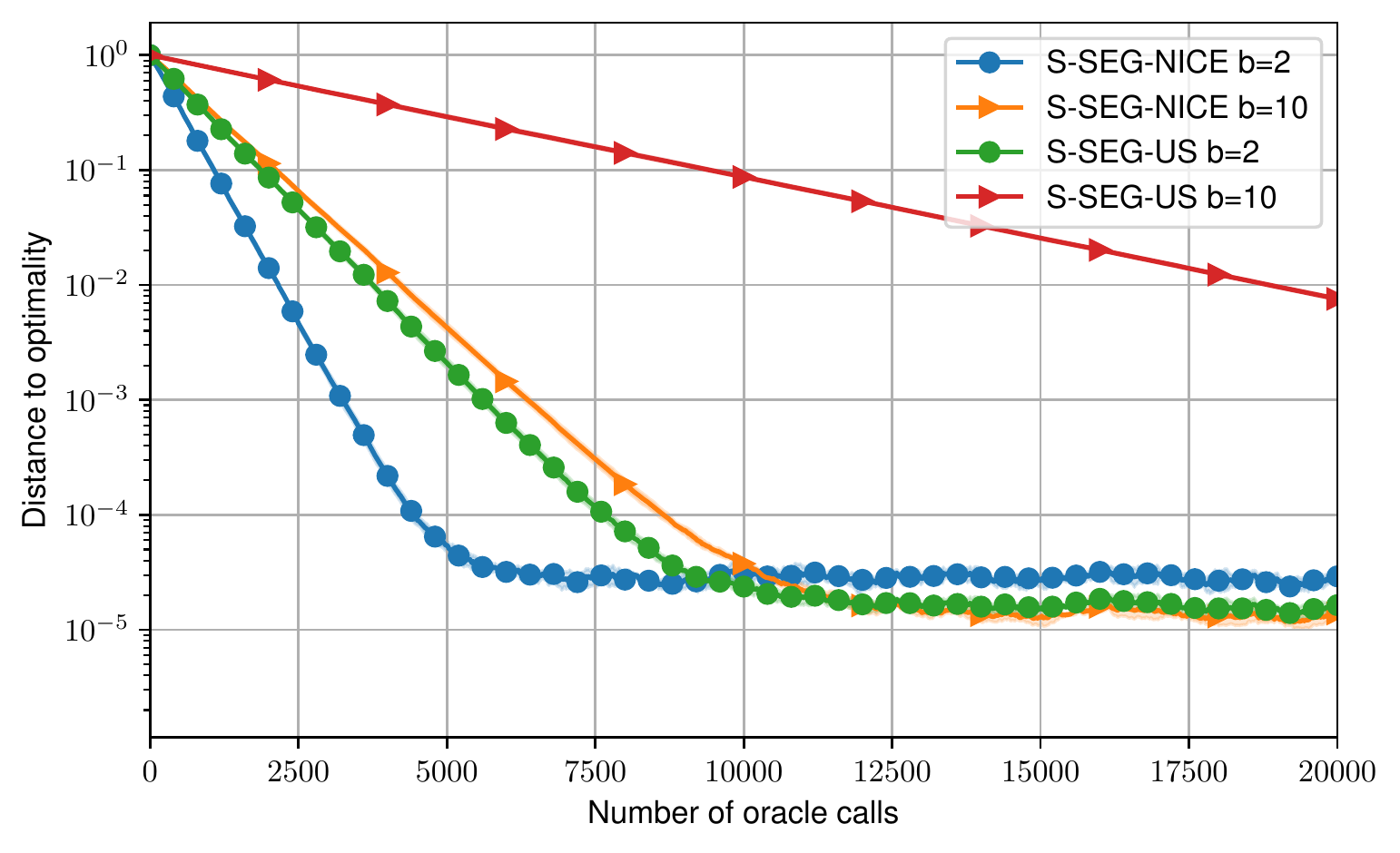}
    \caption{Convergence of \algname{S-SEG-NICE} for different batchsizes. In this experiment $L_{\max}=10$.}
    \label{fig:sseg_nice}
\end{figure}

\newpage

\section{DISCUSSION OF THE RESULTS UNDER QUASI MONOTONICITY}\label{sec:discussion_quasi_mon}

\begin{table}[H]
    \centering
    \footnotesize
    \caption{\small Summary of the state-of-the-art results for \algname{S-SEG} and \algname{I-SEG} for \textbf{quasi monotone VIPs}, i.e., for \algname{S-SEG} it means that $\overline{\mu} = \frac{1}{n}\sum_{i\in[n]:\mu_i \ge 0} \mu_i + \frac{4}{n}\sum_{i\in[n]:\mu_i < 0} \mu_i = 0$ and for \algname{I-SEG} -- $\mu=0$. Moreover, for \algname{I-SEG} we assume that $\delta = 0$ (see Assumption~\ref{as:UBV_and_quadr_growth}). Our results are highlighted in green. Columns: ``Norm?'' indicates whether the rate is given for the expected squared norm of the operator, ``Gap?'' indicates whether the rate is given for the expected gap function $\EE\left[\text{Gap}_{\cC} (z)\right] = \EE\left[\max_{u \in \mathcal{C}}  \langle F(u),  z - u  \rangle\right]$ (here $\cC$ is a compact set containing the solution set), ``Unbounded Set?'' indicates whether the analysis works for the case of unbounded sets, ``$b = \cO(1)$?'' indicates whether the analysis works with the batchsize independent of the target accuracy of the solution.}
    \label{tab:comparison_of_rates_monotone}
    \begin{threeparttable}
        \begin{tabular}{|c|c|c c c c c|}
        \hline
        Setup & Method & Citation & Norm? & Gap? & Unbounded Set? & $b = \cO(1)$?\\
        \hline\hline
        \multirow{3}{3cm}{\begin{tabular}{c}
            $F(x) = \frac{1}{n}\sum\limits_{i=1}^n F_i(x)$\\
            + As.~\ref{as:F_xi_lip},~\ref{as:F_xi_str_monotonicity}
        \end{tabular}} & \multirow{2}{*}{\centering\algname{S-SEG-US}} & \citep{mishchenko2020revisiting}\tnote{{\color{blue}(1)}} & \xmark & \cmark\tnote{{\color{blue}(2)}} & \cmark & \cmark\tnote{{\color{blue}(3)}}\\
        & & \cellcolor{bgcolor2} This paper & \cellcolor{bgcolor2} \cmark & \cellcolor{bgcolor2} \xmark & \cellcolor{bgcolor2} \cmark & \cellcolor{bgcolor2} \xmark\cmark\tnote{{\color{blue}(4)}}\\
        \cline{2-7}
        & \cellcolor{bgcolor2} \algname{S-SEG-IS} &\cellcolor{bgcolor2} This paper & \cellcolor{bgcolor2} \cmark & \cellcolor{bgcolor2} \xmark & \cellcolor{bgcolor2} \cmark & \cellcolor{bgcolor2} \xmark\cmark\tnote{{\color{blue}(4)}}\\
        \hline\hline
        \multirow{2}{3cm}{\begin{tabular}{c}
            $F(x) = \Exp_{\xi}[F_\xi(x)]$\\
            + As.~\ref{as:lipschitzness},~\ref{as:str_monotonicity},~\ref{as:UBV_and_quadr_growth} 
        \end{tabular}} & \multirow{2}{*}{\centering\algname{I-SEG}} & \citep{beznosikov2020distributed}\tnote{{\color{blue}(5)}} & \xmark & \cmark & \xmark & \cmark \\
        & &\cellcolor{bgcolor2} This paper &\cellcolor{bgcolor2} \cmark & \cellcolor{bgcolor2} \xmark & \cellcolor{bgcolor2} \cmark & \cellcolor{bgcolor2} \xmark\\
        \hline
    \end{tabular}
    \begin{tablenotes}
        {\scriptsize\item [{\color{blue}(1)}] \citet{mishchenko2020revisiting} consider a regularized version of \eqref{eq:main_problem} with convex regularization, $F(x) = \Exp_{\xi}[F_\xi(x)]$ and $F_{\xi}(x)$ being monotone and $L_\xi$-Lipschitz. If regularization is zero in the obtained problem and $\Exp_{\xi}[F_\xi(x)] = \frac{1}{n}\sum_{i=1}^n F_i(x)$, the problem from \cite{mishchenko2020revisiting} fits the considered setup with $\mu_i = 0$ for all $i\in[n]$.
        \item [{\color{blue}(2)}] The rate is derived for $\max_{u \in \cC}\Exp\left[\langle F(u), \hat{x}^K - u \rangle + R(\hat x^K) - R(u)\right]$, where $R(x)$ is the regularization term (in our settings, $R(x) \equiv 0$) and $\hat{x}^K$ is the average of the iterates produced by the method. This guarantee is weaker than the one for $\EE\left[\text{Gap}_{\cC} (\hat{x}^K)\right]$.
        \item [{\color{blue}(3)}] \citet{mishchenko2020revisiting} use uniformly bounded variance assumption on a compact set that defines the gap function (Assumption~\ref{as:UBV_and_quadr_growth} with $\delta=0$ on a compact).
        \item [{\color{blue}(4)}] In general, our results in this case require using batchsize dependent on the target accuracy. However, when $F_i(x^*) = 0$ for all $i \in [n]$, i.e., when interpolation conditions are satisfied, batchsizes can be chosen arbitrarily, e.g., $b = 1$, to achieve the convergence to any predefined accuracy.
        \item [{\color{blue}(5)}] \citet{beznosikov2020distributed} study a distributed version of \algname{I-SEG}.
        }
    \end{tablenotes}
    \end{threeparttable}
\end{table}

\paragraph{Results under (quasi) monotonicity.} The state-of-the-art results for the convergence of \algname{S-SEG} and \algname{I-SEG} for (quasi) monotone VIP are summarized in Table~\ref{tab:comparison_of_rates_monotone}. For \algname{S-SEG} by quasi-monotonicity we mean that Assumption~\ref{as:F_xi_str_monotonicity} holds and $\Exp_{\xi^k}[\gamma_{1,\xi^k}\mu_{\xi^k}(\obf_{\{\mu_{\xi^k} \ge 0\}} + 4\cdot\obf_{\{\mu_{\xi^k} < 0\}})] = 0$. In the context of finite-sum problems, it means that $\overline{\mu} = \frac{1}{n}\sum_{i\in[n]:\mu_i \ge 0} \mu_i + \frac{4}{n}\sum_{i\in[n]:\mu_i < 0} \mu_i = 0$ both for \algname{S-SEG-US} and \algname{S-SEG-IS}. For \algname{I-SEG} we use the term quasi-monotonicity to describe the problems satisfying Assumption~\ref{as:str_monotonicity} with $\mu = 0$. The resulting inequality $\langle F(x), x - x^* \rangle \ge 0$ is also known as variational stability condition~\citep{hsieh2020explore, loizou2021stochastic}.

The best-known results \citep{mishchenko2020revisiting, beznosikov2020distributed} provide convergence guarantees in terms of the gap function \citep{nesterov2007dual}: $\text{Gap}_{\cC}(z) = \max_{u\in \cC}\langle F(u), z - u \rangle$, where $\cC$ is a compact set containing the solution set of \eqref{eq:main_problem}. In particular, \citet{beznosikov2020distributed} derive a convergence guarantee for $\Exp[\text{Gap}_{\cC}(\hat{x}^k)]$, where $\hat{x}^K$ is the average of the iterates produced by the method and the problem is assumed to be defined on a compact set. The last requirement is quite restrictive, since many practically important problems are naturally unconstrained. \citet{mishchenko2020revisiting} do not make such an assumption and consider VIPs with regularization, but derive convergence guarantees for $\max_{u \in \cC}\Exp\left[\langle F(u), \hat{x}^K - u \rangle + R(\hat x^K) - R(u)\right]$, where $R(x)$ is the regularization term (in our settings, $R(x) \equiv 0$). That is, when $R(x) \equiv 0$ \citet{mishchenko2020revisiting} obtain upper bounds for $\text{Gap}_{\cC}(\Exp[\hat{x}^k])$ that is a weaker measure of convergence than $\Exp[\text{Gap}_{\cC}(\hat{x}^k)]$. 

However, \citet{mishchenko2020revisiting, beznosikov2020distributed} analyze \algname{SEG} without using large batchsizes. In contrast, our convergence results for \algname{S-SEG} and \algname{I-SEG} are given for the expected squared norm of the operator and hold in the unconstrained case, but, in general, require using target accuracy dependent batchsizes. However, when $F_{\xi}(x^*) = 0$ for all $\xi$, i.e., interpolation conditions are satisfied, our results for \algname{S-SEG} provide convergence guarantees to any predefined accuracy of the solution even with unit batchsizes ($b = 1$).

\paragraph{Last-iterate convergence rates without (quasi) strong monotonicity.} All the results from Table~\ref{tab:comparison_of_rates_monotone} are derived either for the best-iterate or for the averaged-iterate. However, last-iterate convergence results are much more valuable, since the last-iterate is usually used as an output of a method in practical applications. Unfortunately, without additional assumptions a little is known about convergence of \algname{SEG} in this settings. In fact, even for deterministic \algname{EG} tight $\cO(\nicefrac{1}{K})$ last-iterate convergence results were obtained \citep{golowich2020last} under the additional assumption that the Jacobian of $F$ is Lipschitz-continuous, and only recently \citet{gorbunov2021extragradient} derive $\cO(\nicefrac{1}{K})$ last-iterate convergence rate without using this additional assumption. There are also several linear last-iterate convergence results under the assumption that the operator $F$ is affine and satisfies $\|F(x)\| \ge \mu\|x - x^*\|$ ($x^*$ is the closest solution to $x$) \citep{hsieh2020explore} and under the assumption that $F$ corresponds to the bilinear game \citep{mishchenko2020revisiting}.

\newpage

\section{BASIC INEQUALITIES AND AUXILIARY RESULTS}

\subsection{Basic Inequalities}

For all $a,b, a_1, a_2, \ldots, a_n \in \R^d$, $n \ge 1$ the following inequalities hold:
\begin{eqnarray}
	\left\|\sum\limits_{i=1}^n a_i\right\|^2 &\leq& n\sum\limits_{i=1}^n\|a_i\|^2,\label{eq:a+b}\\
	\|a+b\|^2 &\geq& \frac{1}{2}\|a\|^2 - \|b\|^2, \label{eq:a+b_lower}\\
	2\langle a, b\rangle &=& \|a\|^2 + \|b\|^2 - \|a-b\|^2. \label{eq:inner_product_representation}
\end{eqnarray}

\subsection{Auxiliary Results}

We use the following lemma from \cite{stich2019unified} to derive the final convergence rates from our results on linear convergence to the neighborhood.
\begin{lemma}[Simplified version of Lemma 3 from \cite{stich2019unified}]\label{lem:stich_lemma_for_str_cvx_conv}
	Let the non-negative sequence $\{r_k\}_{k\ge 0}$ satisfy the relation
	\begin{equation*}
		r_{k+1} \leq (1 - a\gamma_k)r_k + c\gamma_k^2
	\end{equation*}
	for all $k \geq 0$, parameters $a,c\ge 0$, and any non-negative sequence $\{\gamma_k\}_{k\ge 0}$ such that $\gamma_k \leq \nicefrac{1}{h}$ for some $h \ge a$, $h> 0$. Then, for any $K \ge 0$ one can choose $\{\gamma_k\}_{k \ge 0}$ as follows:
	\begin{eqnarray*}
		\text{if } K \le \frac{h}{a}, && \gamma_k = \frac{1}{h},\\
		\text{if } K > \frac{h}{a} \text{ and } k < k_0, && \gamma_k = \frac{1}{h},\\
		\text{if } K > \frac{h}{a} \text{ and } k \ge k_0, && \gamma_k = \frac{2}{a(\kappa + k - k_0)},
	\end{eqnarray*}
	where $\kappa = \nicefrac{2h}{a}$ and $k_0 = \left\lceil \nicefrac{K}{2} \right\rceil$. For this choice of $\gamma_k$ the following inequality holds:
	\begin{eqnarray*}
		r_{K} \le \frac{32hr_0}{a}\exp\left(-\frac{a K}{2h}\right) + \frac{36c}{a^2 K}.
	\end{eqnarray*}
\end{lemma}

\newpage

\section{GENERAL ANALYSIS OF \algname{SEG}: MISSING PROOFS}

\begin{theorem}[Theorem~\ref{thm:main_theorem_general_main}]\label{thm:main_theorem_general}
	Consider the method \eqref{eq:general_method}. Let Assumption~\ref{as:unified_assumption_general} hold and $A\leq \nicefrac{1}{2}$. Then for all $K\ge 0$
	\begin{equation}
		\Exp \left[\|x^{K+1} - x^*\|^2\right] \leq  (1+C-\rho)\Exp\left[\|x^K - x^*\|^2\right] + D_1 + D_2, \label{eq:main_result_str_mon_one_step}
	\end{equation}
	\begin{equation}
		\Exp \left[\|x^K - x^*\|^2\right] \leq  (1+C-\rho)^K\|x^0 - x^*\|^2 + \frac{D_1 + D_2}{\rho - C}, \label{eq:main_result_str_mon}
	\end{equation}
	when $\rho > C \ge 0$, and
	\begin{equation}
		\frac{1}{K+1}\sum\limits_{k=0}^K \Exp[G_k] \leq  \frac{\|x^0 - x^*\|^2}{B(K+1)} + \frac{D_1+D_2}{B}, \label{eq:main_result_mon}
	\end{equation}
	when $\rho = C = 0$ and $B > 0$.
\end{theorem}
\begin{proof}
     Since $x^{k+1} = x^k - \gamma_{\xi^k} g_{\xi^k}(x^k)$, we have
     \begin{eqnarray*}
         \|x^{k+1} - x^*\|^2 &=& \|x^k - \gamma_{\xi^k} g_{\xi^k}(x^k) - x^*\|^2\\
         &=& \|x^k - x^*\|^2 - 2\gamma_{\xi^k}\langle g_{\xi^k}(x^k), x^k - x^*\rangle + \gamma_{\xi^k}^2\|g_{\xi^k}(x^k)\|^2.
     \end{eqnarray*}
     Taking the expectation, conditioned on $\xi^k$,  using our Assumption~\ref{as:unified_assumption_general} and the definition of $P_k = \Exp_{\xi^k}\left[\gamma_{\xi^k}\langle g_{\xi^k}(x^k), x^k - x^* \rangle\right]$, we continue our derivation:
     \begin{eqnarray*}
         \Exp_{\xi^k}\left[\|x^{k+1} - x^*\|^2\right] &=& \|x^k - x^*\|^2 - 2P_k + \Exp_{\xi^k}[\gamma_{\xi^k}^2\|g_{\xi^k}(x^k)\|^2]\\
         &\overset{\eqref{eq:second_moment_bound_general}}{\le}& \|x^k - x^*\|^2 - 2P_k + 2AP_k + C\|x^k - x^*\|^2 + D_1\\
         &\overset{A \le \nicefrac{1}{2}}{\leq}& (1+C) \|x^k - x^*\|^2 - P_k + D_1\\
         &\overset{\eqref{eq:P_k_general}}{\le}& (1+C-\rho)\|x^k - x^*\|^2 - BG_k + D_1 + D_2.
     \end{eqnarray*}
     Next, we take the full expectation from the both sides
     \begin{equation}
         \Exp\left[\|x^{k+1} - x^*\|^2\right] \le (1+C-\rho)\Exp\left[\|x^k - x^*\|^2\right] - B\Exp[G_k] + D_1 + D_2. \label{eq:njcksnjkajciubkscnkd}
     \end{equation}
     If $\rho > C \ge 0$, then in the above inequality we can get rid of the non-positive term $(-B\Exp[G_k])$
    \begin{equation*}
         \Exp\left[\|x^{k+1} - x^*\|^2\right] \le (1+C-\rho)\Exp\left[\|x^k - x^*\|^2\right] + D_1 + D_2
    \end{equation*}
    and get \eqref{eq:main_result_str_mon_one_step}. Unrolling the recurrence, we derive \eqref{eq:main_result_str_mon}:
    \begin{eqnarray*}
        \Exp \left[\|x^K - x^*\|^2\right] &\leq&  (1+C-\rho)^K\|x^0 - x^*\|^2 + (D_1 + D_2) \sum\limits_{k=0}^{K-1}(1+C -\rho)^k\\
        &\leq&  (1+C-\rho)^K\|x^0 - x^*\|^2 + (D_1 + D_2) \sum\limits_{k=0}^{\infty}(1+C -\rho)^k\\
        &=& (1+C-\rho)^K\|x^0 - x^*\|^2 + \frac{D_1 + D_2}{\rho - C}.
    \end{eqnarray*}
    If $\rho = C = 0$ and $B > 0$, then \eqref{eq:njcksnjkajciubkscnkd} is equivalent to
    \begin{equation*}
        B\Exp[ G_k] \le \Exp\left[\|x^k - x^*\|^2\right] - \Exp\left[\|x^{k+1} - x^*\|^2\right] + D_1 + D_2.
    \end{equation*}
    Summing up these inequalities for $k = 0,1, \ldots, K$ and dividing the result by $B(K+1)$, we get \eqref{eq:main_result_mon}:
    \begin{eqnarray*}
        \frac{1}{K+1}\sum\limits_{k=0}^{K}\Exp[G_k] &\leq& \frac{1}{B(K+1)}\sum\limits_{k=0}^{K} \left(\Exp\left[\|x^k - x^*\|^2\right] - \Exp\left[\|x^{k+1} - x^*\|^2\right]\right) + \frac{D_1 + D_2}{B(K+1)}\\
        &=& \frac{1}{B(K+1)}\left(\|x^0 - x^*\|^2 - \Exp\left[\|x^{K+1} - x^*\|^2\right]\right) + \frac{D_1 + D_2}{B(K+1)}\\
        &\leq& \frac{\|x^0 - x^*\|^2}{B(K+1)} + \frac{D_1 + D_2}{B(K+1)}.
    \end{eqnarray*}
\end{proof}

\newpage

\section{SAME-SAMPLE \algname{SEG} (\algname{S-SEG}): MISSING PROOFS AND ADDITIONAL DETAILS}
In this section, we provide full proofs and missing details from Section~\ref{sec:S_SEG} on \algname{S-SEG}. Recall that our analysis of \algname{S-SEG} based on the three following assumptions:
\begin{itemize}
    \item $F_\xi(x)$ is $L_\xi$-Lipschitz: $\|F_\xi (x) - F_\xi(y)\| \le L_\xi \|x - y\|$ for all $x,y\in \R^d$ (Assumption~\ref{as:F_xi_lip}),
    \item $F_\xi(x)$ is $(\mu_\xi, x^*)$-strongly monotone (with possibly negative $\mu_\xi$): $\langle F_\xi(x) - F_\xi(x^*), x- x^*\rangle \ge \mu_\xi\|x - x^*\|^2$ for all $x\in \R^d$ (Assumption~\ref{as:F_xi_str_monotonicity}),
    \item the following conditions (inequalities \eqref{eq:SS_SEG_AS_stepsizes_1}-\eqref{eq:SS_SEG_AS_stepsizes_2}) hold:
    \begin{eqnarray*}
        \Exp_{\xi^k}[\gamma_{1,\xi^k}F_{\xi^k}(x^*)] = 0, \quad \Exp_{\xi^k}[\gamma_{1,\xi^k}\mu_{\xi^k}(\obf_{\{\mu_{\xi^k} \ge 0\}} + 4\cdot\obf_{\{\mu_{\xi^k} < 0\}})] \ge 0. \notag
    \end{eqnarray*}
\end{itemize}

\subsection{Details on the Examples of Arbitrary Sampling}\label{sec:AS_examples_appendix}
In Section~\ref{sec:S_SEG}, we provide several examples when the assumptions above are satisfied. In all examples, we assume that $F(x)$ has a finite-sum form
\begin{equation}
	F(x) = \frac{1}{n}\sum\limits_{i=1}^n F_i(x) \label{eq:F_sum}
\end{equation}
and $F_i$ is $L_i$-Lipschitz and $(\mu_i, x^*)$-strongly monotone. First, we consider \algname{S-SEG} with independent sampling with replacement, which covers uniform sampling (Example~\ref{ex:uniform_sampling}) and importance sampling (Example~\ref{ex:importance_sampling}).

\begin{example}[Independent sampling with replacement]\label{ex:independent_sampl_with_repl}
    Let random indices $j_1, \ldots, j_b$ are sampled independently from the the distribution $\cD$ such that for $j \sim \cD$ we have $\Prob{j = i} = p_i > 0$ for $i = 1,\ldots, n$, $\sum_{i=1}^n p_i = 1$. Let $\xi = (j_1,\ldots,j_b)$ and $F_{\xi}(x) = \frac{1}{b}\sum_{l=1}^b F_{j_l}(x)$. Moreover, assume that 
    \begin{equation*}
       \sum\limits_{j_1,\ldots, j_b :\mu_{(j_1,\ldots,j_b)} \ge 0}\mu_{(j_1,\ldots,j_b)} + 4\sum\limits_{j_1,\ldots, j_b :\mu_{(j_1,\ldots,j_b)} < 0}\mu_{(j_1,\ldots,j_b)} \ge 0,
    \end{equation*}
    where $\mu_{(j_1,\ldots,j_b)} \ge \frac{1}{b}\sum_{i=1}^b \mu_{j_l}$ is such that the operator $\frac{1}{b}\sum_{l=1}^b F_{j_l}(x)$ is $(\mu_{(j_1,\ldots,j_b)},x^*)$-strongly monotone. For example, the above inequality is satisfied when all $\mu_i \ge 0$. Then, Assumptions~\ref{as:F_xi_lip}~and~\ref{as:F_xi_str_monotonicity} hold with $L_{\xi} \le \frac{1}{b}\sum_{l=1}^b L_{j_l}$, $\mu_{\xi} \ge \frac{1}{b}\sum_{i=1}^b \mu_{j_l}$, and for the stepsize 
    \begin{equation*}
        \gamma_{1,\xi} = \frac{\gamma b}{n^b p_\xi},\quad \gamma > 0,\quad p_\xi = \Prob{\xi = (j_1,\ldots, j_b)} = p_{j_1}\ldots p_{j_b}
    \end{equation*}
    we have
    \begin{eqnarray*}
        \Exp_{\xi^k}[\gamma_{1,\xi^k}F_{\xi^k}(x^*)] &=& \frac{\gamma}{n^b}\sum\limits_{j_1,\ldots, j_b=1}^n\sum\limits_{l=1}^b F_{j_l}(x^*) = \frac{\gamma}{n}\sum\limits_{i=1}^n F_i(x^*) = \gamma F(x^*) = 0 \notag\\
    \end{eqnarray*}
    and
    \begin{eqnarray*}
        \Exp_{\xi^k}[\gamma_{1,\xi^k}\mu_{\xi^k}\obf_{\{\mu_{\xi^k} \ge 0\}} + 4\gamma_{1,\xi^k}\mu_{\xi^k}\obf_{\{\mu_{\xi^k} < 0\}}] &=& \frac{\gamma b}{n^b}\sum\limits_{j_1,\ldots, j_b :\mu_{(j_1,\ldots,j_b)} \ge 0}\mu_{(j_1,\ldots,j_b)}\\
        &&\quad +\frac{4\gamma b}{n^b}\sum\limits_{j_1,\ldots, j_b:\mu_{(j_1,\ldots,j_b)} < 0}\mu_{(j_1,\ldots,j_b)}\\
        &\ge& 0,
    \end{eqnarray*}
    i.e., conditions from \eqref{eq:SS_SEG_AS_stepsizes_1}-\eqref{eq:SS_SEG_AS_stepsizes_2} are satisfied.
\end{example}

Taking $b=1$ and $p_1 = \ldots = p_n = \nicefrac{1}{n}$ in the previous example we recover single-batch uniform sampling (Example~\ref{ex:uniform_sampling}) as as special case. If $p_i = \nicefrac{L_i}{\sum_{j=1}^n L_j}$, then we get single-batch importance sampling (Example~\ref{ex:importance_sampling}) as a special case of the previous example.

Finally, we consider two without-replacement sampling strategies. The first one called $b$-nice sampling is described in Section~\ref{sec:S_SEG} (Example~\ref{ex:nice}). Below we prove that conditions \eqref{eq:SS_SEG_AS_stepsizes_1}-\eqref{eq:SS_SEG_AS_stepsizes_2} hold for this example. For the reader's convenience, we also provide a complete description of this sampling.

\begin{example}[$b$-nice sampling]\label{ex:nice_full}
    Let $\xi$ be a random subset of size $b \in [n]$ chosen from the uniform distribution on the family of all subsets of $[n]$ of size $b$. Then, for each $S\subseteq [n], |S| = b$ we have
    \begin{equation*}
        p_S = \Prob{\xi = S} = \frac{1}{\binom{n}{b}}.
    \end{equation*}
    Next, let $F_{\xi}(x) = \frac{1}{b}\sum_{i\in \xi} F_i(x)$ and $\gamma_{1,\xi} = \gamma$. Moreover, assume that
    \begin{equation*}
       \overline{\mu}_{b-\algname{NICE}} = \frac{1}{\binom{n}{b}}\left(\sum\limits_{S\subseteq [n], |S| = b :\mu_{S} \ge 0}\mu_{S} + 4\sum\limits_{S\subseteq [n], |S| = b :\mu_{S} < 0}\mu_{S}\right) \ge 0,
    \end{equation*}
    where $\mu_{S} \ge \frac{1}{b}\sum_{i\in S} \mu_{i}$ is such that the operator $\frac{1}{b}\sum_{i\in S} F_{i}(x)$ is $(\mu_{S},x^*)$-strongly monotone. For example, the above inequality is satisfied when all $\mu_i \ge 0$. Then, Assumptions~\ref{as:F_xi_lip}~and~\ref{as:F_xi_str_monotonicity} hold with $L_{\xi} \le \frac{1}{b}\sum_{i\in \xi} L_{i}$, $\mu_{\xi} \ge \frac{1}{b}\sum_{i\in \xi} \mu_{i}$, and  we have
    \begin{eqnarray*}
        \Exp_{\xi^k}[\gamma_{1,\xi^k}F_{\xi^k}(x^*)] &=& \frac{\gamma}{b\binom{n}{b}}\sum\limits_{S\subseteq [n],|S| = b}\sum\limits_{i\in S} F_{i}(x^*) = \frac{\gamma\binom{n-1}{b-1}}{\binom{n}{b}b}\sum\limits_{i=1}^n F_i(x^*) = \gamma F(x^*) = 0 \notag\\
    \end{eqnarray*}
     and
    \begin{eqnarray*}
        \Exp_{\xi^k}[\gamma_{1,\xi^k}\mu_{\xi^k}\obf_{\{\mu_{\xi^k} \ge 0\}} + 4\gamma_{1,\xi^k}\mu_{\xi^k}\obf_{\{\mu_{\xi^k} < 0\}}] &=& \frac{\gamma}{\binom{n}{b}}\sum\limits_{S\subseteq [n], |S| = b :\mu_{S} \ge 0}\mu_{S} \\
        &&\quad + \frac{4\gamma}{\binom{n}{b}}\sum\limits_{S\subseteq [n], |S| = b :\mu_{S} < 0}\mu_{S}\\
        &=& \gamma\overline{\mu}_{b-\algname{NICE}} \ge  0,
    \end{eqnarray*}
    i.e., conditions from \eqref{eq:SS_SEG_AS_stepsizes_1}-\eqref{eq:SS_SEG_AS_stepsizes_2} are satisfied.
\end{example}

The second without-sampling strategy, which we consider, is independent sampling without replacement.

\begin{example}[Independent sampling without replacement]\label{ex:iswor}
    Let $\xi$ be a random subset of $[n]$ such that each $i$ is picked with probability $p_i$ independently from other elements. It means that the size of $\xi$ is a random variable as well and $\Exp[|\xi|] = \sum_{i=1}^n p_i$. Next, we define
    \begin{equation*}
        F_{\xi}(x) = \frac{1}{|\xi|}\sum\limits_{i\in\xi} F_i(x)
    \end{equation*}
    and
    \begin{equation*}
        \gamma_{1,\xi} = \frac{\gamma |\xi|}{p_\xi 2^{n-1} n},\quad p_\xi = \Prob{\xi = S} = \prod\limits_{i\in S}p_i \cdot \prod\limits_{i\not\in S}(1-p_i)
    \end{equation*}
    for any $S \subseteq [n]$. Moreover, assume that 
    \begin{equation*}
        \sum\limits_{S\subseteq [n] :\mu_{S} \ge 0}|S|\mu_{S} + 4\sum\limits_{S\subseteq [n] :\mu_{S} < 0}|S|\mu_{S} \ge 0,
    \end{equation*}
    where $\mu_{S} \ge \frac{1}{|S|}\sum_{i\in S} \mu_{i}$ is such that the operator $\frac{1}{|S|}\sum_{i\in S} F_{i}(x)$ is $(\mu_{S},x^*)$-strongly monotone. For example, the above inequality is satisfied when all $\mu_i \ge 0$. Then, Assumptions~\ref{as:F_xi_lip}~and~\ref{as:F_xi_str_monotonicity} hold with $L_{\xi} \le \frac{1}{|\xi|}\sum_{i\in \xi} L_{i}$, $\mu_{\xi} \ge \frac{1}{|\xi|}\sum_{i\in \xi} \mu_{i}$, and  we have
    \begin{eqnarray*}
        \Exp_{\xi^k}[\gamma_{1,\xi^k}F_{\xi^k}(x^*)] &=& \frac{\gamma}{2^{n-1}n}\sum\limits_{S\subseteq [n]}\sum\limits_{i\in S} F_{i}(x^*) = \frac{\gamma}{n}\sum\limits_{i=1}^n F_i(x^*) = \gamma F(x^*) = 0 \notag\\
    \end{eqnarray*}
     and
    \begin{eqnarray*}
        \Exp_{\xi^k}[\gamma_{1,\xi^k}\mu_{\xi^k}\obf_{\{\mu_{\xi^k} \ge 0\}} + 4\gamma_{1,\xi^k}\mu_{\xi^k}\obf_{\{\mu_{\xi^k} < 0\}}] &=& \frac{\gamma}{2^{n-1}n}\sum\limits_{S\subseteq [n] :\mu_{S} \ge 0}|S|\mu_{S} \\
        &&\quad + \frac{4\gamma}{2^{n-1}n}\sum\limits_{S\subseteq [n] :\mu_{S} < 0}|S|\mu_{S}\\
        &\ge& 0,
    \end{eqnarray*}
    i.e., conditions from \eqref{eq:SS_SEG_AS_stepsizes_1}-\eqref{eq:SS_SEG_AS_stepsizes_2} are satisfied.
\end{example}

\subsection{Proof of the Main Result}
\label{Appendix_MainResult}

The proof is based on two lemmas showing that Assumption~\ref{as:unified_assumption_general} is satisfied.

\begin{lemma}\label{lem:second_moment_bound_S_SEG}
	Let Assumptions~\ref{as:F_xi_lip}~and~\ref{as:F_xi_str_monotonicity} hold. If $\gamma_{1,\xi^k}$ satisfies \eqref{eq:SS_SEG_AS_stepsizes_1}-\eqref{eq:SS_SEG_AS_stepsizes_2} and
	\begin{equation}
	    \gamma_{1,\xi^k} \leq \frac{1}{4|\mu_{\xi^k}| + \sqrt{2}L_{\xi^k}} \label{eq:SS_SEG_AS_stepsizes_11}
	\end{equation}
    then $g^k = F_{\xi^k}\left(x^k - \gamma_{1,\xi^k} F_{\xi^k}(x^k)\right)$ satisfies the following inequality
	\begin{eqnarray}
		\Exp_{\xi^k}\left[\gamma_{1,\xi^k}^2\|g^k\|^2\right] &\leq& 4 \widehat{P}_k + 6 \Exp_{\xi^k}\left[\gamma_{1,\xi^k}^2\|F_{\xi^k}(x^*)\|^2\right], \label{eq:second_moment_bound_SS_SEG_AS}
	\end{eqnarray}
	where $\widehat{P}_k = \Exp_{\xi^k}\left[\gamma_{1,\xi^k}\langle g^k, x^k - x^* \rangle\right]$.
\end{lemma}
\begin{proof}
Using the auxiliary iterate\footnote{We use $\widehat{x}^{k+1}$ as a tool in the proof. There is no need to compute $\widehat{x}^{k+1}$ during the run of the method.} $\widehat{x}^{k+1} = x^k - \gamma_{1,\xi^k} g^k$, we get
	\begin{eqnarray}
		\|\widehat{x}^{k+1} - x^*\|^2 &=& \|x^k - x^*\|^2 -2\gamma_{1,\xi^k}\langle x^k-x^*, g^k \rangle + \gamma_{1,\xi^k}^2 \|g^k\|^2\label{eq:sec_mom_SS_SEG_AS_technical_1}\\
		&=& \|x^k - x^*\|^2 -2\gamma_{1,\xi^k}\left\langle x^k - \gamma_{1,\xi^k} F_{\xi^k}(x^k) - x^*, g^k \right\rangle \notag\\
		&&\quad - 2\gamma_{1,\xi^k}^2 \langle F_{\xi^k}(x^k), g^k \rangle + \gamma_{1,\xi^k}^2 \|g^k\|^2\notag \\
		&=& \|x^k - x^*\|^2 -2\gamma_{1,\xi^k}\left\langle x^k - \gamma_{1,\xi^k} F_{\xi^k}(x^k) - x^*, g^k - F_{\xi^k}(x^*) \right\rangle\notag \\
		&&\quad - 2\gamma_{1,\xi^k}^2 \langle F_{\xi^k}(x^k), g^k - F_{\xi^k}(x^*) \rangle - 2\gamma_{1,\xi^k}\langle x^k - x^*, F_{\xi^k}(x^*) \rangle + \gamma_{1,\xi^k}^2 \|g^k\|^2.\notag
	\end{eqnarray}
	Taking the expectation w.r.t.\ $\xi^k$ from the above identity, using $\Exp_{\xi^k}[\gamma_{1,\xi^k}\langle x^k - x^*, F_{\xi^k}(x^*) \rangle] = \langle x^k - x^*, \Exp_{\xi^k}[\gamma_{1,\xi^k}F_{\xi^k}(x^*)] \rangle \overset{\eqref{eq:SS_SEG_AS_stepsizes_1}}{=} 0$, $g^k = F_{\xi^k}\left(x^k - \gamma_{1,\xi^k} F_{\xi^k}(x^k)\right)$ and $(\mu_\xi,x^*)$-strong monotonicity of $F_{\xi}(x)$, we derive
	\begin{eqnarray*}
		\Exp_{\xi^k}\left[\|\widehat{x}^{k+1} - x^*\|^2\right] &\leq& \|x^k - x^*\|^2 - 2 \Exp_{\xi^k}\left[\gamma_{1,{\xi^k}}\left\langle x^k - \gamma_{1,{\xi^k}} F(x^k, {\xi^k}) - x^*, g^k - F(x^*, {\xi^k}) \right\rangle\right]\\
		&&\quad - 2\Exp_{\xi^k}\left[\gamma_{1,{\xi^k}}^2\langle F(x^k, {\xi^k}), g^k - F(x^*, {\xi^k}) \rangle\right] + \Exp_{\xi^k}\left[\gamma_{1,{\xi^k}}^2\|g^k\|^2\right]\\
		&\overset{\eqref{eq:F_xi_str_monotonicity}}{\le}& \|x^k - x^*\|^2 - 2\Exp_{\xi^k}\left[\gamma_{1,\xi^k}\mu_{\xi^k}\|x^k - x^* - \gamma_{1,{\xi^k}} F(x^k, {\xi^k})\|^2\right]\\
		&&\quad - 2\Exp_{\xi^k}\left[\gamma_{1,\xi^k}^2\langle F_{\xi^k}(x^k), g^k - F_{\xi^k}(x^*) \rangle\right] +  \Exp_{\xi^k}\left[\gamma_{1,\xi^k}^2\|g^k\|^2\right]\\
		&=& \|x^k - x^*\|^2 - 2\Exp_{\xi^k}\left[\gamma_{1,\xi^k}\mu_{\xi^k}\obf_{\{\mu_{\xi^k} \ge 0\}}\|x^k - x^* - \gamma_{1,{\xi^k}} F(x^k, {\xi^k})\|^2\right]\\
		&&\quad - 2\Exp_{\xi^k}\left[\gamma_{1,\xi^k}\mu_{\xi^k}\obf_{\{\mu_{\xi^k} < 0\}}\|x^k - x^* - \gamma_{1,{\xi^k}} F(x^k, {\xi^k})\|^2\right]\\
		&&\quad - 2\Exp_{\xi^k}\left[\gamma_{1,\xi^k}^2\langle F_{\xi^k}(x^k), g^k - F_{\xi^k}(x^*) \rangle\right] +  \Exp_{\xi^k}\left[\gamma_{1,\xi^k}^2\|g^k\|^2\right]\\
		&\overset{\eqref{eq:a+b_lower}}{\leq}& \|x^k - x^*\|^2 - \Exp_{\xi^k}[\gamma_{1,\xi^k}\mu_{\xi^k}\obf_{\{\mu_{\xi^k} \ge 0\}}]\|x^k-x^*\|^2\\
		&&\quad + 2\Exp_{\xi^k}[\gamma_{1,\xi^k}^3\mu_{\xi^k}\obf_{\{\mu_{\xi^k} \ge 0\}}\|F_{\xi^k}(x^k)\|^2]\\
		&&\quad - 2\Exp_{\xi^k}\left[\gamma_{1,\xi^k}\mu_{\xi^k}\obf_{\{\mu_{\xi^k} < 0\}}\|x^k - x^* - \gamma_{1,{\xi^k}} F(x^k, {\xi^k})\|^2\right]\\
		&&\quad - 2\Exp_{\xi^k}\left[\gamma_{1,\xi^k}^2\langle F_{\xi^k}(x^k), g^k - F_{\xi^k}(x^*) \rangle\right] +  \Exp_{\xi^k}\left[\gamma_{1,\xi^k}^2\|g^k\|^2\right]\\
		&\overset{\eqref{eq:a+b}}{\leq}& \|x^k - x^*\|^2 - \Exp_{\xi^k}[\gamma_{1,\xi^k}\mu_{\xi^k}\obf_{\{\mu_{\xi^k} \ge 0\}} + 4\gamma_{1,\xi^k}\mu_{\xi^k}\obf_{\{\mu_{\xi^k} < 0\}}]\|x^k-x^*\|^2\\
		&&\quad +  \Exp_{\xi^k}\left[\gamma_{1,{\xi^k}}^2\|g^k\|^2\right]\\
		&&\quad + 2\Exp_{\xi^k}\left[\gamma_{1,{\xi^k}}^3(\mu_{\xi^k}\obf_{\{\mu_{\xi^k} \ge 0\}} - 2\mu_{\xi^k}\obf_{\{\mu_{\xi^k} < 0\}}))\|F_{\xi^k}(x^k)\|^2\right]\\
		&&\quad - \Exp_{\xi^k}\left[\gamma_{1,{\xi^k}}^2\|g^k - F_{\xi^k}(x^*)\|^2\right]\\
		&&\quad + \Exp_{\xi^k}\left[\gamma_{1,{\xi^k}}^2\|F_{\xi^k}(x^k) - g^k + F_{\xi^k}(x^*)\|^2\right]\\
		&\overset{\eqref{eq:inner_product_representation}}{\leq}& \|x^k - x^*\|^2 - \Exp_{\xi^k}[\gamma_{1,\xi^k}\mu_{\xi^k}\obf_{\{\mu_{\xi^k} \ge 0\}} + 4\gamma_{1,\xi^k}\mu_{\xi^k}\obf_{\{\mu_{\xi^k} < 0\}}]\|x^k-x^*\|^2\\
		&&\quad +  \Exp_{\xi^k}\left[\gamma_{1,{\xi^k}}^2\|g^k\|^2\right]\\
		&&\quad - \Exp_{\xi^k}\left[\gamma_{1,{\xi^k}}^2(1 - 2\gamma_{1,\xi^k}(\mu_{\xi^k}\obf_{\{\mu_{\xi^k} \ge 0\}} - 2\mu_{\xi^k}\obf_{\{\mu_{\xi^k} < 0\}}))\|F_{\xi^k}(x^k)\|^2\right]\\
		&&\quad - \Exp_{\xi^k}\left[\gamma_{1,{\xi^k}}^2\|g^k - F_{\xi^k}(x^*)\|^2\right]\\
		&&\quad + \Exp_{\xi^k}\left[\gamma_{1,{\xi^k}}^2\|F_{\xi^k}(x^k) - g^k + F_{\xi^k}(x^*)\|^2\right]\\
		&\overset{\eqref{eq:SS_SEG_AS_stepsizes_2}}{\le}& \|x^k - x^*\|^2 +  \Exp_{\xi^k}\left[\gamma_{1,{\xi^k}}^2\|g^k\|^2\right]\\
		&&\quad - \Exp_{\xi^k}\left[\gamma_{1,{\xi^k}}^2(1 - 4\gamma_{1,\xi^k}|\mu_{\xi^k}|)\|F_{\xi^k}(x^k)\|^2\right]\\
		&&\quad - \Exp_{\xi^k}\left[\gamma_{1,{\xi^k}}^2\|g^k - F_{\xi^k}(x^*)\|^2\right]\\
		&&\quad + \Exp_{\xi^k}\left[\gamma_{1,{\xi^k}}^2\|F_{\xi^k}(x^k) - g^k + F_{\xi^k}(x^*)\|^2\right],
	\end{eqnarray*}
	where in the last inequality we use\footnote{When all $\mu_{\xi} \ge 0$, which is often assumed in the analysis of \algname{S-SEG}, numerical constants in our proof can be tightened. Indeed, in the last step, we can get $- \Exp_{\xi^k}\left[\gamma_{1,{\xi^k}}^2(1 - 2\gamma_{1,\xi^k}\mu_{\xi^k})\|F_{\xi^k}(x^k)\|^2\right]$ instead of $- \Exp_{\xi^k}\left[\gamma_{1,{\xi^k}}^2(1 - 4\gamma_{1,\xi^k}|\mu_{\xi^k}|)\|F_{\xi^k}(x^k)\|^2\right]$.} $\mu_{\xi^k}\obf_{\{\mu_{\xi^k} \ge 0\}} - 2\mu_{\xi^k}\obf_{\{\mu_{\xi^k} < 0\}} \le 2|\mu_{\xi^k}|$. To upper bound the last two terms we use simple inequalities \eqref{eq:a+b_lower} and \eqref{eq:a+b}, and apply $L_{\xi^k}$-Lipschitzness of $F_{\xi^k}(x)$:
	\begin{eqnarray*}
		\Exp\left[\|\widehat{x}^{k+1} - x^*\|^2\mid x^k\right] &\overset{\eqref{eq:a+b_lower}, \eqref{eq:a+b}}{\leq}& \|x^k - x^*\|^2 +  \Exp_{\xi^k}\left[\gamma_{1,\xi^k}^2\|g^k\|^2\right]\\
		&&\quad - \Exp_{\xi^k}\left[\gamma_{1,{\xi^k}}^2(1 - 4\gamma_{1,\xi^k}|\mu_{\xi^k}|)\|F_{\xi^k}(x^k)\|^2\right]\\
		&&\quad - \frac{1}{2}\Exp_{\xi^k}\left[\gamma_{1,\xi^k}^2\|g^k\|^2\right] + \Exp_{\xi^k}\left[\gamma_{1,\xi^k}^2\|F_{\xi^k}(x^*)\|^2\right]\\
		&&\quad + 2\Exp_{\xi^k}\left[\gamma_{1,\xi^k}^2\|F_{\xi^k}(x^k) - g^k\|^2\right] + 2\Exp_{\xi^k}\left[\gamma_{1,\xi^k}^2\|F_{\xi^k}(x^*)\|^2\right]\\
		&=& \|x^k - x^*\|^2 +  \frac{1}{2}\Exp_{\xi^k}\left[\gamma_{1,\xi^k}^2\|g^k\|^2\right]\\
		&&\quad - \Exp_{\xi^k}\left[\gamma_{1,{\xi^k}}^2(1 - 4\gamma_{1,\xi^k}|\mu_{\xi^k}|)\|F_{\xi^k}(x^k)\|^2\right]\\
		&&\quad + 3 \Exp_{\xi^k}\left[\gamma_{1,\xi^k}^2\|F_{\xi^k}(x^*)\|^2\right]\\
		&&\quad + 2 \Exp_{\xi^k}\left[\gamma_{1,\xi^k}^2\|F_{\xi^k}(x^k) - F_{\xi^k}(x^k - \gamma_{1,\xi^k} F_{\xi^k}(x^k))\|^2\right]\\
		&\overset{\eqref{eq:F_xi_lip}}{\leq}& \|x^k - x^*\|^2 + \frac{1}{2}\Exp_{\xi^k}\left[\gamma_{1,\xi^k}^2\|g^k\|^2\right] + 3 \Exp_{\xi^k}\left[\gamma_{1,\xi^k}^2\|F_{\xi^k}(x^*)\|^2\right]\\
		&&\quad - \Exp_{\xi^k}\left[\gamma_{1,\xi^k}^2\left(1 - 4\gamma_{1,\xi^k}|\mu_{\xi^k}| - 2L_{\xi^k}^2\gamma_{1,\xi^k}^2\right)\|F_{\xi^k}(x^k)\|^2\right]\\
		&\overset{\eqref{eq:SS_SEG_AS_stepsizes_11}}{\le}& \|x^k - x^*\|^2 + \frac{1}{2}\Exp_{\xi^k}\left[\gamma_{1,\xi^k}^2\|g^k\|^2\right] + 3 \Exp_{\xi^k}\left[\gamma_{1,\xi^k}^2\|F(x^*, \xi^k)\|^2\right],
	\end{eqnarray*}
	Finally, we use the above inequality together with \eqref{eq:sec_mom_SS_SEG_AS_technical_1}:
	\begin{eqnarray*}
		\| x^k - x^*\|^2 -2\widehat{P}_k +  \Exp_{\xi^k}\left[\gamma_{1,\xi^k}^2\|g^k\|^2\right] &\le& \|x^k - x^*\|^2 + \frac{1}{2}\Exp_{\xi^k}\left[\gamma_{1,\xi^k}^2\|g^k\|^2\right] + 3 \Exp_{\xi^k}\left[\gamma_{1,\xi^k}^2\|F_{\xi^k}(x^*)\|^2\right],
	\end{eqnarray*}		
	where $\widehat{P}_k = \Exp_{\xi^k}\left[\gamma_{1,\xi^k}\langle g^k, x^k - x^* \rangle\right]$. Rearranging the terms, we obtain \eqref{eq:second_moment_bound_SS_SEG_AS}.
\end{proof}

\begin{lemma}\label{lem:P_k_bound_S_SEG}
    Let Assumptions~\ref{as:F_xi_lip}~and~\ref{as:F_xi_str_monotonicity} hold. If $\gamma_{1,\xi^k}$ satisfies \eqref{eq:SS_SEG_AS_stepsizes_1},\eqref{eq:SS_SEG_AS_stepsizes_2}, and \eqref{eq:SS_SEG_AS_stepsizes_11}, then $g^k = F_{\xi^k}\left(x^k - \gamma_{1,\xi^k} F_{\xi^k}(x^k)\right)$ satisfies the following inequality
    \begin{equation}
        \widehat{P}_k \ge \widehat{\rho}\|x^k - x^*\|^2 + \frac{1}{2}\widehat{G}_k - \frac{3}{2}\Exp_{\xi^k}\left[\gamma_{1,\xi^k}^2\|F_\xi^k(x^*)\|^2\right] \label{eq:P_k_bound_S_SEG_AS}
    \end{equation}
    where $\widehat{P}_k = \Exp_{\xi^k}\left[\gamma_{1,\xi^k}\langle g^k, x^k - x^* \rangle\right]$ and
    \begin{eqnarray*}
        \widehat{\rho} = \frac{1}{2}\Exp_{\xi^k}[\gamma_{1,\xi^k}\mu_{\xi^k}(\obf_{\{\mu_{\xi^k} \ge 0\}} + 4\cdot\obf_{\{\mu_{\xi^k} < 0\}})],\\
        \widehat{G}_k = \Exp_{\xi^k}\left[\gamma_{1,\xi^k}^2\left(1 - 4|\mu_{\xi^k}|\gamma_{1,\xi^k} - 2 L_{\xi^k}^2\gamma_{1,\xi^k}^2\right)\|F_{\xi^k}(x^k)\|^2\right].
    \end{eqnarray*}
\end{lemma}
\begin{proof}
    We start with rewriting $\widehat{P}_k$:
    \begin{eqnarray}
        -\widehat{P}_k &=& -\Exp_{\xi^k}\left[\gamma_{1,\xi^k}\langle g^k, x^k - x^*\rangle\right] \overset{\eqref{eq:SS_SEG_AS_stepsizes_1}}{=} -\Exp_{\xi^k}\left[\gamma_{1,\xi^k}\langle g^k - F_{\xi^k}(x^*), x^k - x^*\rangle\right] \notag\\
        &=& -\Exp_{\xi^k}\left[\gamma_{1,\xi^k}\langle g^k - F_{\xi^k}(x^*), x^k - \gamma_{1,\xi^k}F_{\xi^k}(x^k) - x^*\rangle\right] \notag\\
        &&\quad - \Exp_{\xi^k}\left[\gamma_{1,\xi^k}^2\langle g^k - F_{\xi^k}(x^*),F_{\xi^k}(x^k)\rangle\right] \notag\\
        &\overset{\eqref{eq:inner_product_representation}}{=}& \underbrace{-\Exp_{\xi^k}\left[\gamma_{1,\xi^k}\left\langle F_{\xi^k}(x^k - \gamma_{1,\xi^k}F_{\xi^k}(x^k)) - F_{\xi^k}(x^*), x^k - \gamma_{1,\xi^k}F_{\xi^k}(x^k) - x^*\right\rangle\right]}_{T_1} \notag \\
        &&\quad \underbrace{-\frac{1}{2}\Exp_{\xi^k}\left[\gamma_{1,\xi^k}^2\|g^k - F_{\xi^k}(x^*)\|^2\right] + \frac{1}{2}\Exp_{\xi^k}\left[\gamma_{1,\xi^k}^2\|g^k - F_{\xi^k}(x^k) - F_{\xi^k}(x^*)\|^2\right]}_{T_2} \notag\\
        &&\quad -\frac{1}{2}\Exp_{\xi^k}\left[\gamma_{1,\xi^k}^2\|F_{\xi^k}(x^k)\|^2\right]. \label{eq:vnsinjkxaisd}
    \end{eqnarray}
    Next, we upper bound terms $T_1$ and $T_2$. From $(\mu_{\xi^k}, x^*)$-strong monotonicity of $F_{\xi^k}$ we have\footnote{When all $\mu_{\xi} \ge 0$, which is often assumed in the analysis of \algname{S-SEG}, numerical constants in our proof can be tightened. Indeed, in the last step of the derivation below, we can get $\Exp_{\xi^k}\left[\mu_{\xi^k}\gamma_{1,\xi^k}^3\|F_{\xi^k}(x^k)\|^2\right]$ instead of $2\Exp_{\xi^k}\left[|\mu_{\xi^k}|\gamma_{1,\xi^k}^3\|F_{\xi^k}(x^k)\|^2\right]$.}
    \begin{eqnarray*}
        T_1 &\overset{\eqref{eq:F_xi_str_monotonicity}}{\le}& -\Exp_{\xi^k}\left[\mu_{\xi^k}\gamma_{1,\xi^k}\left\|x^k - x^* - \gamma_{1,\xi^k}F_{\xi^k}(x^K)\right\|^2\right]\\
        &=& -\Exp_{\xi^k}\left[\obf_{\{\mu_{\xi^k} \ge 0\}}\mu_{\xi^k}\gamma_{1,\xi^k}\left\|x^k - x^* - \gamma_{1,\xi^k}F_{\xi^k}(x^K)\right\|^2\right]\\
        &&\quad -\Exp_{\xi^k}\left[\obf_{\{\mu_{\xi^k} < 0\}}\mu_{\xi^k}\gamma_{1,\xi^k}\left\|x^k - x^* - \gamma_{1,\xi^k}F_{\xi^k}(x^K)\right\|^2\right]\\
        &\overset{\eqref{eq:a+b_lower},\eqref{eq:a+b}}{\le}& -\frac{1}{2}\Exp_{\xi^k}\left[\obf_{\{\mu_{\xi^k} \ge 0\}}\mu_{\xi^k}\gamma_{1,\xi^k}\right]\|x^k - x^*\|^2 + \Exp_{\xi^k}\left[\obf_{\{\mu_{\xi^k} \ge 0\}}\mu_{\xi^k}\gamma_{1,\xi^k}^3\|F_{\xi^k}(x^k)\|^2\right]\\
        &&\quad -2\Exp_{\xi^k}\left[\obf_{\{\mu_{\xi^k} < 0\}}\mu_{\xi^k}\gamma_{1,\xi^k}\right]\|x^k - x^*\|^2 - 2\Exp_{\xi^k}\left[\obf_{\{\mu_{\xi^k} < 0\}}\mu_{\xi^k}\gamma_{1,\xi^k}^3\|F_{\xi^k}(x^k)\|^2\right]\\
        &\le& -\frac{1}{2}\Exp_{\xi^k}\left[\left(\obf_{\{\mu_{\xi^k} \ge 0\}} + 4\cdot\obf_{\{\mu_{\xi^k} < 0\}}\right)\mu_{\xi^k}\gamma_{1,\xi^k}\right]\|x^k - x^*\|^2 + 2\Exp_{\xi^k}\left[|\mu_{\xi^k}|\gamma_{1,\xi^k}^3\|F_{\xi^k}(x^k)\|^2\right].
    \end{eqnarray*}
    Using simple inequalities \eqref{eq:a+b_lower} and \eqref{eq:a+b} and applying $L_{\xi^k}$-Lipschitzness of $F_{\xi^k}(x)$, we upper bound $T_2$:
    \begin{eqnarray*}
        T_2 &\overset{\eqref{eq:a+b_lower},\eqref{eq:a+b}}{\leq}& -\frac{1}{4}\Exp_{\xi^k}\left[\gamma_{1,\xi^k}^2\|g^k\|^2\right] + \frac{1}{2}\Exp\left[\gamma_{1,\xi^k}^2\|F_{\xi^k}(x^*)\|^2\right]\\
        &&\quad + \Exp_{\xi^k}\left[\gamma_{1,\xi^k}^2\|g^k - F_{\xi^k}(x^k)\|^2\right] + \Exp_{\xi^k}\left[\gamma_{1,\xi^k}^2\|F_{\xi^k}(x^*)\|^2\right]\\
        &\le& \Exp_{\xi^k}\left[\gamma_{1,\xi^k}^2\|g^k - F_{\xi^k}(x^k)\|^2\right] + \frac{3}{2}\Exp_{\xi^k}\left[\gamma_{1,\xi^k}^2\|F_{\xi^k}(x^*)\|^2\right]\\
        &=& \Exp_{\xi^k}\left[\gamma_{1,\xi^k}^2\|F_{\xi^k}(x^k - \gamma_{1,\xi^k}F_{\xi^k}(x^k)) - F_{\xi^k}(x^k)\|^2\right] + \frac{3}{2}\Exp_{\xi^k}\left[\gamma_{1,\xi^k}^2\|F_{\xi^k}(x^*)\|^2\right]\\
        &\overset{\eqref{eq:F_xi_lip}}{\le}& \Exp_{\xi^k}\left[L_{\xi^k}^2\gamma_{1,\xi^k}^4\|F_{\xi^k}(x^k)\|^2\right] + \frac{3}{2}\Exp_{\xi^k}\left[\gamma_{1,\xi^k}^2\|F_{\xi^k}(x^*)\|^2\right].
    \end{eqnarray*}
    Putting all together in \eqref{eq:vnsinjkxaisd}, we derive
    \begin{eqnarray*}
        -\widehat{P}_k &\le& -\frac{1}{2}\Exp_{\xi^k}\left[\left(\obf_{\{\mu_{\xi^k} \ge 0\}} + 4\cdot\obf_{\{\mu_{\xi^k} < 0\}}\right)\mu_{\xi^k}\gamma_{1,\xi^k}\right]\|x^k - x^*\|^2 + \frac{3}{2}\Exp_{\xi^k}\left[\gamma_{1,\xi^k}^2\|F_{\xi^k}(x^*)\|^2\right]\\
        &&\quad - \frac{1}{2}\Exp_{\xi^k}\left[\gamma_{1,\xi^k}^2\left(1 - 4|\mu_{\xi^k}|\gamma_{1,\xi^k} - 2 L_{\xi^k}^2\gamma_{1,\xi^k}^2\right)\|F_{\xi^k}(x^k)\|^2\right],
    \end{eqnarray*}
    where the last term is non-negative due to \eqref{eq:SS_SEG_AS_stepsizes_11}. This finishes the proof.
\end{proof}

Combining two previous lemmas with Theorem~\ref{thm:main_theorem_general_main}, we derive the following result.

\begin{theorem}[Theorem~\ref{thm:SEG_same_sample_convergence_AS_main}]\label{thm:SEG_same_sample_convergence_AS}
	Let Assumptions~\ref{as:F_xi_lip}~and~\ref{as:F_xi_str_monotonicity} hold. If $\gamma_{2,\xi^k} = \alpha \gamma_{1,\xi_k}$, $\alpha > 0$, and $\gamma_{1,\xi^k}$ satisfies \eqref{eq:SS_SEG_AS_stepsizes_1}, \eqref{eq:SS_SEG_AS_stepsizes_2}, and \eqref{eq:SS_SEG_AS_stepsizes_11}, then $g^k = F_{\xi^k}\left(x^k - \gamma_{1,\xi^k} F_{\xi^k}(x^k)\right)$ from \eqref{eq:S_SEG} satisfies Assumption~\ref{as:unified_assumption_general} with the following parameters:
	\begin{gather*}
		A = 2\alpha,\quad C = 0,\quad D_1 = 6\alpha^2\sigma_{\algname{AS}}^2 = 6\alpha^2\Exp_\xi\left[\gamma_{1,\xi}^2\|F_\xi(x^*)\|^2\right],\\
		\rho = \frac{\alpha}{2}\Exp_{\xi^k}[\gamma_{1,\xi^k}\mu_{\xi^k}(\obf_{\{\mu_{\xi^k} \ge 0\}} + 4\cdot\obf_{\{\mu_{\xi^k} < 0\}})],\\
		G_k = \alpha \Exp_{\xi^k}\left[\gamma_{1,\xi^k}^2\left(1 - 4|\mu_{\xi^k}|\gamma_{1,\xi^k} - 2L_{\xi^k}^2\gamma_{1,\xi^k}^2\right)\|F_{\xi^k}(x^k)\|^2\right],\quad B = \frac{1}{2},\quad D_2 = \frac{3\alpha}{2}\sigma_{\algname{AS}}^2.
	\end{gather*}
	If additionally $\alpha \leq \nicefrac{1}{4}$, then for all $K\ge 0$ we have for the case $\rho > 0$
	\begin{equation*}
		\Exp\left[\|x^{K+1} - x^*\|^2\right] \leq \left(1 - \rho\right)\Exp\left[\|x^K - x^*\|^2\right] + \frac{3\alpha}{2}\left(4\alpha + 1\right)\sigma_{\algname{AS}}^2,
	\end{equation*}
	\begin{equation*}
		\Exp\left[\|x^K - x^*\|^2\right] \leq \left(1 - \rho\right)^K\|x^0 - x^*\|^2 + \frac{3\alpha\left(4\alpha + 1\right)\sigma_{\algname{AS}}^2}{2\rho},
	\end{equation*}
	and for the case $\rho = 0$
	\begin{equation*}
	    \frac{1}{K+1}\sum\limits_{k=0}^K\Exp\left[\gamma_{1,\xi^k}^2\left(1 - 4|\mu_{\xi^k}|\gamma_{1,\xi^k} - 2L_{\xi^k}^2\gamma_{1,\xi^k}^2\right)\|F_{\xi^k}(x^k)\|^2\right] \le \frac{2\|x^0 - x^*\|^2}{\alpha(K+1)} + 3(4\alpha +1)\sigma_{\algname{AS}}^2.
	\end{equation*}
\end{theorem}
\begin{proof}
    \algname{S-SEG} fits the unified update rule \eqref{eq:general_method} with $\gamma_{\xi^k} = \gamma_{2,\xi^k}$ and $g^k = F_{\xi^k}\left(x^k - \gamma_{1,\xi^k} F_{\xi^k}(x^k)\right)$. Moreover, Lemmas~\ref{lem:second_moment_bound_S_SEG}~and~\ref{lem:P_k_bound_S_SEG} imply
    \begin{eqnarray}
		\Exp_{\xi^k}\left[\gamma_{1,\xi^k}^2\|g^k\|^2\right] &\leq& 4 \widehat{P}_k + 6 \Exp_{\xi^k}\left[\gamma_{1,\xi^k}^2\|F_{\xi^k}(x^*)\|^2\right],\label{eq:asnjjskskis1}\\
		\widehat{P}_k &\ge& \widehat{\rho}\|x^k - x^*\|^2 + \frac{1}{2}\widehat{G}_k - \frac{3}{2}\Exp_{\xi^k}\left[\gamma_{1,\xi^k}^2\|F_\xi^k(x^*)\|^2\right],\label{eq:asnjjskskis2}
	\end{eqnarray}
	where $\widehat{P}_k = \Exp_{\xi^k}\left[\gamma_{1,\xi^k}\langle g^k, x^k - x^* \rangle\right]$ and
    \begin{eqnarray*}
        \widehat{\rho} = \frac{1}{2}\Exp_{\xi^k}[\gamma_{1,\xi^k}\mu_{\xi^k}(\obf_{\{\mu_{\xi^k} \ge 0\}} + 4\cdot\obf_{\{\mu_{\xi^k} < 0\}})],\\
        \widehat{G}_k = \Exp_{\xi^k}\left[\gamma_{1,\xi^k}^2\left(1 - 4|\mu_{\xi^k}|\gamma_{1,\xi^k} - 2 L_{\xi^k}^2\gamma_{1,\xi^k}^2\right)\|F_{\xi^k}(x^k)\|^2\right].
    \end{eqnarray*}
    Since $\gamma_{\xi^k} = \gamma_{2,\xi^k} = \alpha \gamma_{1,\xi^k}$, we multiply \eqref{eq:asnjjskskis1} by $\alpha^2$ and \eqref{eq:asnjjskskis2} by $\alpha$ and get that Assumption~\ref{as:unified_assumption_general} holds with the parameters given in the statement of the theorem. Applying Theorem~\ref{thm:main_theorem_general_main} we get the result.
\end{proof}

The next corollary establishes the convergence rate with diminishing stepsizes allowing to reduce the size of the neighborhood, when $\rho > 0$.

\begin{corollary}[$\rho > 0$; Corollary~\ref{cor:str_mon_AS_main}]\label{cor:str_mon_AS}
    Let Assumptions~\ref{as:F_xi_lip}~and~\ref{as:F_xi_str_monotonicity} hold, $\gamma_{2,\xi^k} = \alpha \gamma_{1,\xi_k}$, $\alpha = \nicefrac{1}{4}$, $\gamma_{1,\xi^k} = \beta_k \cdot\gamma_{\xi^k}$, and $\gamma_{\xi^k}$ satisfies \eqref{eq:SS_SEG_AS_stepsizes_1}, \eqref{eq:SS_SEG_AS_stepsizes_2}, and \eqref{eq:SS_SEG_AS_stepsizes_11}. Assume that
	\begin{equation}
	    \widetilde{\rho} = \frac{1}{8}\Exp_{\xi^k}[\gamma_{\xi^k}\mu_{\xi^k}(\obf_{\{\mu_{\xi^k} \ge 0\}} + 4\cdot\obf_{\{\mu_{\xi^k} < 0\}})] > 0. \notag
	\end{equation}
	Then, for all $K \ge 0$ and $\{\beta_k\}_{k\ge 0}$ such that
	\begin{eqnarray*}
		\text{if } K \le \frac{1}{\widetilde{\rho}}, && \beta_k = 1,\\
		\text{if } K > \frac{1}{\widetilde{\rho}} \text{ and } k < k_0, && \beta_k = 1,\\
		\text{if } K > \frac{1}{\widetilde{\rho}} \text{ and } k \ge k_0, && \beta_k = \frac{2}{2 + \widetilde{\rho}(k - k_0)},
	\end{eqnarray*}
	for $k_0 = \left\lceil\nicefrac{K}{2} \right\rceil$ we have
	\begin{equation*}
	    \Exp\left[\|x^K - x^*\|^2\right] \le \frac{32\|x^0 - x^*\|^2}{\widetilde{\rho}}\exp\left(-\frac{\widetilde{\rho} K}{2}\right) + \frac{27\sigma_{\algname{AS}}^2}{\widetilde{\rho}^2 K},
	\end{equation*}
	where $\sigma_{\algname{AS}}^2 = \Exp_\xi\left[\gamma_{\xi}^2\|F_\xi(x^*)\|^2\right]$
\end{corollary}
\begin{proof}
    In Theorem~\ref{thm:SEG_same_sample_convergence_AS}, we establish the following recurrence:
    \begin{eqnarray*}
        \Exp\left[\|x^{k+1} - x^*\|^2\right] &\leq& \left(1 - \beta_k\widetilde{\rho}\right)\Exp\left[\|x^k - x^*\|^2\right] + \frac{3\alpha}{2}\left(4\alpha + 1\right)\beta_k^2\sigma_{\algname{AS}}^2\\
        &\overset{\alpha = \nicefrac{1}{4}}{=}& \left(1 - \beta_k\widetilde{\rho}\right)\Exp\left[\|x^k - x^*\|^2\right] + \beta_k^2\frac{3\sigma_{\algname{AS}}^2}{4},
    \end{eqnarray*}
	where we redefined $\rho$ and $\sigma_{\algname{AS}}^2$ to better handle decreasing stepsizes. Applying Lemma~\ref{lem:stich_lemma_for_str_cvx_conv} for $r_k = \Exp\left[\|x^k - x^*\|^2\right]$, $\gamma_k = \beta_k$, $a = \widetilde{\rho}$, $c = \nicefrac{3\sigma_{\algname{AS}}^2}{4}$, $h = 1$, we get the result. 
\end{proof}

When $\rho = 0$, we use large batchiszes to reduce the size of the neighborhood.

\begin{corollary}[$\rho = 0$]\label{cor:mon_AS}
    Let Assumptions~\ref{as:F_xi_lip}~and~\ref{as:F_xi_str_monotonicity} hold, $\gamma_{2,\xi^k} = \alpha \gamma_{1,\xi_k}$, $\alpha = \nicefrac{1}{4}$, and $\gamma_{1,\xi^k}$ satisfies \eqref{eq:SS_SEG_AS_stepsizes_1}-\eqref{eq:SS_SEG_AS_stepsizes_2}, and
	\begin{equation}
	    0 < \gamma_{1,\xi^k} \leq \frac{1}{8|\mu_{\xi^k}| + 2\sqrt{2}L_{\xi^k}}. \notag
	\end{equation}
	Assume that
	\begin{equation}
	    \rho = \frac{1}{8}\Exp_{\xi^k}[\gamma_{1,\xi^k}\mu_{\xi^k}(\obf_{\{\mu_{\xi^k} \ge 0\}} + 4\cdot\obf_{\{\mu_{\xi^k} < 0\}})] = 0, \notag
	\end{equation}
	\begin{equation*}
	    \Exp_{\xi^k}\left[\gamma_{1,\xi^k} F_{\xi^k}(x^k)\right] = \gamma F(x^k)
	\end{equation*}
	for some $\gamma > 0$ and $F_{\xi^k}(x^k)$ is computed via $\cO(b)$ stochastic oracle calls and\footnote{This can be achieved with i.i.d.\ batching from the distribution $\widehat{\cD}$, satisfying Assumptions~\ref{as:F_xi_str_monotonicity}~and~\ref{as:F_xi_lip}.} 
	\begin{equation*}
	    \Exp_{\xi^k \sim \cD}\left[\gamma_{1,\xi^k}^2 \|F_{\xi^k}(x^*)\|^2\right] \le \frac{1}{b}\Exp_{\xi^k \sim \widehat{\cD}}\left[\gamma_{1,\xi^k}^2 \|F_{\xi^k}(x^*)\|^2\right] = \frac{\sigma_{\algname{AS}}^2}{b},
	\end{equation*}
	where $\widehat{\cD}$ satisfies Assumptions~\ref{as:F_xi_lip}~and~\ref{as:F_xi_str_monotonicity}.
	Then, for all $K \ge 0$ we have
	\begin{equation*}
	    \frac{1}{K+1}\sum\limits_{k=0}^K\Exp\left[\|F(x^k)\|^2\right] \le \frac{16\|x^0 - x^*\|^2}{\gamma^2(K+1)} + \frac{12\sigma_{\algname{AS}}^2}{\gamma^2 b},
	\end{equation*}
	and each iteration requires $\cO(b)$ stochastic oracle calls.
\end{corollary}
\begin{proof}
    Theorem~\ref{thm:SEG_same_sample_convergence_AS} implies that
    \begin{eqnarray*}
        \frac{1}{K+1}\sum\limits_{k=0}^K\Exp\left[\gamma_{1,\xi^k}^2\left(1 - 4|\mu_{\xi^k}|\gamma_{1,\xi^k} - 2L_{\xi^k}^2\gamma_{1,\xi^k}^2\right)\|F_{\xi^k}(x^k)\|^2\right] &\\
        &\hspace{-4cm}\le\frac{2\|x^0 - x^*\|^2}{\alpha(K+1)}+ 3(4\alpha +1)\Exp_{\xi^k \sim \cD}\left[\gamma_{1,\xi^k}^2 \|F_{\xi^k}(x^*)\|^2\right]\\
        &\hspace{-8.9cm}\overset{\alpha = \nicefrac{1}{4}}{\le} \frac{8\|x^0 - x^*\|^2}{K+1} + \frac{6\sigma_{\algname{AS}}^2}{b}.
    \end{eqnarray*}
    Since 
    \begin{equation}
	    0 < \gamma_{1,\xi^k} \leq \frac{1}{8|\mu_{\xi^k}| + 2\sqrt{2}L_{\xi^k}}, \notag
	\end{equation}
	we have
	\begin{eqnarray*}
        \frac{1}{2(K+1)}\sum\limits_{k=0}^K\Exp\left[\gamma_{1,\xi^k}^2\|F_{\xi^k}(x^k)\|^2\right] &\le& \frac{8\|x^0 - x^*\|^2}{K+1} + \frac{6\sigma_{\algname{AS}}^2}{b}.
    \end{eqnarray*}
    Finally, we use Jensen's inequality and $\Exp_{\xi^k}\left[\gamma_{1,\xi^k} F_{\xi^k}(x^k)\right] = \gamma F(x^k)$:
    \begin{eqnarray*}
        \frac{\gamma^2}{2(K+1)}\sum\limits_{k=0}^K\Exp\left[\|F(x^k)\|^2\right] &=& \frac{1}{2(K+1)}\sum\limits_{k=0}^K\Exp\left[\left\|\Exp_{\xi^k}\left[\gamma_{1,\xi^k}F_{\xi^k}(x^k)\right]\right\|^2\right]\\
        &\le& \frac{1}{2(K+1)}\sum\limits_{k=0}^K\Exp\left[\Exp_{\xi^k}\left[\left\|\gamma_{1,\xi^k}F_{\xi^k}(x^k)\right\|^2\right]\right]\\
        &=& \frac{1}{2(K+1)}\sum\limits_{k=0}^K\Exp\left[\gamma_{1,\xi^k}^2\|F_{\xi^k}(x^k)\|^2\right]\\ &\le& \frac{8\|x^0 - x^*\|^2}{K+1} + \frac{6\sigma_{\algname{AS}}^2}{b}.
    \end{eqnarray*}
    Multiplying the inequality by $\nicefrac{2}{\gamma^2}$, we get the result.
\end{proof}

\subsection{\algname{S-SEG} with Uniform Sampling (\algname{S-SEG-US})}

\begin{theorem}\label{thm:SEG_same_sample_convergence_US}
	Consider the setup from Example~\ref{ex:uniform_sampling}. If $\gamma_{2,\xi^k} = \alpha \gamma_{1,\xi_k}$, $\alpha > 0$, and $\gamma_{1,\xi^k} = \gamma \leq \nicefrac{1}{6L_{\max}}$, where $L_{\max} = \max_{i\in [n]}L_i$, then $g^k = F_{\xi^k}\left(x^k - \gamma_{1,\xi^k} F_{\xi^k}(x^k)\right)$ from \eqref{eq:S_SEG} satisfies Assumption~\ref{as:unified_assumption_general} with the following parameters:
	\begin{gather*}
		A = 2\alpha,\quad C = 0,\quad D_1 = 6\alpha^2\gamma^2\sigma_{\algname{US*}}^2 = \frac{6\alpha^2\gamma^2}{n}\sum\limits_{i=1}^n\|F_i(x^*)\|^2,\quad \rho = \frac{\alpha\gamma\overline{\mu}}{2},\\
		G_k = \frac{\alpha\gamma^2}{n}\sum\limits_{i=1}^n \left(1 - 4|\mu_{i}|\gamma - 2L_{i}^2\gamma^2\right)\|F_{i}(x^k)\|^2,\quad B = \frac{1}{2},\quad D_2 = \frac{3\alpha\gamma^2}{2}\sigma_{\algname{US*}}^2.
	\end{gather*}
	If additionally $\alpha \leq \nicefrac{1}{4}$, then for all $K\ge 0$ we have for the case $\overline{\mu} > 0$
	\begin{equation*}
		\Exp\left[\|x^{K+1} - x^*\|^2\right] \leq \left(1 - \frac{\alpha\gamma\overline{\mu}}{2}\right)\Exp\left[\|x^K - x^*\|^2\right] + \frac{3\alpha}{2}\left(4\alpha + 1\right)\gamma^2\sigma_{\algname{US*}}^2,
	\end{equation*}
	\begin{equation*}
		\Exp\left[\|x^K - x^*\|^2\right] \leq \left(1 - \frac{\alpha\gamma\overline{\mu}}{2}\right)^K\|x^0 - x^*\|^2 + \frac{3\left(4\alpha + 1\right)\gamma\sigma_{\algname{US*}}^2}{\overline{\mu}},
	\end{equation*}
	and for the case $\overline{\mu} = 0$
	\begin{equation*}
	    \frac{1}{K+1}\sum\limits_{k=0}^K\Exp\left[\frac{1}{n}\sum\limits_{i=1}^n \left(1 - 4|\mu_{i}|\gamma - 2L_{i}^2\gamma^2\right)\|F_{i}(x^k)\|^2\right] \le \frac{2\|x^0 - x^*\|^2}{\alpha\gamma^2(K+1)} + 3(4\alpha +1)\sigma_{\algname{US*}}^2.
	\end{equation*}
\end{theorem}
\begin{proof}
    Since $\gamma \le \nicefrac{1}{6L_{\max}}$ and $|\mu_i| \le L_i$, condition \eqref{eq:SS_SEG_AS_stepsizes_11} is satisfied. In Example~\ref{ex:independent_sampl_with_repl}, we show that conditions \eqref{eq:SS_SEG_AS_stepsizes_1} and \eqref{eq:SS_SEG_AS_stepsizes_2} hold as well. Therefore, Theorem~\ref{thm:SEG_same_sample_convergence_AS} implies the desired result with
    \begin{eqnarray*}
        \sigma_{\algname{AS}}^2 &=& \Exp_\xi\left[\gamma_{1,\xi}^2\|F_\xi(x^*)\|^2\right] = \frac{\gamma^2}{n}\sum_{i=1}^n \|F_i(x^*)\|^2 = \gamma^2 \sigma_{\algname{US}*}^2,\\
        \rho &=& \frac{\alpha}{2}\Exp_{\xi^k}[\gamma_{1,\xi^k}\mu_{\xi^k}(\obf_{\{\mu_{\xi^k} \ge 0\}} + 4\cdot\obf_{\{\mu_{\xi^k} < 0\}})] = \frac{\alpha\gamma}{2n}\left(\sum\limits_{i: \mu_i \geq 0} \mu_i + 4\sum\limits_{i: \mu_i < 0} \mu_i\right) = \frac{\alpha\gamma\overline{\mu}}{2},\\
        G_k &=& \alpha\Exp_{\xi^k}\left[\gamma_{1,\xi^k}^2\left(1 - 4|\mu_{\xi^k}|\gamma_{1,\xi^k} - 2 L_{\xi^k}^2\gamma_{1,\xi^k}^2\right)\|F_{\xi^k}(x^k)\|^2\right]\\
        &=& \frac{\alpha\gamma^2}{n}\sum\limits_{i=1}^n \left(1 - 4|\mu_{i}|\gamma - 2L_{i}^2\gamma^2\right)\|F_{i}(x^k)\|^2.
    \end{eqnarray*}
\end{proof}

\begin{corollary}[$\overline{\mu} > 0$]\label{cor:str_mon_US}
    Consider the setup from Example~\ref{ex:uniform_sampling}. Let $\overline{\mu} > 0$, $\gamma_{2,\xi^k} = \alpha \gamma_{1,\xi_k}$, $\alpha = \nicefrac{1}{4}$, and $\gamma_{1,\xi^k} = \beta_k\gamma = \nicefrac{\beta_k}{6L_{\max}}$, where $L_{\max} = \max_{i\in [n]}L_i$ and $0 < \beta_k \leq 1$.
	Then, for all $K \ge 0$ and $\{\beta_k\}_{k\ge 0}$ such that
	\begin{eqnarray*}
		\text{if } K \le \frac{48L_{\max}}{\overline{\mu}}, && \beta_k = 1,\\
		\text{if } K > \frac{48L_{\max}}{\overline{\mu}} \text{ and } k < k_0, && \beta_k = 1,\\
		\text{if } K > \frac{48L_{\max}}{\overline{\mu}} \text{ and } k \ge k_0, && \beta_k = \frac{96L_{\max}}{96L_{\max} + \overline{\mu}(k - k_0)},
	\end{eqnarray*}
	for $k_0 = \left\lceil\nicefrac{K}{2} \right\rceil$ we have
	\begin{equation*}
	    \Exp\left[\|x^K - x^*\|^2\right] \le \frac{1536 L_{\max}\|x^0 - x^*\|^2}{\overline{\mu}}\exp\left(-\frac{\overline{\mu} K}{96 L_{\max}}\right) + \frac{1728\sigma_{\algname{US*}}^2}{\overline{\mu}^2 K}.
	\end{equation*}
\end{corollary}
\begin{proof}
    Corollary~\ref{cor:str_mon_AS} implies the needed result with
    \begin{eqnarray*}
        \widetilde{\rho} &=& \frac{1}{8}\Exp_{\xi^k}[\gamma_{\xi^k}\mu_{\xi^k}(\obf_{\{\mu_{\xi^k} \ge 0\}} + 4\cdot\obf_{\{\mu_{\xi^k} < 0\}})] = \frac{\gamma}{8n}\left(\sum\limits_{i: \mu_i \geq 0} \mu_i + 4\sum\limits_{i: \mu_i < 0} \mu_i\right) = \frac{\overline{\mu}}{48L_{\max}},\\
        \sigma_{\algname{AS}}^2 &=& \Exp_\xi\left[\gamma_{\xi}^2\|F_\xi(x^*)\|^2\right] = \frac{\gamma^2}{n}\sum_{i=1}^n \|F_i(x^*)\|^2 = \gamma^2 \sigma_{\algname{US}*}^2.
    \end{eqnarray*}
\end{proof}

\begin{corollary}[$\overline{\mu} = 0$]\label{cor:mon_US}
    Consider the setup from Example~\ref{ex:uniform_sampling}. Let $\overline{\mu} = 0$, $\gamma_{2,\xi^k} = \alpha \gamma_{1,\xi_k}$, $\alpha = \nicefrac{1}{4}$, and $\gamma_{1,\xi^k} = \gamma \leq \nicefrac{1}{6L_{\max}}$, where $L_{\max} = \max_{i\in [n]}L_i$. Assume that
    \begin{equation*}
        F_{\xi^k}(x^k) = \frac{1}{b}\sum\limits_{i=1}^b F_{\xi_i^k}(x),
    \end{equation*}
    where $\xi_1^k,\ldots,\xi_b^k$ are i.i.d.\ samples from the uniform distribution on $[n]$.
	Then, for all $K \ge 0$ we have
	\begin{equation*}
	    \frac{1}{K+1}\sum\limits_{k=0}^K\Exp\left[\|F(x^k)\|^2\right] \le \frac{16\|x^0 - x^*\|^2}{\gamma^2(K+1)} + \frac{12\sigma_{\algname{US*}}^2}{b},
	\end{equation*}
	and each iteration requires $\cO(b)$ stochastic oracle calls.
\end{corollary}
\begin{proof}
    Since
    \begin{eqnarray*}
        \Exp_{\xi^k}\left[\gamma_{1,\xi^k} F_{\xi^k}(x^k)\right] &=& \frac{\gamma}{n}\sum\limits_{i=1}^n F_i(x^k) = \gamma F(x^k), 
    \end{eqnarray*}
    Corollary~\ref{cor:mon_AS} implies the needed result with
    \begin{eqnarray*}
        \sigma_{\algname{AS}}^2 &=& \Exp_\xi\left[\gamma_{\xi}^2\|F_\xi(x^*)\|^2\right] = \frac{\gamma^2}{n}\sum_{i=1}^n \|F_i(x^*)\|^2 = \gamma^2 \sigma_{\algname{US}*}^2.
    \end{eqnarray*}
\end{proof}

\subsection{\algname{S-SEG} with $b$-Nice Sampling (\algname{S-SEG-NICE})}

\begin{theorem}\label{thm:SEG_same_sample_convergence_NICE}
	Consider the setup from Example~\ref{ex:nice}. If $\gamma_{2,\xi^k} = \alpha \gamma_{1,\xi_k}$, $\alpha > 0$, and $\gamma_{1,\xi^k} = \gamma \leq \nicefrac{1}{6L_{b-\algname{NICE}}}$, where $L_{b-\algname{NICE}} = \max_{S\subseteq [n], |S|=b}L_S$ and $L_S$ is the Lipschitz constant of $F_S(x) = \frac{1}{|S|}\sum_{i=1}^n F_i(x)$, then $g^k = F_{\xi^k}\left(x^k - \gamma_{1,\xi^k} F_{\xi^k}(x^k)\right)$ from \eqref{eq:S_SEG} satisfies Assumption~\ref{as:unified_assumption_general} with the following parameters:
	\begin{gather*}
		A = 2\alpha,\quad C = 0,\quad D_1 = 6\alpha^2\gamma^2\sigma_{b-\algname{NICE*}}^2 = \frac{6\alpha^2\gamma^2}{\binom{n}{b}}\sum\limits_{S\subseteq [n], |S| = b}\|F_S(x^*)\|^2,\quad \rho = \frac{\alpha\gamma\overline{\mu}_{b-\algname{NICE}}}{2},\\
		G_k = \frac{\alpha\gamma^2}{\binom{n}{b}}\sum\limits_{S\subseteq [n], |S| = b} \left(1 - 4|\mu_{S}|\gamma - 2L_{S}^2\gamma^2\right)\|F_{S}(x^k)\|^2,\quad B = \frac{1}{2},\quad D_2 = \frac{3\alpha\gamma^2}{2}\sigma_{b-\algname{NICE}*}^2.
	\end{gather*}
	If additionally $\alpha \leq \nicefrac{1}{4}$, then for all $K\ge 0$ we have for the case $\overline{\mu}_{b-\algname{NICE}} > 0$
	\begin{equation*}
		\Exp\left[\|x^{K+1} - x^*\|^2\right] \leq \left(1 - \frac{\alpha\gamma\overline{\mu}_{b-\algname{NICE}}}{2}\right)\Exp\left[\|x^K - x^*\|^2\right] + \frac{3\alpha}{2}\left(4\alpha + 1\right)\gamma^2\sigma_{b-\algname{NICE}*}^2,
	\end{equation*}
	\begin{equation*}
		\Exp\left[\|x^K - x^*\|^2\right] \leq \left(1 - \frac{\alpha\gamma\overline{\mu}_{b-\algname{NICE}}}{2}\right)^K\|x^0 - x^*\|^2 + \frac{3\left(4\alpha + 1\right)\gamma\sigma_{b-\algname{NICE}*}^2}{\overline{\mu}},
	\end{equation*}
	and for the case $\overline{\mu}_{b-\algname{NICE}} = 0$
	\begin{equation*}
	    \frac{1}{K+1}\sum\limits_{k=0}^K\Exp\left[\frac{1}{\binom{n}{b}}\sum\limits_{S\subseteq [n], |S| = b} \left(1 - 4|\mu_{S}|\gamma - 2L_{S}^2\gamma^2\right)\|F_{S}(x^k)\|^2\right] \le \frac{2\|x^0 - x^*\|^2}{\alpha\gamma^2(K+1)} + 3(4\alpha +1)\sigma_{b-\algname{NICE}*}^2.
	\end{equation*}
\end{theorem}
\begin{proof}
    Since $\gamma \le \nicefrac{1}{6L_{b-\algname{NICE}}}$ and $|\mu_S| \le L_S$ for all $S\subseteq [n]$, condition \eqref{eq:SS_SEG_AS_stepsizes_11} is satisfied. In Example~\ref{ex:nice_full}, we show that conditions \eqref{eq:SS_SEG_AS_stepsizes_1} and \eqref{eq:SS_SEG_AS_stepsizes_2} hold as well. Therefore, Theorem~\ref{thm:SEG_same_sample_convergence_AS} implies the desired result with
    \begin{eqnarray*}
        \sigma_{\algname{AS}}^2 &=& \Exp_\xi\left[\gamma_{1,\xi}^2\|F_\xi(x^*)\|^2\right] = \frac{\gamma^2}{\binom{n}{b}}\sum_{S\subseteq [n]} \|F_S(x^*)\|^2 = \gamma^2 \sigma_{b-\algname{NICE}*}^2,\\
        \rho &=& \frac{\alpha}{2}\Exp_{\xi^k}[\gamma_{1,\xi^k}\mu_{\xi^k}(\obf_{\{\mu_{\xi^k} \ge 0\}} + 4\cdot\obf_{\{\mu_{\xi^k} < 0\}})] = \frac{\alpha\gamma}{2\binom{n}{b}}\left(\sum\limits_{\overset{S\subseteq [n],}{|S| = b :\mu_{S} \ge 0}}\mu_{S} + 4\sum\limits_{\overset{S\subseteq [n],}{|S| = b :\mu_{S} < 0}}\mu_{S}\right)\\
        &=& \frac{\alpha\gamma\overline{\mu}_{b-\algname{NICE}}}{2},\\
        G_k &=& \alpha\Exp_{\xi^k}\left[\gamma_{1,\xi^k}^2\left(1 - 4|\mu_{\xi^k}|\gamma_{1,\xi^k} - 2 L_{\xi^k}^2\gamma_{1,\xi^k}^2\right)\|F_{\xi^k}(x^k)\|^2\right]\\
        &=& \frac{\alpha\gamma^2}{\binom{n}{b}}\sum\limits_{S\subseteq [n], |S| = b} \left(1 - 4|\mu_{S}|\gamma - 2L_{S}^2\gamma^2\right)\|F_{S}(x^k)\|^2.
    \end{eqnarray*}
\end{proof}

\begin{remark}\label{rem:b_nice}
    {\small We notice that
    \begin{eqnarray}
        L_{b-\algname{NICE}} &=& \max\limits_{S\subseteq [n], |S|=b}L_S \le \max\limits_{S\subseteq [n], |S|=b}\frac{1}{b}\sum_{i\in S}L_i \le \max\limits_{i\in [n]}L_i = L_{\max},\notag\\
        \mu_{b-\algname{NICE}} &=& \frac{1}{\binom{n}{b}}\left(\sum\limits_{\overset{S\subseteq [n],}{|S| = b :\mu_{S} \ge 0}}\mu_{S} + 4\sum\limits_{\overset{S\subseteq [n],}{|S| = b :\mu_{S} < 0}}\mu_{S}\right) \ge \frac{1}{\binom{n}{b}}\left(\sum\limits_{\overset{S\subseteq [n],}{|S| = b :\mu_{S} \ge 0}}\frac{1}{b}\sum\limits_{i\in S}\mu_i + 4\sum\limits_{\overset{S\subseteq [n],}{|S| = b :\mu_{S} < 0}}\frac{1}{b}\sum\limits_{i\in S}\mu_i\right)\notag\\
        &=& \frac{1}{\binom{n}{b}}\sum\limits_{\overset{S\subseteq [n],}{|S| = b :\mu_{S} \ge 0}}\frac{1}{b}\left(\sum\limits_{i\in S:\mu_i \ge 0}\mu_i + \sum\limits_{i\in S:\mu_i < 0}\mu_i\right) + \frac{4}{\binom{n}{b}}\sum\limits_{\overset{S\subseteq [n],}{|S| = b :\mu_{S} < 0}}\frac{1}{b}\left(\sum\limits_{i\in S:\mu_i \ge 0}\mu_i + \sum\limits_{i\in S:\mu_i < 0}\mu_i\right)\notag\\
        &\ge& \frac{1}{\binom{n}{b}}\sum\limits_{\overset{S\subseteq [n],}{|S| = b :\mu_{S} \ge 0}}\frac{1}{b}\left(\sum\limits_{i\in S:\mu_i \ge 0}\mu_i + \sum\limits_{i\in S:\mu_i < 0}4\mu_i\right)\notag \\
        &&\quad + \frac{1}{\binom{n}{b}}\sum\limits_{\overset{S\subseteq [n],}{|S| = b :\mu_{S} < 0}}\frac{1}{b}\left(\sum\limits_{i\in S:\mu_i \ge 0}\mu_i + \sum\limits_{i\in S:\mu_i < 0}4\mu_i\right)\notag\\
        &=& \frac{1}{\binom{n}{b}}\sum\limits_{S\subseteq [n], |S| =  b }\frac{1}{b}\left(\sum\limits_{i\in S:\mu_i \ge 0}\mu_i + \sum\limits_{i\in S:\mu_i < 0}4\mu_i\right) = \frac{\binom{n-1}{b-1}}{b\cdot\binom{n}{b}}\left(\sum\limits_{i:\mu_i \ge 0}\mu_i + \sum\limits_{i:\mu_i < 0}4\mu_i\right)\notag\\
        &=& \frac{1}{n}\sum\limits_{i: \mu_i \geq 0} \mu_i + \frac{4}{n}\sum\limits_{i: \mu_i < 0} \mu_i = \overline{\mu},\notag\\
        \sigma_{b-\algname{NICE}*}^2 &=& \frac{1}{\binom{n}{b}}\sum_{S\subseteq [n]} \|F_S(x^*)\|^2 = \frac{1}{\binom{n}{b}}\sum_{S\subseteq [n]} \left\|\frac{1}{b}\sum\limits_{i\in S}F_i(x^*)\right\|^2\notag\\
        &=& \frac{1}{b^2\cdot\binom{n}{b}}\sum_{S\subseteq [n]} \left(\sum\limits_{i\in S}\|F_i(x^*)\|^2 + 2\sum\limits_{i,j\in S, i< j}\langle F_i(x^*), F_j(x^*)\rangle\right)\notag\\
        &=& \frac{1}{b^2\cdot\binom{n}{b}}\left(\binom{n-1}{b-1}\sum\limits_{i=1}^n\|F_i(x^*)\|^2 + 2\binom{n-2}{b-2}\sum\limits_{1\le i < j \le n}\langle F_i(x^*), F_j(x^*)\rangle\right)\notag\\
        &=& \frac{\binom{n-1}{b-1} - \binom{n-2}{b-2}}{b^2\cdot\binom{n}{b}}\sum\limits_{i=1}^n\|F_i(x^*)\|^2 + \frac{\binom{n-2}{b-2}}{b^2\cdot\binom{n}{b}}\left\|\sum\limits_{i=1}^n F_i(x^*)\right\|^2 = \frac{n-b}{bn(n-1)}\sum\limits_{i=1}^n \|F_i(x^*)\|^2\notag\\
        &=& \frac{n-b}{b(n-1)}\sigma_{\algname{US}*}^2. \label{eq:ncajxkancsjhdbvuj}
    \end{eqnarray}
    Therefore, \algname{S-SEG-NICE} converges faster to the smaller neighborhood than \algname{S-SEG-US}. Moreover, the size of the neighborhood $\sigma_{b-\algname{NICE}*}^2$ is smaller than $\nicefrac{\sigma_{\algname{US}*}^2}{b}$, which corresponds to the variance in the case of i.i.d.\ sampling (Example~\ref{ex:independent_sampl_with_repl}).}
\end{remark}

\begin{corollary}[$\overline{\mu}_{b-\algname{NICE}} > 0$]\label{cor:str_mon_NICE}
    	Consider the setup from Example~\ref{ex:nice}. If $\gamma_{2,\xi^k} = \alpha \gamma_{1,\xi_k}$, $\alpha = \nicefrac{1}{4}$, and $\gamma_{1,\xi^k} = \beta_k\gamma = \nicefrac{\beta_k}{6L_{b-\algname{NICE}}}$, where $L_{b-\algname{NICE}} = \max_{S\subseteq [n], |S|=b}L_S$, $L_S$ is the Lipschitz constant of $F_S(x) = \frac{1}{|S|}\sum_{i=1}^n F_i(x)$, and $0 < \beta_k \leq 1$.
	Then, for all $K \ge 0$ and $\{\beta_k\}_{k\ge 0}$ such that
	\begin{eqnarray*}
		\text{if } K \le \frac{48L_{b-\algname{NICE}}}{\overline{\mu}_{b-\algname{NICE}}}, && \beta_k = 1,\\
		\text{if } K > \frac{48L_{b-\algname{NICE}}}{\overline{\mu}_{b-\algname{NICE}}} \text{ and } k < k_0, && \beta_k = 1,\\
		\text{if } K > \frac{48L_{b-\algname{NICE}}}{\overline{\mu}_{b-\algname{NICE}}} \text{ and } k \ge k_0, && \beta_k = \frac{96L_{b-\algname{NICE}}}{96L_{b-\algname{NICE}} + \overline{\mu}_{b-\algname{NICE}}(k - k_0)},
	\end{eqnarray*}
	for $k_0 = \left\lceil\nicefrac{K}{2} \right\rceil$ we have
	\begin{equation*}
	    \Exp\left[\|x^K - x^*\|^2\right] \le \frac{1536 L_{b-\algname{NICE}}\|x^0 - x^*\|^2}{\overline{\mu}_{b-\algname{NICE}}}\exp\left(-\frac{\overline{\mu}_{b-\algname{NICE}} K}{96 L_{b-\algname{NICE}}}\right) + \frac{1728(n-b)\sigma_{\algname{US}*}^2}{\overline{\mu}_{b-\algname{NICE}}^2 K}.
	\end{equation*}
\end{corollary}
\begin{proof}
    Corollary~\ref{cor:str_mon_AS} implies the needed result with
    \begin{eqnarray*}
        \widetilde{\rho} &=& \frac{1}{8}\Exp_{\xi^k}[\gamma_{\xi^k}\mu_{\xi^k}(\obf_{\{\mu_{\xi^k} \ge 0\}} + 4\cdot\obf_{\{\mu_{\xi^k} < 0\}})] = \frac{\gamma}{8\binom{n}{b}}\left(\sum\limits_{\overset{S\subseteq [n],}{|S| = b :\mu_{S} \ge 0}}\mu_{S} + 4\sum\limits_{\overset{S\subseteq [n],}{|S| = b :\mu_{S} < 0}}\mu_{S}\right)\\
        &=& \frac{\overline{\mu}_{b-\algname{NICE}}}{48L_{b-\algname{NICE}}},\\
        \sigma_{\algname{AS}}^2 &=& \Exp_\xi\left[\gamma_{\xi}^2\|F_\xi(x^*)\|^2\right] = \frac{\gamma^2}{\binom{n}{b}}\sum_{S\subseteq [n]} \|F_S(x^*)\|^2 = \gamma^2 \sigma_{b-\algname{NICE}*}^2 \overset{\eqref{eq:ncajxkancsjhdbvuj}}{=} \frac{n-b}{b(n-1)}\sigma_{\algname{US}*}^2.
    \end{eqnarray*}
\end{proof}

\subsection{\algname{S-SEG} with Importance Sampling (\algname{S-SEG-IS})}

\begin{theorem}\label{thm:SEG_same_sample_convergence_IS}
	Consider the setup from Example~\ref{ex:importance_sampling}. If $\gamma_{2,\xi^k} = \alpha \gamma_{1,\xi_k}$, $\alpha > 0$, and $\gamma_{1,\xi^k} = \nicefrac{\gamma \overline{L}}{L_{\xi^k}}$, $\gamma \leq \nicefrac{1}{6\overline{L}}$, where $\overline{L} = \frac{1}{n}\sum_{i=1}^n L_i$, then $g^k = F_{\xi^k}\left(x^k - \gamma_{1,\xi^k} F_{\xi^k}(x^k)\right)$ from \eqref{eq:S_SEG} satisfies Assumption~\ref{as:unified_assumption_general} with the following parameters:
	\begin{gather*}
		A = 2\alpha,\quad C = 0,\quad D_1 = 6\alpha^2\gamma^2\sigma_{\algname{IS*}}^2 = \frac{6\alpha^2\gamma^2}{n}\sum\limits_{i=1}^n\frac{\overline{L}}{L_i}\|F_i(x^*)\|^2,\quad \rho = \frac{\alpha\gamma\overline{\mu}}{2},\\
		G_k = \frac{\alpha\gamma^2}{n}\sum\limits_{i=1}^n \frac{\overline{L}}{L_i}\left(1 - 4\frac{|\mu_{i}|}{L_i}\overline{L}\gamma - 2\overline{L}^2\gamma^2\right)\|F_{i}(x^k)\|^2,\quad B = \frac{1}{2},\quad D_2 = \frac{3\alpha\gamma^2}{2}\sigma_{\algname{IS*}}^2.
	\end{gather*}
	If additionally $\alpha \leq \nicefrac{1}{4}$, then for all $K\ge 0$ we have for the case $\overline{\mu} > 0$
	\begin{equation*}
		\Exp\left[\|x^{K+1} - x^*\|^2\right] \leq \left(1 - \frac{\alpha\gamma\overline{\mu}}{2}\right)\Exp\left[\|x^K - x^*\|^2\right] + \frac{3\alpha}{2}\left(4\alpha + 1\right)\gamma^2\sigma_{\algname{IS*}}^2,
	\end{equation*}
	\begin{equation*}
		\Exp\left[\|x^K - x^*\|^2\right] \leq \left(1 - \frac{\alpha\gamma\overline{\mu}}{2}\right)^K\|x^0 - x^*\|^2 + \frac{3\left(4\alpha + 1\right)\gamma\sigma_{\algname{IS*}}^2}{\overline{\mu}},
	\end{equation*}
	and for the case $\overline{\mu} = 0$
	\begin{equation*}
	    \frac{1}{K+1}\sum\limits_{k=0}^K\Exp\left[\frac{1}{n}\sum\limits_{i=1}^n \frac{\overline{L}}{L_i}\left(1 - 4\frac{|\mu_{i}|}{L_i}\overline{L}\gamma - 2\overline{L}^2\gamma^2\right)\|F_{i}(x^k)\|^2\right] \le \frac{2\|x^0 - x^*\|^2}{\alpha\gamma^2(K+1)} + 3(4\alpha +1)\sigma_{\algname{IS*}}^2.
	\end{equation*}
\end{theorem}
\begin{proof}
    Since $\gamma \le \nicefrac{1}{6\overline{L}}$ and $|\mu_i| \le L_i$, condition \eqref{eq:SS_SEG_AS_stepsizes_11} is satisfied. In Example~\ref{ex:independent_sampl_with_repl}, we show that conditions \eqref{eq:SS_SEG_AS_stepsizes_1} and \eqref{eq:SS_SEG_AS_stepsizes_2} hold as well. Therefore, Theorem~\ref{thm:SEG_same_sample_convergence_AS} implies the desired result with
    \begin{eqnarray*}
        \sigma_{\algname{AS}}^2 &=& \Exp_\xi\left[\gamma_{1,\xi}^2\|F_\xi(x^*)\|^2\right] = \frac{\gamma^2}{n}\sum_{i=1}^n \frac{\overline{L}}{L_i} \|F_i(x^*)\|^2 = \gamma^2 \sigma_{\algname{IS}*}^2,\\
        \rho &=& \frac{\alpha}{2}\Exp_{\xi^k}[\gamma_{1,\xi^k}\mu_{\xi^k}(\obf_{\{\mu_{\xi^k} \ge 0\}} + 4\cdot\obf_{\{\mu_{\xi^k} < 0\}})] = \frac{\alpha}{2} \sum\limits_{i=1}^n \frac{\gamma \overline{L}}{L_i}\mu_{i}(\obf_{\{\mu_{i} \ge 0\}} + 4\cdot\obf_{\{\mu_{i} < 0\}})\cdot\frac{L_i}{n\overline{L}}\\
        &=& \frac{\alpha\gamma}{2n}\left(\sum\limits_{i: \mu_i \geq 0} \mu_i + 4\sum\limits_{i: \mu_i < 0} \mu_i\right) = \frac{\alpha\gamma\overline{\mu}}{2},\\
        G_k &=& \alpha\Exp_{\xi^k}\left[\gamma_{1,\xi^k}^2\left(1 - 4|\mu_{\xi^k}|\gamma_{1,\xi^k} - 2 L_{\xi^k}^2\gamma_{1,\xi^k}^2\right)\|F_{\xi^k}(x^k)\|^2\right]\\
        &=& \frac{\alpha\gamma^2}{n}\sum\limits_{i=1}^n \frac{\overline{L}}{L_i}\left(1 - 4\frac{|\mu_{i}|}{L_i}\overline{L}\gamma - 2\overline{L}^2\gamma^2\right)\|F_{i}(x^k)\|^2.
    \end{eqnarray*}
\end{proof}

\begin{corollary}[$\overline{\mu} > 0$]\label{cor:str_mon_IS}
    Consider the setup from Example~\ref{ex:importance_sampling}. Let $\overline{\mu} > 0$, $\gamma_{2,\xi^k} = \alpha \gamma_{1,\xi_k}$, $\alpha = \nicefrac{1}{4}$, and $\gamma_{1,\xi^k} = \nicefrac{\beta_k\gamma \overline{L}}{L_{\xi^k}} = \nicefrac{\beta_k}{6L_{\xi^k}}$, where $\overline{L} = \frac{1}{n}\sum_{i=1}^n L_i$ and $0 < \beta_k \leq 1$.
	Then, for all $K \ge 0$ and $\{\beta_k\}_{k\ge 0}$ such that
	\begin{eqnarray*}
		\text{if } K \le \frac{48\overline{L}}{\overline{\mu}}, && \beta_k = 1,\\
		\text{if } K > \frac{48\overline{L}}{\overline{\mu}} \text{ and } k < k_0, && \beta_k = 1,\\
		\text{if } K > \frac{48\overline{L}}{\overline{\mu}} \text{ and } k \ge k_0, && \beta_k = \frac{96\overline{L}}{96\overline{L} + \overline{\mu}(k - k_0)},
	\end{eqnarray*}
	for $k_0 = \left\lceil\nicefrac{K}{2} \right\rceil$ we have
	\begin{equation*}
	    \Exp\left[\|x^K - x^*\|^2\right] \le \frac{1536 \overline{L}\|x^0 - x^*\|^2}{\overline{\mu}}\exp\left(-\frac{\overline{\mu} K}{96 \overline{L}}\right) + \frac{1728\sigma_{\algname{IS*}}^2}{\overline{\mu}^2 K}.
	\end{equation*}
\end{corollary}
\begin{proof}
    Corollary~\ref{cor:str_mon_AS} implies the needed result with
    \begin{eqnarray*}
        \widetilde{\rho} &=& \frac{1}{8}\Exp_{\xi^k}[\gamma_{\xi^k}\mu_{\xi^k}(\obf_{\{\mu_{\xi^k} \ge 0\}} + 4\cdot\obf_{\{\mu_{\xi^k} < 0\}})] = \frac{\gamma}{8n}\left(\sum\limits_{i: \mu_i \geq 0} \mu_i + 4\sum\limits_{i: \mu_i < 0} \mu_i\right) = \frac{\overline{\mu}}{48\overline{L}},\\
        \sigma_{\algname{AS}}^2 &=& \Exp_\xi\left[\gamma_{\xi}^2\|F_\xi(x^*)\|^2\right] = \frac{\gamma^2}{n}\sum_{i=1}^n \frac{\overline{L}}{L_i}\|F_i(x^*)\|^2 = \gamma^2 \sigma_{\algname{IS}*}^2.
    \end{eqnarray*}
\end{proof}

\begin{corollary}[$\overline{\mu} = 0$]\label{cor:mon_IS}
   Consider the setup from Example~\ref{ex:importance_sampling}. Let $\overline{\mu} = 0$, $\gamma_{2,\xi^k} = \alpha \gamma_{1,\xi_k}$, $\alpha = \nicefrac{1}{4}$, and $\gamma_{1,\xi^k} = \nicefrac{\gamma \overline{L}}{L_{\xi^k}}$, $\gamma \leq \nicefrac{1}{6\overline{L}}$, where $\overline{L} = \frac{1}{n}\sum_{i=1}^n L_i$. Assume that
    \begin{equation*}
        F_{\xi^k}(x^k) = \frac{1}{b}\sum\limits_{i=1}^b F_{\xi_i^k}(x),
    \end{equation*}
    where $\xi_1^k,\ldots,\xi_b^k$ are i.i.d.\ samples from the distribution on $[n]$ from Example~\ref{ex:importance_sampling}.
	Then, for all $K \ge 0$ we have
	\begin{equation*}
	    \frac{1}{K+1}\sum\limits_{k=0}^K\Exp\left[\|F(x^k)\|^2\right] \le \frac{16\|x^0 - x^*\|^2}{\gamma^2(K+1)} + \frac{12\sigma_{\algname{IS*}}^2}{b},
	\end{equation*}
	and each iteration requires $\cO(b)$ stochastic oracle calls.
\end{corollary}
\begin{proof}
   Since
    \begin{eqnarray*}
        \Exp_{\xi^k}\left[\gamma_{1,\xi^k} F_{\xi^k}(x^k)\right] &=& \frac{\gamma}{n}\sum\limits_{i=1}^n F_i(x^k) = \gamma F(x^k), 
    \end{eqnarray*}
    Corollary~\ref{cor:mon_AS} implies the needed result with
    \begin{eqnarray*}
        \sigma_{\algname{AS}}^2 &=& \Exp_\xi\left[\gamma_{\xi}^2\|F_\xi(x^*)\|^2\right] = \frac{\gamma^2}{n}\sum_{i=1}^n \frac{\overline{L}}{L_i}\|F_i(x^*)\|^2 = \gamma^2 \sigma_{\algname{IS}*}^2.
    \end{eqnarray*}
\end{proof}

\subsection{\algname{S-SEG} with Independent Sampling Without Replacement (\algname{S-SEG-ISWOR})}

\begin{theorem}\label{thm:SEG_same_sample_convergence_ISWOR}
	Consider the setup from Example~\ref{ex:iswor}. If $\gamma_{2,\xi^k} = \alpha \gamma_{1,\xi_k}$, $\alpha > 0$, and $\gamma_{1,\xi^k} = \nicefrac{\gamma |\xi|}{p_\xi 2^{n-1}n}$, $\gamma \le \nicefrac{1}{6L_{\algname{ISWOR}}}$, where $L_{\algname{ISWOR}} = \max_{S\subseteq [n]} (\nicefrac{|S|L_{S}}{p_S 2^{n-1}n})$, then \newline $g^k = F_{\xi^k}\left(x^k - \gamma_{1,\xi^k} F_{\xi^k}(x^k)\right)$ from \eqref{eq:S_SEG} satisfies Assumption~\ref{as:unified_assumption_general} with the following parameters:
	\begin{gather*}
		A = 2\alpha,\quad C = 0,\quad D_1 = 6\alpha^2\gamma^2\sigma_{\algname{ISWOR*}}^2 = \frac{6\alpha^2\gamma^2}{2^{2n-2}n^2}\sum_{S\subseteq [n]} \frac{|S|^2}{p_S}\|F_S(x^*)\|^2,\quad \rho = \frac{\alpha\gamma\overline{\mu}_{\algname{ISWOR}}}{2},\\
		G_k = \frac{\alpha\gamma^2}{2^{2n-2}n^2}\sum\limits_{S\subseteq [n]} \frac{|S|^2}{p_{S}}\left(1 - 4\frac{|\mu_{S}|\cdot |S|}{p_S 2^{n-1}n}\gamma - 2\frac{L_{S}^2|S|^2}{p_S^2 2^{2n-2}n^2}\gamma^2\right)\|F_{S}(x^k)\|^2,\\
		B = \frac{1}{2},\quad D_2 = \frac{3\alpha\gamma^2}{2}\sigma_{\algname{ISWOR*}}^2,
	\end{gather*}
	where
	\begin{equation*}
	    \overline{\mu}_{\algname{ISWOR}} = \frac{1}{2^{n-1}n}\left(\sum\limits_{S\subseteq [n] :\mu_{S} \ge 0}|S|\mu_{S} + 4\sum\limits_{S\subseteq [n] :\mu_{S} < 0}|S|\mu_{S}\right).
	\end{equation*}
	If additionally $\alpha \leq \nicefrac{1}{4}$, then for all $K\ge 0$ we have for the case $\overline{\mu}_{\algname{ISWOR*}} > 0$
	\begin{equation*}
		\Exp\left[\|x^{K+1} - x^*\|^2\right] \leq \left(1 - \frac{\alpha\gamma\overline{\mu}_{\algname{ISWOR}}}{2}\right)\Exp\left[\|x^K - x^*\|^2\right] + \frac{3\alpha}{2}\left(4\alpha + 1\right)\gamma^2\sigma_{\algname{ISWOR*}}^2,
	\end{equation*}
	\begin{equation*}
		\Exp\left[\|x^K - x^*\|^2\right] \leq \left(1 - \frac{\alpha\gamma\overline{\mu}_{\algname{ISWOR}}}{2}\right)^K\|x^0 - x^*\|^2 + \frac{3\left(4\alpha + 1\right)\gamma\sigma_{\algname{ISWOR*}}^2}{\overline{\mu}},
	\end{equation*}
	and for the case $\overline{\mu}_{\algname{ISWOR*}} = 0$
	\begin{eqnarray*}
	    \frac{1}{K+1}\sum\limits_{k=0}^K\Exp\left[\frac{1}{2^{2n-2}n^2}\sum\limits_{S\subseteq [n]} \frac{|S|^2}{p_{S}}\left(1 - 4\frac{|\mu_{S}|\cdot |S|}{p_S 2^{n-1}n}\gamma - 2\frac{L_{S}^2|S|^2}{p_S^2 2^{2n-2}n^2}\gamma^2\right)\|F_{S}(x^k)\|^2\right] &\\
	    &\hspace{-4cm}\le \frac{2\|x^0 - x^*\|^2}{\alpha\gamma^2(K+1)} + 3(4\alpha +1)\sigma_{\algname{ISWOR*}}^2
	\end{eqnarray*}
\end{theorem}
\begin{proof}
    Since $\gamma \le \nicefrac{1}{6L_{\algname{ISWOR}}}$ and $|\mu_S| \le L_S$ for all $S\subseteq [n]$, condition \eqref{eq:SS_SEG_AS_stepsizes_11} is satisfied. In Example~\ref{ex:iswor}, we show that conditions \eqref{eq:SS_SEG_AS_stepsizes_1} and \eqref{eq:SS_SEG_AS_stepsizes_2} hold as well. Therefore, Theorem~\ref{thm:SEG_same_sample_convergence_AS} implies the desired result with
    \begin{eqnarray*}
        \sigma_{\algname{AS}}^2 &=& \Exp_\xi\left[\gamma_{1,\xi}^2\|F_\xi(x^*)\|^2\right] = \frac{\gamma^2}{2^{2n-2}n^2}\sum_{S\subseteq [n]} \frac{|S|^2}{p_S}\|F_S(x^*)\|^2 = \gamma^2 \sigma_{\algname{ISWOR}*}^2,\\
        \rho &=& \frac{\alpha}{2}\Exp_{\xi^k}[\gamma_{1,\xi^k}\mu_{\xi^k}(\obf_{\{\mu_{\xi^k} \ge 0\}} + 4\cdot\obf_{\{\mu_{\xi^k} < 0\}})] = \frac{\alpha\gamma}{2^{n}n}\left(\sum\limits_{\overset{S\subseteq [n]:}{\mu_{S} \ge 0}}|S|\mu_{S} + 4\sum\limits_{\overset{S\subseteq [n]:}{\mu_{S} < 0}}|S|\mu_{S}\right)\\
        &=& \frac{\alpha\gamma\overline{\mu}_{\algname{ISWOR}}}{2},\\
        G_k &=& \alpha\Exp_{\xi^k}\left[\gamma_{1,\xi^k}^2\left(1 - 4|\mu_{\xi^k}|\gamma_{1,\xi^k} - 2 L_{\xi^k}^2\gamma_{1,\xi^k}^2\right)\|F_{\xi^k}(x^k)\|^2\right]\\
        &=& \frac{\alpha\gamma^2}{2^{2n-2}n^2}\sum\limits_{S\subseteq [n]} \frac{|S|^2}{p_{S}}\left(1 - 4\frac{|\mu_{S}|\cdot |S|}{p_S 2^{n-1}n}\gamma - 2\frac{L_{S}^2|S|^2}{p_S^2 2^{2n-2}n^2}\gamma^2\right)\|F_{S}(x^k)\|^2.
    \end{eqnarray*}
\end{proof}

\begin{corollary}[$\overline{\mu}_{\algname{ISWOR*}} > 0$]\label{cor:str_mon_ISWOR}
    Consider the setup from Example~\ref{ex:iswor}. Let $\overline{\mu}_{\algname{ISWOR*}} > 0$, $\gamma_{2,\xi^k} = \alpha \gamma_{1,\xi_k}$, $\alpha = \nicefrac{1}{4}$, and $\gamma_{1,\xi^k} = \nicefrac{\beta_k\gamma |\xi|}{p_\xi 2^{n-1}n}$, $\gamma = \nicefrac{1}{6L_{\algname{ISWOR}}}$, where $L_{\algname{ISWOR}} = \max_{S\subseteq [n]} (\nicefrac{|S|L_{S}}{p_S 2^{n-1}n})$ and $0 < \beta_k \leq 1$.
	Then, for all $K \ge 0$ and $\{\beta_k\}_{k\ge 0}$ such that
	\begin{eqnarray*}
		\text{if } K \le \frac{48L_{\algname{ISWOR}}}{\overline{\mu}_{\algname{ISWOR}}}, && \beta_k = 1,\\
		\text{if } K > \frac{48L_{\algname{ISWOR}}}{\overline{\mu}_{\algname{ISWOR}}} \text{ and } k < k_0, && \beta_k = 1,\\
		\text{if } K > \frac{48L_{\algname{ISWOR}}}{\overline{\mu}_{\algname{ISWOR}}} \text{ and } k \ge k_0, && \beta_k = \frac{96L_{\algname{ISWOR}}}{96L_{\algname{ISWOR}} + \overline{\mu}_{\algname{ISWOR}}(k - k_0)},
	\end{eqnarray*}
	for $k_0 = \left\lceil\nicefrac{K}{2} \right\rceil$ we have
	\begin{equation*}
	    \Exp\left[\|x^K - x^*\|^2\right] \le \frac{1536 L_{\algname{ISWOR}}\|x^0 - x^*\|^2}{\overline{\mu}_{\algname{ISWOR}}}\exp\left(-\frac{\overline{\mu}_{\algname{ISWOR}} K}{96 L_{\algname{ISWOR}}}\right) + \frac{1728\sigma_{\algname{ISWOR*}}^2}{\overline{\mu}_{\algname{ISWOR}}^2 K}.
	\end{equation*}
\end{corollary}
\begin{proof}
    Corollary~\ref{cor:str_mon_AS} implies the needed result with
    \begin{eqnarray*}
        \widetilde{\rho} &=& \frac{1}{8}\Exp_{\xi^k}[\gamma_{\xi^k}\mu_{\xi^k}(\obf_{\{\mu_{\xi^k} \ge 0\}} + 4\cdot\obf_{\{\mu_{\xi^k} < 0\}})] = \frac{\alpha\gamma}{2^{n+2}n}\left(\sum\limits_{\overset{S\subseteq [n]:}{\mu_{S} \ge 0}}|S|\mu_{S} + 4\sum\limits_{\overset{S\subseteq [n]:}{\mu_{S} < 0}}|S|\mu_{S}\right)\\
        &=& \frac{\overline{\mu}_{\algname{ISWOR}}}{48L_{\algname{ISWOR}}},\\
        \sigma_{\algname{AS}}^2 &=& \Exp_\xi\left[\gamma_{\xi}^2\|F_\xi(x^*)\|^2\right] = \frac{\gamma^2}{2^{2n-2}n^2}\sum_{S\subseteq [n]} \frac{|S|^2}{p_S}\|F_S(x^*)\|^2 = \gamma^2 \sigma_{\algname{ISWOR}*}^2.
    \end{eqnarray*}
\end{proof}

\newpage

\section{INDEPENDENT-SAMPLES \algname{SEG} (\algname{I-SEG}): MISSING PROOFS AND ADDITIONAL DETAILS}
\label{AppendixISEG}

In this section, we provide full proofs and missing details from Section~\ref{sec:I_SEG} on \algname{I-SEG}. Recall that our analysis of \algname{I-SEG} based on the three following assumptions:
\begin{itemize}
    \item $F(x)$ is $L$-Lipschitz: $\|F (x) - F(y)\| \le L \|x - y\|$ for all $x,y\in \R^d$ (Assumption~\ref{as:lipschitzness}),
    \item $F(x)$ is $\mu$-quasi strongly monotone: $\langle F(x), x- x^*\rangle \ge \mu\|x - x^*\|^2$ for all $x\in \R^d$ (Assumption~\ref{as:str_monotonicity}),
    \item $F_\xi(x)$ satisfies the following conditions (Assumption~\ref{as:UBV_and_quadr_growth}): $\Exp_\xi[F_\xi(x)] = F(x)$ and
    \begin{eqnarray*}
        \Exp_{\xi}\left[\|F_\xi(x) - F(x)\|^2\right] \le \delta\|x - x^*\|^2 + \sigma^2.
    \end{eqnarray*}
\end{itemize}

Moreover, we assume that
\begin{equation*}
    F_{\xi_1^k}(x^k) = \frac{1}{b}\sum\limits_{i=1}^b F_{\xi_1^k(i)}(x^k),\quad F_{\xi_2^k}(x^k) = \frac{1}{b}\sum\limits_{i=1}^b F_{\xi_2^k(i)}(x^k - \gamma_1 F_{\xi_1^k}(x^k)),
\end{equation*}
where $\xi_1^k(1),\ldots, \xi_1^k(b), \xi_2^k(1), \ldots, \xi_2^k(b)$ are i.i.d.\ samples satisfying Assumption~\ref{as:UBV_and_quadr_growth}. Due to independence of $\xi_1^k(1),\ldots, \xi_1^k(b), \xi_2^k(1), \ldots, \xi_2^k(b)$ we have
\begin{eqnarray}
    \Exp_{\xi_1^k}\left[\|F_{\xi_1^k}(x^k) - F(x^k)\|^2\right] &\le& \frac{\delta}{b}\|x^k - x^*\|^2 + \frac{\sigma^2}{b},\label{eq:batched_var_xi1}\\
    \Exp_{\xi_2^k}\left[\|F_{\xi_2^k}(x^k - \gamma_1 F_{\xi_1^k}(x^k)) - F(x^k - \gamma_1 F_{\xi_1^k}(x^k))\|^2\right] &\le& \frac{\delta}{b}\|x^k - \gamma_1 F_{\xi_1^k}(x^k) - x^*\|^2 + \frac{\sigma^2}{b}.\label{eq:batched_var_xi2}
\end{eqnarray}

It turns out that under these assumptions $g^k$ satisfies Assumption~\ref{as:unified_assumption_general}.

\begin{lemma}\label{lem:SEG_ind_sample_second_moment_bound}
	Let Assumptions~\ref{as:lipschitzness},~\ref{as:str_monotonicity}~and~\ref{as:UBV_and_quadr_growth} hold. If 
	\begin{equation}
	    \gamma_1 \le \frac{1}{\sqrt{3(L^2 + \nicefrac{2\delta}{b})}} \label{eq:I_SEG_stepsize_1}
	\end{equation}
	then $g^k = F_{\xi_2^k}\left(x^k - \gamma_1 F_{\xi_1^k}(x^k)\right)$ satisfies the following inequality
	\begin{eqnarray}
		\gamma_1^2\Exp\left[\|g^k\|^2\mid x^k\right] &\leq&  2\widehat{P}_k + \frac{9\delta\gamma_1^2}{b}\|x^k - x^*\|^2 + \frac{6\gamma_1^2\sigma^2}{b}, \label{eq:second_moment_bound_SEG_ind_sample}
	\end{eqnarray}
	where $\widehat{P}_k = \gamma_1\Exp_{\xi_1^k,\xi_2^k}\left[\langle g^k, x^k - x^* \rangle\right]$.
\end{lemma}
\begin{proof}
	Using the auxiliary iterate $\widehat{x}^{k+1} = x^k - \gamma_1 g^k$, we get
	\begin{eqnarray}
		\|\widehat{x}^{k+1} - x^*\|^2 &=& \|x^k - x^*\|^2 -2\gamma_1\langle x^k-x^*, g^k \rangle + \gamma_1^2 \|g^k\|^2\label{eq:sec_mom_ind_samples_technical_1}\\
		&=& \|x^k - x^*\|^2 -2\gamma_1\left\langle x^k - \gamma F_{\xi_1^k}(x^k) - x^*, g^k \right\rangle - 2\gamma_1^2 \langle F_{\xi_1^k}(x^k), g^k \rangle + \gamma_1^2 \|g^k\|^2.\notag
	\end{eqnarray}
	Taking the expectation $\Exp_{\xi_{1}^k, \xi_{2}^k}\left[\cdot\right] = \Exp\left[\cdot\mid x^k\right]$ conditioned on $x^k$ from the above identity, using tower property $\Exp_{\xi_1^k,\xi_2^k}[\cdot] = \Exp_{\xi_1^k}[\Exp_{\xi_2^k}[\cdot]]$, and $\mu$-quasi strong monotonicity of $F(x)$, we derive
	\begin{eqnarray*}
		\Exp_{\xi_{1}^k,\xi_2^k}\left[\|\widehat{x}^{k+1} - x^*\|^2 \right] &=& \|x^k - x^*\|^2 - 2\gamma_1 \Exp_{\xi_1^k,\xi_2^k}\left[\left\langle x^k - \gamma_1 F_{\xi_1^k}(x^k) - x^*, g^k \right\rangle \right]\\
		&&\quad - 2\gamma_1^2 \Exp_{\xi_{1}^k,\xi_2^k}\left[\langle F_{\xi_1^k}(x^k), g^k\rangle\right] + \gamma_1^2 \Exp_{\xi_{1}^k,\xi_2^k}\left[\|g^k\|^2\right]\\
		&=& \|x^k - x^*\|^2\\
		&&\quad - 2\gamma_1 \Exp_{\xi_1^k}\left[\left\langle x^k - \gamma_1 F_{\xi_1^k}(x^k) - x^*, F\left(x^k - \gamma_1 F_{\xi_1^k}(x^k)\right) \right\rangle \right]\\
		&&\quad - 2\gamma_1^2 \Exp_{\xi_{1}^k}\left[\langle F_{\xi_1^k}(x^k), g^k\rangle\right] + \gamma_1^2 \Exp_{\xi_{1}^k,\xi_2^k}\left[\|g^k\|^2\right]\\
		&\overset{\eqref{eq:str_monotonicity},\eqref{eq:inner_product_representation}}{\le}& \|x^k - x^*\|^2 - \gamma_1^2\Exp_{\xi_{1}^k,\xi_2^k}\left[\|F_{\xi_1^k}(x^k)\|^2\right] + \gamma_1^2\Exp_{\xi_{1}^k,\xi_2^k}\left[\|F_{\xi_1^k}(x^k) - g^k\|^2\right].
	\end{eqnarray*}
	To upper bound the last term we use simple inequality \eqref{eq:a+b}, and apply $L$-Lipschitzness of $F(x)$:
	\begin{eqnarray*}
		\Exp_{\xi_{1}^k,\xi_2^k}\left[\|\widehat{x}^{k+1} - x^*\|^2\right] &\overset{\eqref{eq:a+b}}{\le}& \|x^k - x^*\|^2 - \gamma_1^2\Exp_{\xi_{1}^k}\left[\|F_{\xi_1^k}(x^k)\|^2\right]\\
		&&\quad +3\gamma_1^2 \Exp_{\xi_1^k}\left[\left\|F(x^k) - F\left(x^k - \gamma_1 F_{\xi_1^k}(x^k)\right)\right\|^2\right]\\
		&&\quad + 3\gamma_1^2 \Exp_{\xi_1^k}\left[\left\| F_{\xi_1^k}(x^k) - F(x^k)\right\|^2\right]\\
		&&\quad + 3\gamma_1^2 \Exp_{\xi_1^k,\xi_2^k}\!\left[\!\left\|\! F_{\xi_2^k}\!\left(\!x^k \!-\! \gamma_1 F_{\xi_1^k}(x^k)\!\right) - F\!\left(\!x^k \!-\! \gamma_1 F_{\xi_1^k}(x^k)\!\right)\!\right\|^2\!\right]\\
		&\overset{\eqref{eq:lipschitzness}, \eqref{eq:batched_var_xi1}, \eqref{eq:batched_var_xi2}}{\le}& \|x^k - x^*\|^2 - \gamma_1^2\left(1 - 3L^2\gamma_1^2\right)\Exp_{\xi_{1}^k}\left[\|F_{\xi_1^k}(x^k)\|^2\right]\\
		&&\quad + \frac{3\gamma_1^2\delta}{b}\|x^k - x^*\|^2 + \frac{3\gamma_1^2\sigma^2}{b}\\
		&&\quad + \frac{3\gamma_1^2\delta}{b}\Exp_{\xi_1^k}\left[\|x^k - x^* - \gamma_1 F_{\xi_1^k}(x^k)\|^2\right] + \frac{3\gamma_1^2\sigma^2}{b}\\
		&\overset{\eqref{eq:a+b}}{\le}& \left(1 + \frac{9\gamma_1^2\delta}{b}\right)\|x^k - x^*\|^2 \\
		&&\quad - \gamma_1^2\left(1 - 3\gamma_1^2\left(L^2 + \frac{2\delta}{b}\right)\right)\Exp_{\xi_{1}^k}\left[\|F_{\xi_1^k}(x^k)\|^2\right]\\
		&&\quad + \frac{6\gamma_1^2\sigma^2}{b}\\
		&\overset{\eqref{eq:I_SEG_stepsize_1}}{\leq}& \left(1 + \frac{9\gamma_1^2\delta}{b}\right)\|x^k - x^*\|^2 + \frac{6\gamma_1^2\sigma^2}{b}.
	\end{eqnarray*}
	Finally, we use the above inequality together with \eqref{eq:sec_mom_ind_samples_technical_1}:
	\begin{eqnarray*}
		\|x^k - x^*\|^2 -2\widehat{P}_k + \gamma_1^2 \Exp\left[\|g^k\|^2\mid x^k\right] &\le& \left(1 + \frac{9\gamma_1^2\delta}{b}\right)\|x^k - x^*\|^2 + \frac{6\gamma_1^2\sigma^2}{b},
	\end{eqnarray*}		
	where $\widehat{P}_k = \gamma_1\Exp_{\xi_1^k,\xi_2^k}\left[\langle g^k, x^k - x^* \rangle\right]$. Rearranging the terms, we obtain \eqref{eq:second_moment_bound_SEG_ind_sample}.
\end{proof}

\begin{lemma}\label{lem:SEG_ind_sample_P_k_bound}
	Let Assumptions~\ref{as:lipschitzness},~\ref{as:str_monotonicity}~and~\ref{as:UBV_and_quadr_growth} hold. If
	\begin{equation}
		\gamma_1 \leq \min\left\{\frac{\mu b}{18\delta}, \frac{1}{4\mu + \sqrt{6(L^2 + \nicefrac{2\delta}{b})}}\right\}, \label{eq:P_k_SEG_ind_sample_gamma_condition}
	\end{equation}
	then $g^k = F_{\xi_2^k}\left(x^k - \gamma_1 F_{\xi_1^k}(x^k)\right)$ satisfies the following inequality
	\begin{eqnarray}
		\widehat{P}_{k} &\geq& \frac{\mu\gamma_1}{4}\|x^k - x^*\|^2 + \frac{\gamma_1^2}{4}\Exp_{\xi_1^k}\left[\|F_{\xi_1^k}(x^k)\|^2\right] - \frac{6\gamma_1^2 \sigma^2}{b}, \label{eq:P_k_SEG_ind_sample}
	\end{eqnarray}
	where $\widehat{P}_k = \gamma_1\Exp_{\xi_1^k, \xi_2^k}\left[\langle g^k, x^k - x^* \rangle\right]$.
\end{lemma}
\begin{proof}
	Since $\Exp_{\xi_1^k,\xi_2^k}[\cdot] = \Exp[\cdot\mid x^k]$ and $g^k = F_{\xi_2^k}\left(x^k - \gamma_1 F_{\xi_1^k}(x^k)\right)$, we have
	\begin{eqnarray*}
		-\widehat{P}_k &=& - \gamma_1\Exp_{\xi_1^k,\xi_2^k}\left[\langle g^k, x^k-x^* \rangle\right]\\
		&=& - \gamma_1\Exp_{\xi_1^k}\left[\langle \Exp_{\xi_2^k}[g^k], x^k - \gamma_1 F_{\xi_1^k}(x^k) - x^* \rangle\right] - \gamma_1^2\Exp\left[\langle g^k, F_{\xi_1^k}(x^k) \rangle\right]\\
		&\overset{\eqref{eq:inner_product_representation}}{=}& - \gamma_1\Exp_{\xi_1^k}\left[\langle F(x^k - \gamma_1 F_{\xi_1^k}(x^k)), x^k - \gamma_1 F_{\xi_1^k}(x^k) - x^* \rangle\right] \\
		&&\quad - \frac{\gamma_1^2}{2}\Exp_{\xi_1^k,\xi_2^k}\left[\|g^k\|^2\right] - \frac{\gamma_1^2}{2}\Exp_{\xi_1^k}\left[\|F_{\xi_1^k}(x^k)\|^2\right] + \frac{\gamma_1^2}{2}\Exp_{\xi_1^k,\xi_2^k}\left[\|g^k - F_{\xi_1^k}(x^k)\|^2\right]\\
		&\overset{\eqref{eq:str_monotonicity},\eqref{eq:a+b}}{\leq}& -\mu\gamma_1\Exp_{\xi_1^k,\xi_2^k}\left[\|x^k - x^* - \gamma_1 F_{\xi_1^k}(x^k)\|^2\right] - \frac{\gamma_1^2}{2}\Exp_{\xi_1^k}\left[\|F_{\xi_1^k}(x^k)\|^2\right]\\
		&&\quad + \frac{3\gamma_1^2}{2}\Exp_{\xi_1^k}\left[\left\|F(x^k) - F\left(x^k - \gamma_1 F_{\xi_1^k}(x^k)\right)\right\|^2\right] \\
		&&\quad + \frac{3\gamma_1^2}{2}\Exp_{\xi_1^k}\left[\left\|F_{\xi_1^k}(x^k) - F(x^k)\right\|^2\right] \\
		&&\quad + \frac{3\gamma_1^2}{2}\Exp_{\xi_1^k,\xi_2^k}\left[\left\|F_{\xi_2^k}\left(x^k - \gamma_1 F_{\xi_1^k}(x^k)\right) - F\left(x^k - \gamma_1 F_{\xi_1^k}(x^k)\right)\right\|^2\right] \\
		&\overset{\eqref{eq:a+b_lower},\eqref{eq:lipschitzness}, \eqref{eq:UBV_main}}{\le}& -\frac{\mu\gamma_1}{2}\|x^k - x^*\|^2 - \frac{\gamma_1^2}{2} (1 - 2\gamma_1\mu - 3\gamma_1^2 L^2)\Exp_{\xi_1^k}\left[\|F_{\xi_1^k}(x^k)\|^2\right] \\
		&&\quad + \frac{3\gamma_1^2\delta}{2b}\|x^k - x^*\|^2 + \frac{3\gamma_1^2\sigma^2}{2b}\\
		&&\quad + \frac{3\gamma_1^2\delta}{2b}\Exp_{\xi_1^k}\left[\|x^k - x^* - \gamma_1 F_{\xi_1^k}(x^k)\|^2\right] + \frac{3\gamma_1^2\sigma^2}{2b}\\
		&\overset{\eqref{eq:a+b}}{\leq}& -\frac{\mu\gamma_1}{2}\left(1 - \frac{9\gamma_1\delta}{\mu b}\right)\|x^k - x^*\|^2 \\
		&&\quad - \frac{\gamma_1^2}{2} \left(1 - 2\gamma_1\mu - 3\gamma_1^2 \left(L^2 + \frac{2\delta}{b}\right)\right)\Exp_{\xi_1^k}\left[\|F_{\xi_1^k}(x^k)\|^2\right] + \frac{6\gamma_1^2\sigma^2}{2b}\\
		&\overset{\eqref{eq:P_k_SEG_ind_sample_gamma_condition}}{\leq}& -\frac{\mu\gamma_1}{4}\|x^k - x^*\|^2 - \frac{\gamma_1^2}{4}\Exp_{\xi_1^k}\left[\|F_{\xi_1^k}(x^k)\|^2\right] + \frac{6\gamma_1^2 \sigma^2}{b}
	\end{eqnarray*}
	that concludes the proof\footnote{When $\delta = 0$, i.e., when we are in the classical setup of uniformly bounded variance, numerical constants in our proof can be tightened. Indeed, in the last step, we can get $-\frac{\mu\gamma_1}{2}\|x^k - x^*\|^2 - \frac{\gamma_1^2}{4}\Exp_{\xi_1^k}\left[\|F_{\xi_1^k}(x^k)\|^2\right] + \frac{6\gamma_1^2 \sigma^2}{b}$. Moreover, if we are interested in the case when $\mu > 0$, then assuming that $\gamma_1 \le \frac{1}{2\mu + \sqrt{3}L}$, can get $-\frac{\mu\gamma_1}{2}\|x^k - x^*\|^2 + \frac{6\gamma_1^2 \sigma^2}{b}$.}.
\end{proof}

Combining Lemmas~\ref{lem:SEG_ind_sample_second_moment_bound}~and~\ref{lem:SEG_ind_sample_P_k_bound} and applying Theorem~\ref{thm:main_theorem_general_main}, we get the following result.

\begin{theorem}[Theorem~\ref{thm:I_SEG_convergence_main}]\label{thm:I_SEG_convergence}
	Let Assumptions~\ref{as:lipschitzness},~\ref{as:str_monotonicity},~and~\ref{as:UBV_and_quadr_growth} hold. If $\gamma_{2} = \alpha \gamma_{1}$, $\alpha > 0$, and $\gamma_{1} = \gamma$, where\footnote{When $\mu = \delta = 0$, the first term can be ignored.}
	\begin{equation*}
	    \gamma \le \min\left\{\frac{\mu b}{18\delta}, \frac{1}{4\mu + \sqrt{6(L^2 + \nicefrac{2\delta}{b})}}\right\}
	\end{equation*}
	then $g^k = F_{\xi_2^k}\left(x^k - \gamma_{1} F_{\xi_1^k}(x^k)\right)$ from \eqref{eq:I_SEG} satisfies Assumption~\ref{as:unified_assumption_general} with the following parameters:
	\begin{gather*}
		A = 2\alpha,\quad C = \frac{9\delta\alpha^2\gamma^2}{b},\quad D_1 = \frac{6\alpha^2\gamma^2\sigma^2}{b},\quad \rho = \frac{\alpha\gamma\mu}{4},\\
		G_k = \Exp_{\xi_1^k}\left[\|F_{\xi_1^k}(x^k)\|^2\right],\quad B = \frac{\alpha\gamma^2}{4},\quad D_2 = \frac{6\alpha\gamma^2\sigma^2}{b}.
	\end{gather*}
	If additionally $\alpha \leq \nicefrac{1}{4}$, then for all $K\ge 0$ we have for the case $\mu > 0$
	\begin{equation*}
		\Exp\left[\|x^{K+1} - x^*\|^2\right] \leq \left(1 - \frac{\alpha\gamma\mu}{8}\right)\Exp\left[\|x^K - x^*\|^2\right] + 6\alpha\left(\alpha + 1\right)\gamma^2\frac{\sigma^2}{b},
	\end{equation*}
	\begin{equation*}
		\Exp\left[\|x^K - x^*\|^2\right] \leq \left(1 - \frac{\alpha\gamma\mu}{8}\right)^K\|x^0 - x^*\|^2 + \frac{48\left(\alpha + 1\right)\gamma\sigma^2}{\mu b},
	\end{equation*}
	and for the case $\mu = 0$ and $\delta = 0$
	\begin{equation*}
	    \frac{1}{K+1}\sum\limits_{k=0}^K\Exp\left[\|F_{\xi_1^k}(x^k)\|^2\right] \le \frac{4\|x^0 - x^*\|^2}{\alpha\gamma^2(K+1)} + \frac{24(\alpha+1)\sigma^2}{b}.
	\end{equation*}
\end{theorem}
\begin{proof}
    \algname{I-SEG} fits the unified update rule \eqref{eq:general_method} with $\gamma_{\xi^k} = \gamma_{2}$ and $g^k = F_{\xi_2^k}\left(x^k - \gamma_{1} F_{\xi_1^k}(x^k)\right)$. Moreover, Lemmas~\ref{lem:SEG_ind_sample_second_moment_bound}~and~\ref{lem:SEG_ind_sample_P_k_bound} imply
    \begin{eqnarray}
		\gamma_1^2\Exp\left[\|g^k\|^2\mid x^k\right] &\leq&  2\widehat{P}_k + \frac{9\delta\gamma_1^2}{b}\|x^k - x^*\|^2 + \frac{6\gamma_1^2\sigma^2}{b},\label{eq:cdscdscdcsdc1}\\
		\widehat{P}_{k} &\geq& \frac{\mu\gamma_1}{4}\|x^k - x^*\|^2 + \frac{\gamma_1^2}{4}\Exp_{\xi_1^k}\left[\|F_{\xi_1^k}(x^k)\|^2\right] - \frac{6\gamma_1^2 \sigma^2}{b},\label{eq:cdscdscdcsdc2}
	\end{eqnarray}
	where $\widehat{P}_k = \gamma_{1}\Exp_{\xi_1^k,\xi_2^k}\left[\langle g^k, x^k - x^* \rangle\right]$.
    Since $\gamma_{\xi^k} = \gamma_{2} = \alpha \gamma_{1}$, we multiply \eqref{eq:cdscdscdcsdc1} by $\alpha^2$ and \eqref{eq:cdscdscdcsdc2} by $\alpha$ and get that Assumption~\ref{as:unified_assumption_general} holds with the parameters given in the statement of the theorem. Applying Theorem~\ref{thm:main_theorem_general_main} we get the result.
\end{proof}

\begin{corollary}[$\mu > 0$; Corollary~\ref{cor:str_mon_ISEG_main}]\label{cor:str_mon_I_SEG}
    Let Assumptions~\ref{as:lipschitzness},~\ref{as:str_monotonicity},~and~\ref{as:UBV_and_quadr_growth} hold. Let $\mu > 0$, $\gamma_{2,k} = \alpha \gamma_{1,k}$, $\alpha = \nicefrac{1}{4}$, and $\gamma_{1,k} = \beta_k\gamma$, where
	\begin{equation*}
	    \gamma = \min\left\{\frac{\mu b}{18\delta}, \frac{1}{4\mu + \sqrt{6(L^2 + \nicefrac{2\delta}{b})}}\right\}
	\end{equation*}
	and $0 < \beta_k \leq 1$.
	Then, for all $K \ge 0$ and $\{\beta_k\}_{k\ge 0}$ such that
	\begin{eqnarray*}
		\text{if } K \le \frac{32}{\gamma\mu}, && \beta_k = 1,\\
		\text{if } K > \frac{32}{\gamma\mu} \text{ and } k < k_0, && \beta_k = 1,\\
		\text{if } K > \frac{32}{\gamma\mu} \text{ and } k \ge k_0, && \beta_k = \frac{64}{64 + \gamma\mu(k - k_0)},
	\end{eqnarray*}
	for $k_0 = \left\lceil\nicefrac{K}{2} \right\rceil$ we have
	\begin{eqnarray*}
	    \Exp\left[\|x^K - x^*\|^2\right] &\le& \frac{1024 \|x^0 - x^*\|^2}{\gamma\mu}\exp\left(-\frac{\gamma\mu K}{64}\right) + \frac{69120\sigma^2}{\mu^2 bK}\\
	    &=& \cO\!\left(\!\max\left\{\frac{\delta}{\mu^2 b}, \frac{L +\sqrt{\nicefrac{\delta}{b}}}{\mu}\right\}\!\|x^0\!-\! x^*\|^2 \exp\!\left(\!-\frac{K}{\max\left\{\frac{\delta}{\mu^2 b}, \frac{L +\sqrt{\nicefrac{\delta}{b}}}{\mu}\right\}}\!\right)\!+\! \frac{\sigma^2}{\mu^2 bK}\!\right). 
	\end{eqnarray*}
\end{corollary}
\begin{proof}
    In Theorem~\ref{thm:I_SEG_convergence}, we establish the following recurrence:
    \begin{eqnarray*}
        \Exp\left[\|x^{k+1} - x^*\|^2\right] &\leq& \left(1 - \beta_k\frac{\alpha\gamma\mu}{8}\right)\Exp\left[\|x^k - x^*\|^2\right] + 6\alpha\left(\alpha + 1\right)\beta_k^2\gamma^2\frac{\sigma^2}{b}\\
        &\overset{\alpha = \nicefrac{1}{4}}{=}&  \left(1 - \beta_k\frac{\gamma\mu}{32}\right)\Exp\left[\|x^k - x^*\|^2\right] + \beta_k^2\frac{15\sigma^2}{8b}.
    \end{eqnarray*}
    Applying Lemma~\ref{lem:stich_lemma_for_str_cvx_conv} for $r_k = \Exp\left[\|x^k - x^*\|^2\right]$, $\gamma_k = \beta_k$, $a = \nicefrac{\gamma\mu}{32}$, $c = \nicefrac{15\sigma^2}{8b}$, $h = 1$, we get the result.
\end{proof}

\begin{corollary}[$\mu = 0$]\label{cor:mon_I_SEG}
    Let Assumptions~\ref{as:lipschitzness},~\ref{as:str_monotonicity},~and~\ref{as:UBV_and_quadr_growth} hold. Let $\mu = 0$, $\delta = 0$, $\gamma_{2} = \alpha \gamma_{1}$, $\alpha = \nicefrac{1}{4}$, and $\gamma_{1} = \gamma = \nicefrac{1}{\sqrt{6}L}$. Then, for all $K \ge 0$ we have
	\begin{equation*}
	    \frac{1}{K+1}\sum\limits_{k=0}^K\Exp\left[\|F(x^k)\|^2\right] \le \frac{16\sqrt{6}L\|x^0 - x^*\|^2}{K+1} + \frac{30\sigma^2}{b},
	\end{equation*}
	and each iteration requires $\cO(b)$ stochastic oracle calls.
\end{corollary}
\begin{proof}
    Given the result of Theorem~\ref{thm:I_SEG_convergence}, it remains to plug in $\alpha = \nicefrac{1}{4}$.
\end{proof}

\subsection{On the Assumptions in the Analysis of \algname{S-SEG} and \algname{I-SEG}}\label{app:on_the_assumptions}

In this subsection, we provide clarifications on why we use different assumptions to analyze \algname{S-SEG} and \algname{I-SEG}. In particular, our analysis of \algname{S-SEG} requires Lipschitzness and quasi-strong monotonicity of $F(x,\xi)$ for all $\xi$ (Assumptions~\ref{as:F_xi_lip}, \ref{as:F_xi_str_monotonicity}) and no assumptions on the variance of $F(x,\xi)$, while for \algname{I-SEG} we use bounded variance assumption (Assumption~\ref{as:UBV_and_quadr_growth}).

First of all, it is known that deterministic \algname{EG} can be viewed as an approximation of the Proximal Point method \citep{martinet1970regularisation, rockafellar1976monotone} when \textbf{$F$ is $L$-Lipschitz}, e.g., see Theorem 1 from \citep{mishchenko2020revisiting}. In some sense, Lipschitzness of $F$ is a crucial property for the convergence of \algname{EG}. One iteration of \algname{S-SEG} can be seen as a step of deterministic \algname{EG} for the stochastic operator $F(x,\xi)$. Therefore, it is natural that Lipschitzness of $F(x,\xi)$ is important for the analysis of \algname{S-SEG}. On the other side, \algname{I-SEG} uses different samples for extrapolation and update steps. Therefore, Lipschitzness of individual $F(x,\xi)$ does not help here and we need to use something like Assumption~\ref{as:UBV_and_quadr_growth} to handle the stochasticity of the method.

\end{document}